\documentclass[11pt]{amsart}
\usepackage{amssymb,mathrsfs,graphicx,subfigure}
\usepackage{amsfonts,amssymb,amsmath,mathrsfs,enumerate}
\usepackage{graphicx,colortbl}
\usepackage{hyperref}
\usepackage{bbm}
\usepackage{float}
\usepackage{mathtools}
\allowdisplaybreaks

\title[Phase concentration for the Kuramoto-Sakaguchi equation]{Emergence of phase concentration for the Kuramoto-Sakaguchi equation}

\author[Seung-Yeal Ha]{Seung-Yeal Ha}
\address[Seung-Yeal Ha]{\newline Department of Mathematical Sciences and Research Institute of Mathematics, \newline
Seoul National University, Seoul 08826, Korea \newline
Korea Institute for Advanced Study, Hoegiro 85, Seoul, 02455, Korea}
\email{syha@snu.ac.kr}

\author[Young-Heon Kim]{Young-Heon Kim}
\address[Young-Heon Kim]{\newline Department of Mathematics, \newline University of British Columbia, Vancouver, V6T 1Z2 Canada}
\email{yhkim@math.ubc.ca}

\author[Javier Morales]{Javier Morales}
\address[Javier Morales]{\newline Department of Mathematics, \newline University of Texas at Austin, Austin, United States}
\email{jmorales@math.utexas.edu}

\author[Jinyeong Park]{Jinyeong Park}
\address[Jinyeong Park]{\newline Department of Mathematical Sciences, \newline Seoul National University, Seoul 08826, Korea}
\email{pjy40@snu.ac.kr}

\newtheorem{theorem}{Theorem}[section]
\newtheorem{lemma}{Lemma}[section]
\newtheorem{corollary}{Corollary}[section]
\newtheorem{proposition}{Proposition}[section]
\newtheorem{remark}{Remark}[section]

\newtheorem{definition}{Definition}[section]

\date{\today}
\def\BE{\begin{equation}}
\def\EE{\end{equation}}

\renewcommand{\ell}{{l}}

\def\charf {\mbox{{\text 1}\kern-.30em {\text l}}}

\begin{document}

\date{\today}


%

\begin{abstract}
We study the asymptotic phase concentration phenomena for the Kuramoto-Sakaguchi(K-S) equation in a large coupling strength regime. For this, we analyze the detailed dynamics of the order parameters such as the amplitude and the average phase. For the infinite ensemble of oscillators with the identical natural frequency, we show that the total mass distribution concentrates on the  average phase asymptotically, whereas the mass around the antipodal point of the average phase decays to zero exponentially fast in any positive coupling strength regime. Thus, generic initial kinetic densities evolve toward the Dirac measure concentrated on the average phase. In contrast, for the infinite ensemble with distributed natural frequencies, we find a certain time-dependent interval whose length can be explicitly quantified in terms of the coupling strength. Provided that the coupling strength is sufficiently large, the mass on such an interval is eventually non-decreasing over the time. We also show that the amplitude order parameter has a positive lower bound that depends on the size of support of the distribution function for the natural frequencies and the coupling strength. The proposed asymptotic lower bound on the order parameter tends to unity, as the coupling strength increases to infinity. This is reminiscent of practical synchronization for the Kuramoto model, in which the diameter for the phase configuration is inversely proportional to the coupling strength. Our results for the K-S equation generalize the results in \cite{H-K-R} on the emergence of phase-locked states for the Kuramoto model in a large coupling strength regime.
\end{abstract}

\maketitle \centerline{\date}


\section{Introduction} \label{sec:1}
Collective phenomena such as aggregation, flocking, and synchronization, etc., are ubiquitous in biological, chemical, and mechanical systems in nature, e.g., the flashing of fireflies, chorusing of crickets, synchronous firing
of cardiac pacemakers, and metabolic synchrony in yeast cell
suspensions (see for instance \cite{A-B, Bu}). After Huygens' observation on the anti-synchronized motion of two pendulum clocks hanging on the same bar, the synchronization of oscillators were reported in literature from time to time. However, the first rigorous and systematic studies on synchronization 
were pioneered by Winfree \cite{Wi2} and Kuramoto \cite{Ku2} in several decades ago. They introduced phase coupled models for the ensemble of weakly coupled oscillators, and showed that collective synchronization in the ensemble of oscillators can emerge from disordered ensemble via the competing mechanism between intrinsic randomness and sinusoidal nonlinear couplings (see \cite{A-B, D-B2, P-R-K}, for details). In this paper, we are interested in the large-time dynamics of a large ensemble of Kuramoto oscillators. In particular, we assume that the number of Kuramoto oscillators is sufficiently large so that a one-oscillator probability distribution function can describe effectively the dynamics of a large phase-coupled system, i.e., our main concern lies in the mesoscopic description of the ensemble of Kuramoto oscillators. In fact, this kinetic description has been used in physics literature \cite{A-B} to analyze the phase
transition from an incoherent state to a partially synchronized state, as the coupling strength is varied from zero to a large value. \newline

Let $f = f(\theta, \omega,t )$ be the one-oscillator probability density function of the Kuramoto ensemble in phase $\theta \in \mathbb T := \mathbb R/(2\pi \mathbb Z),$ with a natural frequency $\omega$ at time $t$, as in \cite{La}. Suppose that $g=g(\omega)$ is a nonnegative and compactly supported probability
density function for natural frequencies with zero first frequency moment $(\int_\mathbb R \omega g(\omega)d\omega = 0)$. Then,  the dynamics of the kinetic density $f$ is governed by the Kuramoto-Sakaguchi (K-S) equation:
\begin{equation} \label{K-Sa}
\begin{dcases}
\partial_t f + \partial_{\theta} (v[f] f) = 0, \qquad (\theta, \omega) \in \mathbb T \times \mathbb R,~~t > 0, \\
 v[f](\theta, \omega, t) = \omega - K \int_{\mathbb T} \sin(\theta-\theta_*) \rho(\theta_*, t)  ~ d\theta_*, \quad \rho(\theta, t) := \int_{\mathbb R} f(\theta, \omega, t) ~ d\omega,
\end{dcases}
\end{equation}
subject to the initial datum:
\begin{equation} \label{ini}
 f(\theta, \omega, 0) = f_0(\theta, \omega), \quad  \int_{\mathbb T} f_0 ~ d\theta = g(\omega), 
\end{equation}
where $K$ is the positive coupling strength measuring the degree of mean-field interactions between oscillators. The K-S equation \eqref{K-Sa} has been rigorously derived from the Kuramoto model in mean-field limit $(N \to \infty)$, using the method of  particle-in-cell employing empirical measures as an approximation \cite{La}. Several global existence theories have been proposed for \eqref{K-Sa}-\eqref{ini} in different frameworks, e.g., BV-entropic weak solutions \cite{A-H-P}, measure-valued solutions, and classical solutions \cite{C-C-H-K-K, Ch1, La}. Recently, motivated by the success of nonlinear Landau damping in plasma physics, there have been several interesting works \cite{B-C-M, Ch, F-G-G, phase} on the Kuramoto conjecture and Landau damping in relation to stability and instability of incoherent solution in sub-and super-critical regimes. We also refer to  \cite{Fokker, H-X1, H-X2} for the corresponding issues for the Kuramoto-Sakaguchi-Fokker-Planck equation which is a stochastic version of the K-S equation. \newline

The purpose of this paper is to investigate the emergence of phase concentration for the K-S equation via the time-asymptotic approach. The time-asymptotic approach is to show existence of the steady states with some desired properties, as well as their stability to a given time-dependent problem. This approach has been very successful in the realms of hyperbolic conservation laws and kinetic theory, to analyze the large-time behavior of viscous conservation laws and positivity of Boltzmann shocks \cite{L-Y, L-Z}. In spirit, this is close to the mean curvature flow in differential geometry, in which manifolds with constant mean curvature emerges as an asymptotic manifold from a rough manifold via the mean curvature flow. In this time-asymptotic approach, we are able to obtain quantitive estimates on the detailed relaxation dynamics from the initial states not in the resulting attractors. As byproduct, stability and structure of the resulting attractors follows; for a finite-dimensional analogue, we refer to \cite{C-C-H-K-K} where existence, stability and structure of the phase-locked states are presented via the time-asymptotic approach based on the Kuramoto model. For a survey on related issues arising from the classical and quantum synchronization, we refer to the recent review papers \cite{D-B2, H-K-P-Z}.  \newline

The main results of this paper are three-fold. First, we consider the infinite ensemble of identical oscillators in which the density function $g = g(\omega)$ for the natural frequency is given by the Dirac measure concentrated on the average natural frequency. In this case, for any positive coupling strength,  we show that generic ${\mathcal C}^1$- initial datum with a positive order parameter tends to the Dirac measure concentrated on the asymptotic average phase, whereas mass near the antipodal phase of the average phase decays to zero exponentially fast (see Theorem \ref{T3.1}). The latter assertion contrasts the difference between the infinite-dimensional case (the K-S equation) and finite-dimensional case (the Kuramoto model). For the Kuramoto model, bi-polar configurations (say, one oscillator lies on the south pole, and the rest of ensemble lies on the north pole) is possible, although it is unstable. The second and third results deal with mass concentration phenomenon for the distributed natural frequencies. In our second result (Theorem \ref{T3.2}), we construct a time dependent interval $L(t)$ centered at the time-dependent average phase and with constant width, such that the mass over $L(t)$ is nondecreasing and for each fixed natural frequency $w$ in the support of $g$, the integral $\int_{L(t)} |f(\theta, \omega, t)|^2 d\theta$ tends to infinity exponentially fast for large coupling strengths depending on the size of the support of $g$. This is obtained for a well arranged initial datum. Such condition is removed in our third result, where we present a nontrivial lower bound for the asymptotic amplitude order parameter depending only on the size of the support of $g$ and the coupling strength. We also show that there exists a time dependent interval that contains all the mass asymptotically.  The size of the interval is characterized by the coupling strength, maximum of natural frequencies, and the asymptotic amplitude order parameter (see Theorem \ref{T3.3}). Moreover, by choosing the coupling strength large enough, this size can be made arbitrary small, and the amplitude order parameter arbitrary close to $1$. \newline

The rest of this paper is organized as follows. In Section \ref{sec:2}, we
briefly review several concepts of synchronization for the Kuramoto model, the order parameters (amplitude and phase), and gradient flow formulations for the Kuramoto model and the K-S equation. We also recall some relevant previous results for the Kuramoto model. In Section \ref{sec:3}, we discuss our main results for the K-S equation on the emergence of attractors. 
In Section \ref{sec:4}, we present an emergent dynamics of the K-S equation for identical oscillators. In particular, we present dynamics of the amplitude order parameter and using it, we give the proof of Theorem \ref{T3.1}.  In Section \ref{sec:5}, we 
study the dynamics of local order parameters for the sub-ensemble of identical oscillators with the same natural frequency, and using the detailed dynamics of local order parameters, we provide the proof of Theorem \ref{T3.2}. In Section \ref{sec:6}, we provide a nontrivial lower bound for the asymptotic amplitude order parameter in terms of the maximum of the natural frequency, and the coupling strength. This lower bound estimate for the asymptotic order parameter yields a certain practical synchronization that has been introduced for the finite-dimensional Kuramoto model in \cite{H-N-P}. Finally, Section 7 is devoted to the summary of our main results and future directions. In Appendix \ref{App-A}, we provide a short presentation of Otto's calculus which inspired us for some of our proofs. For the readability of the paper, we postpone lengthy proofs of several lemmata and propositions to Appendix \ref{App-B} - Appendix \ref{App-E}. In Appendix \ref{App-F}, we discuss several estimates on the Kuramoto vector field. 

\bigskip

\noindent Notation: For vectors $p, q$ in $\mathbb R^2$, we denote an inner product of $p$ and $q$ by $p \cdot q$, whereas for two complex numbers $z_1, z_2 \in \mathbb C$, we set their inner product by $\langle z_1, z_2 \rangle = z_1 {\bar z}_2 $.

%
%
%
\section{Preliminaries} \label{sec:2}
\setcounter{equation}{0}
In this section, we briefly review two synchronization models, the Kuramoto model and its kinetic counterpart, the Kuramoto-Sakaguchi equation. For these two models, we introduce real-valued order parameters and gradient flow formulations.

\subsection{The Kuramoto model}
Consider a complete network consisting of $N$-nodes with edges connecting all  pair of nodes, and assume that at each node, a Landau-Stuart oscillator is located. We  set $z_i  \in \mathbb C^1$ to be the state of the $i$-th Landau-Stuart oscillator. Then, $z_i$ is governed by the following first-order system of ODEs:
\begin{equation} \label{C-L-S}
\frac{dz_i}{dt} = ( 1 - |z_i|^2 + {\mathrm i} \omega_i) z_i + \frac{K}{N} \sum_{j=1}^{N} (z_j - z_i), \quad j = 1, \cdots, N,
\end{equation}
where $K$ is the uniform coupling strength between oscillators, and $\omega_i$ is the quenched random natural frequency of the $i$-th Stuart-Landau oscillator extracted from a given distribution function $g = g(\omega)$, $\omega \in \mathbb R$:
\[ 
\int_{\mathbb R} g(\omega) ~ d\omega = 1, \quad \int_{\mathbb R} \omega g(\omega) ~ d\omega = 0, \quad \mbox{supp} ~ g(\cdot) 
\subset \! \subset \mathbb R, \quad g(\omega) \geq 0.
\]
The state $z_i = z_i(t)$ governed by the system \eqref{C-L-S} approaches a certain limit-cycle (a circle with radius determined by the coupling strength) asymptotically for a suitable range of $K$ (see \cite{Ku1}). Hence, in the sequel, we are mainly interested in the dynamics of the limit-cycle oscillators so that the amplitude variations can be ignored from the dynamics, and we focus our attention on the phase dynamics. This explains the meaning of ``{\it weakly coupled} oscillator". To see the dynamics of the  phase, we set 
\begin{equation} \label{Ans}
z_i(t) := e^{{\mathrm i} \theta_i(t)}, \quad t \geq 0, \quad 1 \leq i \leq N, 
\end{equation}
and substitute \eqref{Ans} into \eqref{C-L-S}, and compare the imaginary part of the resulting relation to derive the Kuramoto model \cite{Ku1, Ku2}:
\begin{equation} \label{KM}
{\dot \theta}_i = \omega_i + \frac{K}{N} \sum_{j=1}^{N} \sin (\theta_j - \theta_i), \quad i =1, \cdots, N.
\end{equation}
Note that the first term in the right-hand-side of \eqref{KM}, represents the intrinsic randomness of the system, whereas the second term describes the nonlinearity of the  attractive coupling. It is easy to see that the total phase $\sum_{i=1}^{N} \theta_i$ satisfies a balanced law:
\[ \frac{d}{dt} \sum_{i=1}^{N} \theta_i = \sum_{i=1}^{N} \omega_i, \quad t > 0. \]
Thus, when the total sum of natural frequencies is not zero, then system \eqref{KM} cannot have equilibria $\Theta_e = (\theta_{1e}, \cdots, \theta_{Ne})$:
\[ {\dot \theta}_{ie} = 0, \quad 1 \leq i \leq N. \]
However, we may still expect existence of relative equilibria, which are the equilibria of \eqref{KM} in a rotating coordinate frame with the angular velocity $\displaystyle \omega_c := \frac{1}{N} \sum_{i=1}^{N} \omega_i$. The relative equilibrium  for \eqref{KM} is called the phase-locked state. More precisely, we present its formal definition as follows.
\begin{definition} \label{D2.1} 
\emph{\cite{C-H-J-K, H-N-P}}
Let $\Theta = (\theta_1, \cdots, \theta_N)$ be a solution to \eqref{KM}. 
\begin{enumerate} 
\item
$\Theta = (\theta_1, \cdots, \theta_N)$ is a phase-locked state if the transversal phase differences are constant along the Kuramoto flow \eqref{KM}:
\[  |\theta_i(t) - \theta_j(t)| = |\theta_i(0) - \theta_j(0)|, \quad \forall~t \geq 0,~~1 \leq i, j \leq N .    \]
\item
The Kuramoto model \eqref{KM} exhibits  ``{\it complete (frequency) synchronization}" asymptotically if the transversal frequencies differences approach zero asymptotically:
\[  \lim_{t \to \infty} \max_{1 \leq i, j \leq N} |{\dot \theta}_i(t) - {\dot \theta}_j(t)| = 0.    \]
\item
The Kuramoto model \eqref{KM} exhibits ``{\it complete phase synchronization}" asymptotically if the transversal phase differences approach zero asymptotically:
\[  \lim_{t \to \infty} \max_{1 \leq i, j \leq N} |\theta_i(t) - \theta_j(t)| = 0.     \]
\item
The Kuramoto model \eqref{KM} exhibits ``{\it practical (phase) synchronization}" asymptotically if the transversal phase differences satisfy
\[  \lim_{K \to \infty} \limsup_{t \to \infty} \max_{1 \leq i, j \leq N}  |\theta_i(t) - \theta_j(t)| = 0.    \]
\end{enumerate}
\end{definition}
\begin{remark} \label{R2.1} 1. If complete synchronization occurs asymptotically, solutions tend to phase-locked states asymptotically. We also note that for non-identical oscillators, complete phase synchronization is not possible even asymptotically. For details on the phase-locked states, we refer to \cite{C-H-J-K, H-K-R0}.  \newline

\noindent 2. When the average natural frequency $\omega_c$ is zero, the equilibrium solution $\Theta$ to \eqref{KM} which is a solution to the following system of transcendental equations:
\[  \omega_c := \sum_{i=1}^{N} \omega_i = 0, \quad  \omega_i + \frac{K}{N} \sum_{j=1}^{N} \sin(\theta_j - \theta_i) = 0, \quad i = 1, \cdots, N, \]
is a phase-locked solution to \eqref{KM} as well. \newline

\noindent 3. For a brief review on the classical and quantum Kuramoto type models, we refer to a recent survey paper \cite{H-K-P-Z}.
\end{remark}

\subsubsection{Order parameters} \label{2.1.1} In this part, we briefly review real-valued order parameters for the phase configuration $\Theta = (\theta_1, \cdots, \theta_N)$ and their dynamics following the presentation given in \cite{H-K-P}. Consider the average position (centroid) of $N$ limit-cycle oscillators $z_i = e^{{\mathrm i} \theta_i}$: for $t \geq 0$,
\begin{equation} \label{order-p}
r(t)e^{{\mathrm i} \phi(t)} := \frac{1}{N} \sum_{j=1}^{N} e^{{\mathrm i} \theta_j(t)}. 
\end{equation}
Here we call $r$ and $\phi$ the amplitude and the average phase order parameters for the $N$ limit-cycle system, respectively.  Since the right hand side of \eqref{order-p} is a convex combination of $N$-points on the unit circle, the amplitude $r(t)$ lies on the interval $[0, 1]$ and the cases $r = 0$ and $1$ correspond to the splay state and the completely phase synchronized state, respectively. Hence, we can regard $r$ and $\phi$ as quantities measuring the degree of overall synchronization and the average of phases, respectively. Note that the average phase $\phi$ is well-defined when $r>0$.

We divide \eqref{order-p} by {{ $e^{\mathrm{i} \theta_k}$}} to obtain
\[ {{re^{\mathrm{i} (\phi-\theta_k)} = \frac{1}{N} \sum_{j=1}^{N} e^{\mathrm{i}(\theta_j
- \theta_k)},}} \] and compare the real and imaginary parts of the above
relation to find
\begin{equation} \label{B-1}
r \cos (\phi - \theta_k) = \frac{1}{N}
\sum_{j=1}^{N} \cos (\theta_j - \theta_k), \quad
r \sin (\phi - \theta_k) = \frac{1}{N} \sum_{j=1}^{N}  \sin
(\theta_j - \theta_k).
\end{equation}
Similarly, we divide the relation \eqref{order-p} by 
$e^{{\mathrm i} \phi(t)}$ and compare real and imaginary parts to see the following relations:
\begin{equation*} 
r = \frac{1}{N} \sum_{j=1}^{N} \cos(\theta_j - \phi), \quad 0 =  \frac{1}{N} \sum_{j=1}^{N} \sin(\theta_j - \phi).
\end{equation*}
By comparing the second relation in \eqref{B-1} and the coupling terms in \eqref{KM}, it is easy to see that the Kuramoto model \eqref{KM} can be rewritten in a mean-field form:
\begin{equation} \label{KM-mf}
{\dot \theta}_i = \omega_i - K r \sin(\theta_i - \phi), \quad t > 0.
\end{equation}
The equation \eqref{KM-mf} looks decoupled, but it is coupled, because the order parameters $r$ and $\phi$ are functions of other $\theta_j$'s. 

We next study the dynamics of the order parameters $r$ and $\phi$. For this, we differentiate the  equation \eqref{order-p} with respect to $t$ to see
\[ {\dot r} e^{{\mathrm i} \phi} + {\mathrm i} r  e^{{\mathrm i} \phi} {\dot \phi} = \frac{{\mathrm i}}{N}
\sum_{j=1}^{N} e^{{\mathrm i} \theta_j} {\dot \theta}_j. \] We divide the
resulting equation by $e^{{\mathrm i}\phi}$ to find
\begin{equation} \label{B-3}
{\dot r} + {\mathrm i} r {\dot \phi} =  -\frac{1}{N} \sum_{j=1}^{N} \sin
(\theta_j - \phi) {\dot \theta}_j + \frac{{\mathrm i}}{N}\sum_{j=1}^{N}
\cos (\theta_j- \phi) {\dot \theta}_j.
\end{equation}
We now compare the real and imaginary parts of (\ref{B-3}) to obtain
\begin{equation} \label{B-4}
{\dot r} = -\frac{1}{N} \sum_{j=1}^{N} \sin (\theta_j - \phi)
{\dot \theta}_j, \qquad
{\dot \phi} = \frac{1}{rN} \sum_{j=1}^{N} \cos (\theta_j -
\phi) {\dot \theta}_j.
\end{equation}
Thus, we can combine \eqref{KM-mf} and \eqref{B-4} to get the
evolutionary system:
\begin{align}
\begin{aligned} \label{B-5}
{\dot r} &= -\frac{1}{N} \sum_{j=1}^{N}  \sin (\theta_j
- \phi) \Big( \omega_j -  K r\sin (\theta_j - \phi) \Big), \\
{\dot \phi} &= \frac{1}{rN} \sum_{j=1}^{N} \cos (\theta_j - \phi)
\Big( \omega_j - K r\sin (\theta_j - \phi) \Big).
\end{aligned}
\end{align}
Before we close this part, we present the relationship between the phase diameter $D(\Theta) := \max_{1 \leq i, j \leq N} |\theta_i - \theta_j|$ and the order parameter $r$ in the following proposition.
\begin{proposition} \label{P2.1}
\emph{\cite{C-H-J-K}}
Suppose that the phase configuration $\Theta = (\theta_1, \cdots, \theta_N)$ is confined in a half circle such that 
\[ D(\Theta) < \pi. \]
Then, the following estimates hold.
\begin{enumerate}
\item
The order parameter $r$ is rotationally invariant and $r \geq \cos \frac{D(\Theta)}{2}$.
\item
The order parameter $r$ satisfies
\[ r  = 1 \quad \iff \quad D(\Theta) = 0, \quad \mbox{i.e.,} ~~\theta_1 = \theta_2 = \cdots = \theta_N.   \]
\end{enumerate}
\end{proposition}
\begin{proof} (i) For the rotational invariance of $r$, it suffices to show that the configuration $\Theta + \alpha {\bf 1} :=(\theta_1 + \alpha, \cdots, \theta_N + \alpha)$ and $\Theta := (\theta_1, \cdots, \theta_N)$ have the same order parameter. Let ${\bar r}$ and $r$ be the order parameters for $\Theta +\alpha {\bf 1}$ and $\Theta$, respectively.  
\[ {\bar r} = \frac{1}{N} \Big| \sum_{j=1}^{N} e^{{\mathrm i} (\theta_j + \alpha)} \Big| =  \frac{1}{N} \Big|e^{{\mathrm i} \alpha} \sum_{j=1}^{N} e^{{\mathrm i} \theta_j} \Big| =  \frac{1}{N} \Big| \sum_{j=1}^{N} e^{{\mathrm i} \theta_j} \Big| = r.    \]
(ii)  Since the order parameter is rotationally invariant, without loss of generality, we may assume 
\[ \theta_i \in \Big( -\frac{D(\Theta)}{2},   \frac{D(\Theta)}{2} \Big), \quad 1 \leq i \leq N. \]

\noindent $\bullet$~Case A: We first prove that if $D(\Theta) = 0$, then $r = 1$. For this, we note that 
\[ r = \frac{1}{N} \Big |\sum_{j=1}^{N} e^{{\mathrm i} \theta_j} \Big| =  \frac{1}{N} \Big[ \Big( \sum_{j=1}^{N} \cos \theta_j    \Big)^2 + \Big(  \sum_{j=1}^{N} \sin \theta_j     \Big)^2         \Big]^{\frac{1}{2}} \geq \frac{1}{N} \sum_{j=1}^{N} \cos \theta_j  
\geq \cos \frac{D(\Theta)}{2}.   \]
Thus, if $D(\Theta) = 0$, then
\[ 1\geq r \geq \cos \frac{D(\Theta)}{2} = 1, \quad \mbox{i.e.,} \quad r = 1. \]

\noindent $\bullet$ Case B: Note that 
\[ z_i = e^{{\mathrm i} \theta_i} \in \mathbb S^1 \subset \mathbb C, \qquad  |z_1 + \cdots + z_N | = Nr, \]
Suppose that $r = 1$, then we have
\begin{equation} \label{B-6}
  |z_1 + \cdots + z_N | = N. 
\end{equation}   
We now claim:
\[ \theta_i = \theta_j, \quad 1 \leq i, j \leq N. \]
It follows from the relation \eqref{B-6} that we have
\[
0 = |z_1 + \cdots + z_N |^2 - N^2 = \sum_{i=1}^{N} |z_i|^2 - N^2 + 2\sum_{1 \leq i < j \leq N} \cos(\theta_i - \theta_j). 
\]
This yields 
\[ \sum_{1 \leq i < j \leq N} \cos(\theta_i - \theta_j) = \frac{N(N-1)}{2}. \]
This again implies
\[ \cos(\theta_i - \theta_j) = 1 \quad \mbox{i.e.,} \quad \theta_i - \theta_j = 0, \quad 1 \leq i, j \leq N. \]
\end{proof}

\subsubsection{A gradient flow formulation} \label{sec:2.1.2} Note that the right hand side of \eqref{KM} is $2\pi$-periodic, so the system \eqref{KM} is a dynamical system on $N$-tori $\mathbb T^N$. However, for the description of a gradient flow formulation, we lift  the system \eqref{KM} to a dynamical system on $\mathbb R^N$ by a straightforward lifting. So the trajectory of $\Theta = (\theta_1, \cdots, \theta_N)$ is not necessarily bounded as a subset of $\mathbb R^N$.  In  \cite{V-W}, from an analogy with the XY-model in statistical physics, Hemmen and Wreszinski observed that the Kuramoto model \eqref{KM} can be formulated as a gradient flow with an analytic potential. More precisely, they introduced the analytic potential $V_p$: 
\begin{equation} \label{p-pot}
 V_p(\Theta) :=-\sum_{i=1}^N\omega_i \theta_i+\frac{K}{2N}\sum_{i, j=1}^N \Big(1-\cos(\theta_j-\theta_i) \Big). 
\end{equation} 
By direct calculation, it is easy to see that the Kuramoto model \eqref{KM} can be rewritten as a gradient flow form:
\begin{equation} \label{Gr-KM}
  {\dot \Theta}(t)=-\nabla V_p(\Theta).
\end{equation}
Note that the potential function $V_p$ is neither convex nor bounded below a priori. This gradient formulation has been crucially used to prove the emergence of phase-locked states from generic initial configurations \cite{D-X, H-K-R}.

\subsubsection{Emergent dynamics} \label{sec:2.1.3} In this part, we recall the uniform boundedness of fluctuations of phases around the averaged motion, as well as mass concentration of the identical oscillators around the average phase. Since we regard the system \eqref{Gr-KM} as a dynamical system on $\mathbb R^N$, the uniform boundedness of fluctuations around the average phase motion is not clear a priori. In \cite{H-K-R}, authors showed that the relative phases are uniformly bounded, if more than half of oscillators are confined in a small arc and the coupling strength is sufficiently large.
\begin{proposition}\label{P2.2}
\emph{\cite{H-K-R}}
Suppose that the initial configuration $\Theta_0$ satisfies
\[ \theta_{j0} \in [-\pi, \pi), \quad 1 \leq j \leq N, \]
and let $n_0, \ell$, and $K$ satisfy
\begin{align*}
& n_0 \in \mathbb Z_+ \cap \Big(\frac{N}{2}, N \Big ],  \quad \ell  \in \Big(0,2\cos^{-1}\frac{N-n_0}{n_0} \Big), \\
& \max_{1 \leq j, k \leq n_0 } |\theta_{j0} - \theta_{k0}| < \ell, \quad K>\frac{ \displaystyle\max_{i,j} |\omega_i - \omega_j|}{\frac{n_0}{N} \sin \ell - \frac{2(N-n_0)}{N} \sin \frac{\ell}{2}}.
\end{align*}
Let $\Theta$ be a global solution to the system \eqref{KM}. Then, we have
\[ \sup_{0 \leq t < \infty} \max_{i,j} |\theta_i(t) - \theta_j(t)| \leq 4 \pi +  \ell. \]
\end{proposition}
\begin{remark} For identical oscillators with $\omega_i = \omega$, Dong and Xue \cite{D-X} used the gradient flow formulation \eqref{Gr-KM} to show that for all initial configurations and positive coupling strength, the Kuramoto flow \eqref{KM} tends to phase-locked states. Moreover, the uniform boundedness of Proposition \ref{P2.1} and gradient flow formulation yields the formation of phase-locked states. 
\end{remark}
Next, we consider the Kuramoto model for identical oscillators:
\begin{equation} \label{KM-id}
{\dot \theta}_i = \omega + \frac{K}{N} \sum_{j=1}^{N} \sin (\theta_j - \theta_i), \quad i =1, \cdots, N.
\end{equation}
For a dynamical solution $\Theta(t)$ to \eqref{KM-id},  we divide the oscillator set ${\mathcal N} :=\{ 1, \cdots, N \}$ into synchronous and anti-synchronous oscillators, with respect to the overall phase $\phi(t) :=\omega t $ :
\[ {\mathcal I}_s :=\{j: \lim_{t\rightarrow \infty} |\theta_j(t)- \phi(t)| = 0 \}, \qquad
{\mathcal I}_{b}:=\{j: \lim_{t\rightarrow \infty} |\theta_j(t)- \phi(t)|
= \pi\}. \]

\begin{proposition}\label{P2.3}
\emph{\cite{H-K-R}}
Let $\Theta= (\theta_1, \cdots, \theta_N)$ be a solution to \eqref{KM-id} with initial configuration $\Theta_0$ satisfying the following conditions:
\[ \sum_{i=1}^{N} \Omega_i = 0, \quad r_0  > 0, \qquad  \theta_{k0} \not = \theta_{j0}, \qquad 1 \leq k, j \leq N. \]
Then, $\{ {\mathcal I}_s, {\mathcal I}_b \}$ is a partition of ${\mathcal N}$ and we have
\[ |{\mathcal I}_{b}|  \leq 1, \]
where $|A|$ denotes the cardinality of the set $A$.
\end{proposition}
\begin{remark}
1. The state with $|{\mathcal I}_{b}| = 1$ is called the bi-polar state, which is well known to be unstable \cite{C-H-J-K}. 

2. For a detailed survey on the Kuramoto model \eqref{KM}, we refer to  \cite{A-B, B-S, C-H-J-K,  C-S, D-X, D-B2, D-B, H-K-R, H-K-P-Z,  H-L-X, J-M-B}). 
\end{remark}

\subsection{The Kuramoto-Sakaguchi (K-S) equation} In this subsection, we discuss the kinetic counterpart of \eqref{KM}. Consider a situation where the number of oscillators, which we denote by $N$ in \eqref{KM} goes to infinity. In this mean-field limit, it is more convenient to rewrite the system \eqref{KM}, as a dynamical system on the phase space 
$\mathbb T \times \mathbb R,~~\mathbb T :=\mathbb R/2\pi \mathbb Z = [0, 2\pi]$ for $(\theta, \omega)$:
\begin{equation} \label{re-ku}
\begin{dcases}
\frac{d\theta_i}{dt} = \omega_i - \frac{K}{N}
\sum_{j=1}^{N} \sin(\theta_i - \theta_j), \\
\frac{d\omega_i}{dt} = 0, \quad t > 0.
\end{dcases}
\end{equation}
Since we are dealing with a large oscillator system $N \gg 1$, we can introduce a probability density function $f = f(\theta, \omega, t)$ to approximate the $N$-oscillator system \eqref{re-ku}. Based on the standard BBGKY Hierarchy argument \cite{Ke}, we can derive the K-S equation:
\begin{align}
\begin{aligned} \label{K-S}
\displaystyle & \partial_t f + \partial_{\theta} (v[f] f) = 0, \qquad (\theta, \omega) \in \mathbb T \times \mathbb R,~~t > 0, \\
& v[f](\theta, \omega, t) = \omega - K F [\rho], \quad F [\rho] := \int_{\mathbb T} \sin(\theta-\theta_*) \rho(\theta_*,t) ~ d\theta_*.
\end{aligned}
\end{align}
where $\rho = \rho(\theta, t)$ is the local mass density function, which corresponds
to the $\theta$-marginal density function of $f$:
\[ \label{rho}
  \rho(\theta, t) := \int_{\mathbb R} f(\theta, \omega, t) ~ d\omega, \quad t \geq 0. 
\] 

\begin{remark}
The rigorous derivation from \eqref{KM} to \eqref{K-S} was done by Lancellotti \cite{La} using Neunzert's method \cite{Ne}.
In \cite[Theorem 2 and Remark 3]{La}, Lancellotti showed that there exists a unique classical solution to 
\eqref{K-S}, whenever the initial datum is $C^{1}.$  
\end{remark}

We next recall some conservation laws for the K-S equation.
\begin{lemma}
\label{L2.1}
\emph{\cite{A-H-P}}
Let $f =f(\theta, \omega, t)$ be a ${\mathcal C}^1$-solution to \eqref{K-S}, with initial datum
$f_0$ satisfying the following conditions:
\[
  \int_{\mathbb T} f_0 ~ d\theta = g(\omega), \quad \iint_{\mathbb T \times \mathbb R} f_0 ~ d\theta d\omega = 1.
\]
Then, we have
\[ \int_{\mathbb T} f (\theta, \omega) ~ d\theta = g(\omega), \quad \int_{\mathbb T \times \mathbb R } f ~ d\theta d\omega  = 1, \quad t \geq 0.  \]
\end{lemma}
\begin{proof} The identities follow from a direct computation as follows.

\[ \frac{d}{dt} \int_{\mathbb T} f ~ d\theta = \int_{\mathbb T} \partial_t f ~ d\theta = -\int_{\mathbb T} \partial_{\theta} (v[f] f) ~ d\theta =0, 
 \]

and
\[ \frac{d}{dt} \iint_{\mathbb T \times \mathbb R} f ~  d\theta d\omega= \int_{\mathbb R} \frac{d}{dt} \Big( \int_{\mathbb T} f  ~ d\theta \Big)
 d\omega = 0. \]
\end{proof}
\begin{remark} Lemma \ref{L2.1} yields that for any test function $h=h(\omega)$, we have
\[ \iint_{\mathbb T \times \mathbb R} h(\omega) f ~ d\theta d\omega = \int_{\mathbb R} h(\omega) \Big( \int_{\mathbb T} f  ~ d\theta \Big) d\omega =  \int_{\mathbb R} h g ~ d\omega =  \int_{\mathbb T \times \mathbb R} h(\omega) f_0 ~ d\theta d\omega. 
\]
\end{remark}

\subsubsection{Order parameters} \label{sec:2.2.1} In this part, we introduce real-valued order parameters $R = R(t)$ and $\phi = \phi(t)$ which measure the overall degree of synchronization for the K-S equation \eqref{K-S}. 
Such order parameters will be used to simplify the expression of the system \eqref{K-S}.  As a straightforward generalization of the order parameters \eqref{order-p} for the Kuramoto model, we define real order parameters $R$ and $\phi$ for the K-S equation \cite{A-B, H-K-P}:
\begin{equation}\label{O-K}
R(t) e^{{\mathrm i} \phi(t)} := \iint_{\mathbb T \times \mathbb R} e^{{\mathrm i}\theta} f ~  d\theta d\omega  = \int_{\mathbb T} e^{{\mathrm i}\theta} \rho ~ d\theta, \quad t \geq 0.
\end{equation}
To avoid the confusion with the amplitude order parameter $r$ for \eqref{order-p}, we use a capital letter $R$ instead of $r$ for the K-S equation. As for the Kuramoto model, the average phase $\phi$ is well-defined when $R>0$. We divide \eqref{O-K} by $e^{i\phi}$ on both sides to get
\begin{equation}\label{O-K-R}
R(t) = \int_{\mathbb T} e^{{\mathrm i} (\theta - \phi(t))} \rho ~ d\theta =  \int_{\mathbb T} \langle e^{{\mathrm i} \theta}, e^{{\mathrm i} \phi(t)} \rangle \rho ~ d\theta.
\end{equation}
By comparing real and imaginary part on both sides of \eqref{O-K-R}, we obtain 
\begin{equation}\label{O-K-1}
R(t) = \int_{\mathbb T} \cos(\theta - \phi(t)) \rho (\theta, t) ~ d\theta
\qquad  {\rm{and}} \qquad 
0= \int_{\mathbb T} \sin(\theta - \phi(t)) \rho (\theta, t)  ~ d\theta.
\end{equation}
On the other hand, we use \eqref{O-K-1} to rewrite the linear operator $F[\rho]$ in terms of order parameters:
\begin{align}
\begin{aligned}\label{A-1}
& F[\rho] = \int_{\mathbb T }\sin(\theta - \theta_*) \rho (\theta_*, t)  ~ d\theta_*\\
& \hspace{0.5cm} = \int_{\mathbb T} \sin \big( (\theta - \phi) - (\theta_* -\phi) \big) \rho (\theta_*, t)  ~ d\theta_* \\
& \hspace{0.5cm} = \int_{\mathbb T} \big( \sin (\theta - \phi) \cos (\theta_* -\phi)  - \cos(\theta - \phi) \sin (\theta_* -\phi) \big) \rho (\theta_*, t)  ~ d\theta_* \\
&\hspace{0.5cm} = \sin(\theta - \phi) \int_{\mathbb T} \cos(\theta_* - \phi) \rho (\theta_*, t)  ~ d\theta_* - \cos(\theta - \phi) \int_{\mathbb T} \sin(\theta_* - \phi) \rho (\theta_*, t)  ~ d\theta_*\\
&\hspace{0.5cm} = R \sin(\theta - \phi).
\end{aligned}
\end{align}
Thus, we can  rewrite the Kuramoto-Sakaguchi equation \eqref{K-S} as an equivalent form:
\begin{equation}\label{KM-K-I-1}
\partial_t f + \partial_\theta \Big( \omega f - K R \sin(\theta- \phi) f \Big) = 0.  
\end{equation}
For the identical oscillator case with $g(\Omega)=\delta(\omega),$ by integrating \eqref{KM-K-I-1} with respect to $\omega$ we obtain
\begin{equation} \label{D-0-0}
\partial_t \rho+\partial_\theta  \Big(\omega\rho  - K R \sin(\theta- \phi)\rho  \Big) = 0.  
\end{equation}

\subsubsection{A gradient flow formulation} 
In this part,  we discuss the gradient flow formulation of the K-S equation with $g(\omega) = \delta(\omega)$, i.e., we will write the system \eqref{KM} as a gradient flow in the probability space $\mathbb{P}(\mathbb{T})$ equipped with the Wasserstein metric $W_2$. For this, we adapt the differential calculus introduced by Felix Otto \cite{Otto}, which has proven to be a powerful tool for the study of the
dynamics and stability of evolutionary equations (we refer the reader to  \cite{J-K-O,Ottom,Otto} for the pioneering works on this topic) (see Appendix \ref{App-A}). We first set up a potential function for our gradient flow. 
In analogy with \eqref{p-pot} for the Kuramoto model, it is natural to come up with the following potential function for the K-S equation \eqref{KM-K-I}.
\[ \label{PO-K-I}
V_k(\rho(t)) := \frac{K}{2} \iint_{\mathbb T^2}(1-\cos(\theta_* - \theta)) \rho(\theta_*, t) \rho(\theta, t) ~ d\theta_* d\theta.
\]
Here the subscript $k$ stands for ``{\it kinetic}". We next present another handy expression for $V_k(\rho)$ in terms of the order parameter $R$ given in \eqref{O-K}. First, we let $\sigma:\mathbb{T}\rightarrow\mathbb{R}^{2}$ be defined by $\sigma(\theta) := (\cos \theta, \sin \theta)$. Then, given $\rho$ in $\mathbb{P}(\mathbb{T}),$ the above potential function $V_k$ can
be written as 
\begin{align}\label{eq:V}
V_k(\rho(t))=\frac{K}{2}-\frac{K}{2} \iint_{\mathbb T^2} \sigma(\theta)\cdot\sigma(\theta_{*})\rho(\theta,t)\rho(\theta_{*},t) ~ d\theta_* d\theta,
\end{align}
where $\sigma(\theta) \cdot \sigma (\theta_*)$ denotes the inner product between the two vectors in $\mathbb R^2$. 
We define the total momentum of $\rho$ by 
\begin{align}\label{eq:J f}
J:=\int_{\mathbb T} \sigma (\theta) \rho (\theta) ~ d\theta \in \mathbb R^2.
\end{align}
When we regard $Re^{i\phi}$ in \eqref{O-K} as a vector in $\mathbb{R}^{2}$, we have
\[
J=Re^{{\mathrm i}\phi}.
\]

From this, we can write
\begin{equation} \label{eq: V simple}
V_k(\rho)=\frac{K}{2} (1- |J|^{2}), \quad \mbox{equivalently} \quad  V_k (\rho)=\frac{K}{2} (1-R^{2}).
\end{equation}
By straightforward application of Otto calculus (see Appendix \ref{App-A}), we have
\[
{\rm grad}_\rho V_k = -K  \nabla (\sigma(\theta)\cdot J)=K R\sin(\theta-\phi),
\]
where ${\rm grad}_\rho$ denotes the gradient with respect to the Wasserstein metric. We now substitute this in \eqref{eq: grad flow} to get the gradient flow of $V_k$ as the 
 one-parameter family $t \in [0, \varepsilon) \mapsto \rho_t \in \mathbb P(\mathbb T)$
satisfying 
\[
\partial_{t}\rho= \partial_\theta \left(\rho {\rm grad}_\rho V_k   \right) = \partial_\theta \left(\rho K R \sin (\theta-\phi) \right).
\]
This is the same as \eqref{D-0-0} when $\omega = 0$, which verifies the gradient flow structure of the K-S equation. 

\begin{figure}
\centering
\subfigure[$I_\delta^+$]{\includegraphics[width=0.35\textwidth]{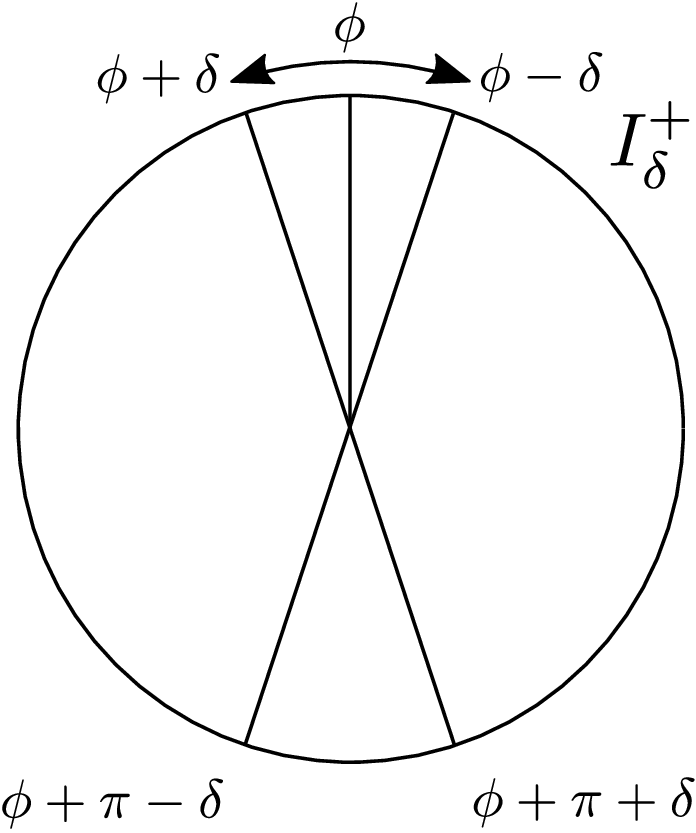}\label{fig:iplus}}
\hspace{0.1\textwidth}
\hfill
\subfigure[$I_\delta^-$]{\includegraphics[width=0.35\textwidth]{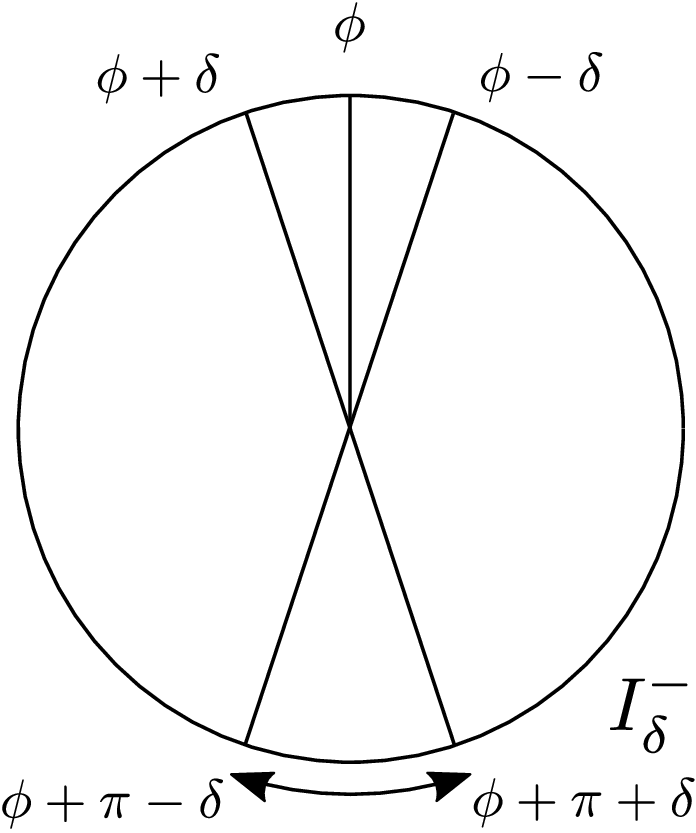}\label{fig:iminus}}
\caption{Geometric descriptions of $I_\delta^+$ and $I_\delta^-$}
\label{fig:1}
\end{figure}

\section{Discussion of the main results} \label{sec:3}
\setcounter{equation}{0}
In this section, we briefly present our three main results on the emergent dynamics of the K-S equation. The proofs will be  given in the following three sections.
\subsection{Emergence of point attractors} In this subsequent, we briefly discuss how point attractors can emerge from 
generic smooth initial data, asymptotically along the Kuramoto-Sakaguchi flow with $g = \delta$. For the finite-dimensional Kuramoto model for $N$-identical oscillators, if we start from generic initial datum 
\[ r_0 = \frac{1}{N} \Big | \sum_{j=1}^{N} e^{{\mathrm i} \theta_{j0}} \Big|  > 0, \qquad \theta_{i0} \not = \theta_{j0}, \quad 1 \leq i, j \leq N, \]
then, any positive coupling strength $K$ will push the initial configuration to two possible scenarios. First scenario to happen is that the phases will concentrate near the time-dependent average phase which is also dynamic along the Kuramoto flow, and the second scenario will be a bi-polar configuration where $(N-1)$ oscillators will aggregate toward the average phase, and the remaining one oscillator will approach the antipodal phase of the average phase (see Proposition \ref{P2.2}). By an easy perturbation argument, we see that the bi-polar configuration is unstable (see \cite{C-H-J-K}). Thus, the completely phase synchronized state which we call one-point attractor in this paper, is the only stable asymptotic state for the Kuramoto model in any positive coupling regime. Thus, it is interesting to figure out what will happen for the K-S equation which can be obtained from the Kuramoto model as $N \to \infty$:
\begin{quote}
Does the same asymptotic patterns as in the Kuramoto flow emerge for the K-S flow from generic initial configurations?
\end{quote}

In fact, we will argue that the unstable bi-polar configurations can not be reachable from a generic smooth initial datum along the K-S flow. To justify and quantify our claim, we will first define a set consisting of two disjoint intervals containing the average phase and its antipodal point, respectively, and then we will show that the mass outside this set will approach to zero asymptotically. Finally, we will argue that the mass around the antipodal point of the average phase will tend to zero exponentially, i.e., all mass will aggregate in any small neighborhood of the average phase. Hence, asymptotically, one-point attractor will emerge. In processing the above procedures, analysis on the dynamics of $R$ and $\phi$ will play a key role.  To quantify the above sketched arguments, we define for a positive constant $\delta >0$, a set $I_{\delta} :=  I_\delta^+ \cup I_\delta^-$ (see Figure~\ref{fig:1}): 
\begin{align}
\begin{aligned} \label{NE-0}
I_\delta^+(t) &:= \{\theta\in \mathbb T: |\theta - \phi(t)|  < \delta  \}, \\
I_\delta^-(t) &:= \{\theta \in\mathbb T: |\theta - (\phi(t) + \pi)| < \delta\}.
\end{aligned}
\end{align}
Consequently, we will see
that eventually all the mass of
$\rho$
  will be concentrated in the region
  $I_{\delta}^{+}\cup I_{\delta}^{-}$
  for every $\delta>0.$ Since $R$
  is monotonic, if $R_0>0$, we necessarily have
\[  \int_{I_\delta^+} \rho ~ d\theta >  \int_{I_\delta^-} \rho ~ d\theta. \]
Then, it is reasonable to expect that  the mass will exit $I_{\delta}^{-}$
  and enter the interval $I_{\delta}^{+}$, and eventually will yield a point attractor,
  provided that we can control the
  rotation of $\phi(t)$. This is  proved
  in \eqref{NN-4} and \eqref{NN-5}, whereas in Proposition \ref{P4.1}
  we can control the rotation of $\phi(t),$ using the fact that
  $\dot{R}(t)\rightarrow 0.$ \newline
  
We are now ready to state our first result on the emergence of point attractors. 
\begin{theorem} \label{T3.1}
Suppose that the coupling strength, the density function $g$ and the initial datum satisfy
\[ K > 0,\quad  g(\omega) =\delta(\omega), \quad \rho_0 \in C^1, \quad \int_{\mathbb T} \rho_0 ~ d\theta = 1 \quad   {\rm{and}} \quad R_0 := R(0) > 0.      \]
Then, for a classical solution $\rho  ~ : ~  \mathbb{T}\times\mathbb{R_{+}}\rightarrow \mathbb{R}$
to \eqref{D-0-0}, the mass concentrates around the average phase $\phi(t)$ asymptotically. More precisely, for any $\delta > 0$, there exists $T_1 = T_{1}(\delta) \geq0,$ such that
\begin{align*}
& \lim_{ t \to \infty} \int_{I_\delta^+(t) } \rho(\theta, t) ~ d\theta = 1 \quad {\rm{and}} \cr
& \int_{I_{\delta}^{-}(t)}|\rho(\theta,t)|^{2} ~ d\theta\leq e^{-\frac{R(0)(\cos\delta)}{2}K(t-T_{1})}\int_{I_{\delta}^{-}(T_{1})}|\rho(\theta,T_{1})|^{2} ~ d\theta \qquad\forall~t\geq T_{1}. 
\end{align*}
\end{theorem}
\begin{proof}
The proof will be given in Section \ref{sec:4.2}.
\end{proof}
\begin{remark} 
Concentration of mass around  $\phi$  has been proved in \cite{B-C-M} by a different argument without the exponential decay estimate of mass in the interval $I_{\delta}^{-}$. 
\end{remark}

\vspace{0.2cm}

\subsection{Emergence of phase concentration} \label{sec:3.2} In this subsection and the next, we will show that there exists an interval centered at $\phi$ where the mass will concentrate asymptotically, when the coupling strength is sufficiently large. Moreover, we will also present a lower bound for the asymptotic amplitude order parameter $R = R(t)$ which tends to unity as $K \to \infty$. Before we discuss our second and third main results, we first recall corresponding results for the Kuramoto model in a large coupling strength regime (see Proposition \ref{P2.2}). Consider an ensemble consisting of $N$ Kuramoto oscillators and assume that more than $\frac{N}{2}$ oscillators are confined in a small interval ${\mathcal I}$ at some instant. We divide the ensemble into two sub-ensembles consisting of confining oscillators in the interval ${\mathcal I}$ and the rest of it. In this situation, if we choose a sufficiently large coupling strength, then the confining oscillators will stay in the interval and drifting oscillators might enter neighboring copies of the interval $({\mathcal I} + 2\pi) \cup ({\mathcal I}-2\pi)$. Thus, trajectories of oscillators will be bounded eventually. We now return to the K-S equation. As we have seen in the previous subsection, for identical oscillators, the total mass will concentrate asymptotically on the average phase. Thus, for the density function $g = g(\omega)$ with a compact support, we can imagine that a similar scenario for the Kuramoto model will happen, i.e., once we choose a large coupling strength compared to the size of the support of the natural frequency density function $g$, we can guess that mass will be confined inside some small interval around the average phase. In fact, this is the case. To be more precise, assume that $g = g(\omega)$ is compactly supported inside the interval $[-M, M]$ and some significant portion of local mass of $f_0$ is concentrated on some time-dependent interval $L(t)$ centered the average phase $\phi(t)$: for some positive constant $C > 0$
\[   \iint_{L(0) \times \mathbb R} f_0(\theta, \omega) ~ d\theta d\omega >   C.  \]
In this setting, for  a large coupling strength $K \gg M$, we will derive 
\begin{align*}
& \frac{d}{dt} \iint_{L(t) \times \mathbb R} f(\theta, \omega, t) ~ d\theta d\omega  \geq 0, \cr
& \frac{d}{dt} \int_{L(t)} |f(\theta, \omega, t)|^2 ~ d\theta  \geq K |{\mathcal O}(1)|  \int_{L(t)} |f(\theta,\omega, t)|^2 ~ d\theta, \quad \omega \in \mbox{supp}~g(\omega). 
\end{align*}
Thus, we have
\[ \frac{d}{dt} \iint_{L(t) \times \mathbb R} f(\theta, \omega, t) ~ d\theta d\omega \geq 0, \qquad  \int_{L(t)} |f(\theta, \omega, t)|^2  ~ d\theta \geq C e^{Ct}, \quad \omega \in \mbox{supp}~g(\omega).    \]
Note that the first estimate says that the mass on the set $L(t) \times \mathbb R$ in the phase space is nondecreasing, i.e., the mass does not leak to complement of this set, and the second estimate tells us that for each fixed $\omega \in \mbox{supp}~g(\omega)$, there should be some mass concentration. In this sense, we may say that the time-dependent set $L(t) \times \mathbb R$ converges to an invariant manifold for the K-S flow. For a precise statement, we set 
\begin{equation} \label{MM}
{\mathcal M}_*(\varepsilon_0, \gamma_0) := \frac{2+\varepsilon_0+\cos\gamma_0}{(1+\sin\gamma_0)(1+\cos\gamma_0)}. \end{equation}
We now ready to state our second main result summarizing the above arguments. 
\begin{theorem} \label{T3.2}
Suppose that the following conditions hold.
\begin{enumerate}
\item
The frequency density function $g=g(\omega)$ and coupling strength $K$ satisfies
\[ \mbox{\rm supp}\,  g(\omega) \subset (-M, M), \quad  K>\frac{M}{\varepsilon_0}\bigg(1+\frac{1}{\varepsilon_0}\bigg), \quad \varepsilon_0 \in  \Big(0, \frac{3\sqrt{3}}{4}-1 \Big). \]
\item
Suppose that initial datum $f_0$ satisfies
\begin{align*}
& (i)~f_0(\theta,\omega) = 0  \hspace{1em}{\rm{in}}\hspace{1em}
\mathbb{T}\times (\mathbb{R}\backslash[-M,M]), \quad ||f_0||_{L^{\infty}} < \infty, \cr
\vspace{0.5cm}
& (ii)~\iint_{L^+_{\gamma_0}(0) \times \mathbb R} f(\theta, \omega, 0) ~ d\theta d\omega  \geq {\mathcal M}_*(\varepsilon_0, \gamma_0), 
\end{align*}
where $\gamma_0$ satisfies
\[ \frac{\pi}{3} \leq \gamma_0 < \arcsin\Big(1 - \frac{2\varepsilon_0}{2\sqrt{3} + 1} \Big). \]
\end{enumerate}
Then, for any $C^1$-solution to \eqref{K-S}, there exists a time-dependent interval $L(t) \subset \mathbb T$ centered around $\phi(t),$ with fixed width  such that 
\[ \frac{d}{dt} \iint_{L(t)\times \mathbb R} f(\theta, \omega, t) ~ d\theta d\omega \geq 0, \qquad  \int_{L(t)} |f(\theta, \omega, t)|^2  ~ d\theta \geq C e^{Ct}, \quad \omega \in \mbox{supp}~g(\omega).    \]
\end{theorem}
\begin{proof} We present its proof in Section \ref{sec:5}.
\end{proof}

\subsection{Asymptotic dynamics of the order parameter} Notice that in Theorem \ref{T3.2}, we assumed a certain lower bound on the mass in a certain interval. We remove such assumption for our third main result which we now describe. For the Kuramoto model \eqref{KM}, the dynamics of the order parameter $r$ does play a key role in the recent resolution of the complete synchronization problem in \cite{H-K-R}. As noticed in Proposition \ref{P2.1}, for the Kuramoto model, if the order parameter $r$ is close to 1, we can say that the configuration is close to complete phase synchronization where all phases are concentrated at some common phase. Our third result is concerned about the estimation of  the order parameter in a large coupling strength regime. More precisely, we will obtain a positive lower bound as
\[  \liminf_{t\rightarrow\infty}R(t) \geq 1 - \frac{|{\mathcal O}(1)|}{\sqrt{K}}, \quad \mbox{for $K \gg 1$}. \]
This certainly implies that as $K \to \infty$, 
\[ \lim_{K \to \infty}  \liminf_{t\rightarrow\infty}R(t) = 1. \]
By Proposition \ref{P2.1}, this means the formation of complete phase synchronization in $K \to \infty$ limit, which can be understood as an emergence of practical synchronization. Below, we state our third result.
\begin{theorem}\label{T3.3} Let $ f  ~ : ~  \mathbb{T}\times{\mathbb{R}}\times\mathbb{R^{+}}\rightarrow \mathbb{R}$
 be a classical solution to \eqref{K-S}. Suppose $g$ is supported on the 
interval $[-M,M]$, $R(0):=R_0 > 0$, and $K$ is sufficiently large
(depending on the support of $g$ and $1/R_0$). Then, 
\[
 \liminf_{t\rightarrow\infty}R(t)\geq R_\infty := 1+\frac{M}{K}-\sqrt{\frac{M^{2}}{K^{2}}+4\frac{M}{K}} 
\]
and
\[
  \lim_{t \rightarrow \infty} || f_{t} {\mathbbm 1}_{\mathbb{T}\backslash L_{\infty}(t)}||_{L^{\infty}(\mathbb{T}\times \mathbb{R})}=0. 
\]
Here, $L_{\infty}(t)\subset \mathbb{T}$ is a time 
dependent interval, centered at $\phi(t)$ with the constant width
\[
{\arccos}\bigg(\sqrt{1-\bigg[\frac{M}{K}\frac{(1+R_\infty)}{R_\infty^{2}}+\frac{1-R_\infty}{R_\infty}\bigg]^{2}}\bigg) + \varepsilon, 
\]
where $\varepsilon$ is an arbitrary constant in $(0,1)$. Notice that as  $K\rightarrow \infty$ the width of $L_{\infty}$ can be made arbitrarily small and $R_\infty$ tends to $1.$  
\begin{proof}
The proof will be given in Section \ref{sec:6.3}.
\end{proof}

\end{theorem}
In the following three sections, we will present proofs of the main theorems and of many lemmata. 
\section{Emergence of point attractors} \label{sec:4}
\setcounter{equation}{0}
In this section, we present existence of point attractors for the K-S equation with $g = \delta$  from a generic initial datum using the time-asymptotic approach.  Without loss of generality, we may assume that the common natural frequency $\omega_c$ is zero and that the corresponding density function $g =g(\omega)$ satisfies 
\[ g(\omega) = \delta_0. \] 

\noindent We first note that the local mass density $\rho$ satisfies the following two equivalent equations:
\begin{equation} \label{KM-K-I}
\begin{dcases}
\partial_t \rho  + \partial_\theta (v[\rho] \rho) = 0, \quad \theta \in \mathbb T,~~t > 0, \\
v[\rho](\theta, t) = - K \int_{\mathbb T} \sin (\theta - \theta_*) \rho(\theta_*, t)  ~ d\theta_*.
\end{dcases}
\quad \mbox{or} \quad \partial_t \rho - \partial_\theta  \Big(K R \rho \sin(\theta- \phi) \Big) = 0.
\end{equation}
In the following two subsections, we will study the dynamics of the amplitude order parameter $R$ and present the proof of Theorem \ref{T3.1}. 

\subsection{Dynamics of order parameters}\label{sec:4.1} 
In this subsection, we derive a coupled dynamical system for the order parameters $R$ and $\phi$, and using this dynamical system, we analyze their asymptotics. Note that the complete phase synchronization occurs if and only if $R \to 1$ as $t\to \infty$. \newline

For the derivation of dynamics of $R$ and $\phi$, we differentiate the defining relation \eqref{O-K} with respect to $t,$ and we obtain
\begin{equation}\label{D-O}
{\dot R} e^{{\mathrm i}\phi} + {\mathrm i} R \dot\phi e^{{\mathrm i}\phi} = \int_{\mathbb T} \partial_t \rho(\theta, t) e^{{\mathrm i}\theta}  ~ d\theta.
\end{equation}
We divide \eqref{D-O} by $e^{i\phi}$ on both sides to get
\begin{equation}\label{D-O-1}
{\dot R} + {\mathrm i} R \dot\phi  = \int_{\mathbb T} \partial_t \rho(\theta, t) e^{{\mathrm i}(\theta-\phi)}  ~ d\theta.
\end{equation}
We compare real and imaginary parts of \eqref{D-O-1} and employ \eqref{KM-K-I} to derive relations for $R$ and $\phi$:
\begin{align}
\begin{aligned} \label{R}
{\dot R} &= \int_{\mathbb T} \partial_t \rho(\theta, t) \cos (\theta-\phi)  ~ d\theta, \quad t > 0, \\
{\dot \phi} &=  \frac{1}{R} \int_{\mathbb T} \partial_t \rho(\theta, t) \sin (\theta-\phi)  ~ d\theta.
\end{aligned}
\end{align}
\begin{lemma}\label{L4.1}
Let $\rho$ be a solution to \eqref{KM-K-I} and let $R$ and $\phi$ be the order parameters defined by the relation \eqref{O-K-R}. Then, $R$ and $\phi$ satisfy 
\begin{align*}
& (i)~{\dot R} = KR \int_{\mathbb T}  \sin^2(\theta - \phi) \rho(\theta, t)  ~ d\theta, \quad {\dot \phi} =  -\frac{K}{2} \int_{\mathbb T} \sin\left(2(\theta - \phi)\right)  \rho(\theta, t)  ~ d\theta. 
\cr
& (ii)~{\ddot R} =  \frac{({\dot R})^2}{R} + 2R(\dot\phi)^2 - 2(KR)^2 \int_{\mathbb T} \sin^2 (\theta - \phi) \cos(\theta - \phi) \rho(\theta, t)   ~ d\theta.
\end{align*}
\end{lemma}
\begin{proof} 
(i) We use  \eqref{R} and \eqref{KM-K-I} to obtain
\begin{align}
\begin{aligned}\label{D-O-R}
{\dot R} &= \int_{\mathbb T} \cos(\theta - \phi) \partial_t \rho  ~ d\theta \\
&= KR \int_{\mathbb T} \cos(\theta - \phi) \partial_\theta \Big[ \rho(\theta, t) \sin(\theta - \phi) \Big]  ~ d\theta  \\
&= KR \int_{\mathbb T}  \sin^2(\theta - \phi) \rho(\theta, t)  ~ d\theta.
\end{aligned}
\end{align}
Similarly, we have
\begin{align*}
\dot \phi &= \frac{1}{R} \int_{\mathbb T} \sin(\theta - \phi) \partial_t \rho(\theta, t)  ~ d\theta \\
&= K \int_{\mathbb T} \sin(\theta - \phi) \partial_\theta \Big[ \rho(\theta, t) \sin(\theta - \phi) \Big]  ~ d\theta\\
&= - K \int_{\mathbb T} \sin(\theta - \phi) \cos(\theta - \phi) \rho(\theta, t)  ~ d\theta\\
& = -\frac{K}{2} \int_{\mathbb T} \sin\left(2(\theta - \phi)\right)  \rho(\theta, t)  ~ d\theta. 
\end{align*}

\noindent (ii) We again differentiate \eqref{D-O-R} with respect to $t$ to get
\begin{align*}
{\ddot R} &= K{\dot R} \int_{\mathbb T} \sin^2 (\theta - \phi) \rho(\theta, t)  ~ d\theta  - 2 K R {\dot \phi} \int_{\mathbb T} \sin(\theta - \phi)\cos(\theta - \phi) \rho(\theta, t)  ~ d\theta \\
&+ KR \int_{\mathbb T } \sin^2(\theta - \phi) \partial_t \rho(\theta, t)  ~ d\theta \\
&= \frac{({\dot R})^2}{R} + 2R({\dot \phi})^2 + (KR)^2 \int_{\mathbb T} \sin^2 (\theta - \phi) \partial_\theta \Big[ \rho(\theta, t) \sin(\theta - \phi) \Big]  ~ d\theta \\
&= \frac{({\dot R})^2}{R} + 2R({\dot \phi})^2 - (KR)^2 \int_{\mathbb T} 2\sin^2 (\theta - \phi) \cos(\theta - \phi) \rho(\theta, t)   ~ d\theta. 
\end{align*}
This yields the desired result. 
\end{proof}
Based on the dynamics given in Lemma \ref{L4.1},  we study asymptotics of $R$ and $\phi$.
\begin{proposition} \label{P4.1}
Let $\rho = \rho(\theta,t)$ be a solution to \eqref{KM-K-I} with the initial datum $\rho_0$ satisfying 
\[ R_0 > 0 \quad \mbox{and} \quad \rho_0 \in  \mathbb P(\mathbb T) \cap C^1(\mathbb T). \]
Then, there exists a positive constant $R_{\infty} \leq 1$ such that
\begin{align*}
&(i)~~\inf_{0 \leq t < \infty} R(t) \geq R_0 > 0, \quad \lim_{t \to \infty} (R(t), {\dot R}(t)) = (R_{\infty}, 0), \cr
&(ii)~|{\dot \phi}(t)| \leq K(1-R(t)), \quad t \geq 0, \quad \lim_{t \to \infty} |{\dot \phi}(t)|  = 0.
\end{align*}
\end{proposition}
\begin{proof}
(i)~Note that estimates in Lemma \ref{L4.1} yield the uniform boundedness of $\dot R, \ddot R$ and ${\dot \phi}$. So $R, \dot R$ and $\phi$ are Lipschitz continuous. Moreover, we have
\begin{equation} \label{D-1}
  {\dot R} \geq 0, \quad \mbox{thus},  \quad R(t) \geq R_0
  \hspace{1em} \forall t \in [0,\infty).   
\end{equation}  
On the other hand, since $R\leq1$, $R$ must converge to $R_{\infty} \leq 1$. 

Suppose $\dot R$ does not converge to zero. Since $\dot R \geq 0$, we can find a sequence of time $\{ t_n \}$ such that $t_n \uparrow \infty$ as $n \to \infty$ and $\dot R(t_n) > \alpha$ for some positive constant $\alpha$. From Lemma \ref{L4.1}, we attain the Lipschitz continuity of $\dot R$ with $|\ddot R| \leq \frac{K^2}{R_0} + K + 2K^2=:C_0$, which yieds
\[
\dot R(s) \geq \dot R(t_n) - C_0|t_n - s| \geq \alpha - \frac{\alpha}{2} = \frac{\alpha}{2}, \quad \forall s\in(t_n - \frac{\alpha}{2C_0}, t_n + \frac{\alpha}{2C_0}).
\]
Thus, we have
\[
\int_{t_n-\frac{\alpha}{2C_0}}^\infty \dot R(s) ~ds \geq  \int_{t_n-\frac{\alpha}{2C_0}}^{t_n + \frac{\alpha}{2C_0}} \dot R(s) ~ds \geq  \frac{\alpha^2}{2C_0}.
\]
This contradicts the convergence of $R$, i.e.,
\[
\lim_{a\to\infty}\int_a^\infty \dot R(s) ~ ds = 0.
\]
Hence, we attain $\dot R \to 0$ as $t \to \infty$.

\noindent (ii)~We next derive the estimate: 
\begin{equation} \label{D-1-1}
  -K(1-R) \leq   {\dot \phi} \leq K (1-R). 
\end{equation}  
For the second inequality, we use the second result in Lemma \ref{L4.1} to obtain
\begin{align*}
\dot\phi &= -K \int_{\mathbb T} \sin(\theta - \phi) \cos(\theta - \phi) \rho(\theta, t)  ~ d\theta\\
&= -K \int_{\mathbb T} \underbrace{\big(\sin(\theta - \phi) - 1 \big)\big(\cos(\theta - \phi) - 1\big)}_{\geq 0} \rho(\theta, t) ~ d\theta \\
&- K \int_\mathbb T \big(\cos(\theta - \phi) + \sin(\theta-\phi)  - 1\big)\rho(\theta, t)  ~ d\theta\\
&\leq -K \int_{\mathbb T} \cos(\theta - \phi) \rho(\theta,t)  ~ d\theta -K \int_{\mathbb T} \sin(\theta - \phi) \rho(\theta,t)  ~ d\theta + K\\
&= -KR + K = K(1-R).
\end{align*}
In the last line we used \eqref{O-K-1}. Similarly, we get the first inequality in \eqref{D-1-1}. For the remaining estimate, we use the formulas for ${\dot \phi}$ and ${\dot R}$ in Lemma \ref{L4.1}, the monotonicity of 
$R,$ and the Cauchy-Schwarz inequality to get
\begin{align*}
\mid\dot{\phi}\mid &\leq K \int_{\mathbb T} \mid\sin(\theta-\phi)\mid \rho  ~ d\theta\leq K \bigg(\int_{\mathbb T} \mid\sin(\theta-\phi)\mid^{2} \rho  ~ d\theta\bigg)^{\frac{1}{2}} \\
& \leq \sqrt{\frac{K}{R}} \sqrt{|{\dot R}|}  \leq \sqrt{\frac{K}{R(0)}} \sqrt{|{\dot R}|}, \hbox{ (since $R \geq R_0$ from (i))}. 
\end{align*}
Since ${\dot R} \to 0$ (see item (i)), we conclude \[ \lim_{t \to \infty} |{\dot \phi}(t)| = 0. \]
\end{proof}
We are now ready to provide the proof of Theorem \ref{T3.1} in the following subsection. 

\subsection{Proof of Theorem \ref{T3.1}} \label{sec:4.2} In this subsection, we present the proof of Theorem~\ref{T3.1}. Before we present a rigorous argument, we first discuss heuristics for the emergence of point attractor. Suppose that the coupling strength $K$ is positive,
the initial datum $\rho_{0}$ is ${\mathcal C}^{1},$ and $R_0 >0$. 
Then, since $R$ is bounded and monotonically increasing, 
$\dot{R}\rightarrow0$ as $t\rightarrow\infty$ (see Proposition \ref{P4.1}). It follows from Lemma \ref{L4.1}(i) that 
\[  \lim_{t \to \infty} \int_{\mathbb T}  \sin^2(\theta - \phi(t)) \rho(\theta, t)  ~ d\theta =0.         \]
Thus, the limiting behavior of $\rho$ will be one of the following states: for positive constant $\varepsilon \in (0, \frac{1}{2})$,
\[ \delta_{\phi_{\infty}}, \quad (1-\varepsilon)\delta_{\phi_{\infty}} +\varepsilon \delta_{(\phi_{\infty} + \pi)}. \]
 We then show that the latter case, i.e., bipolar state is not possible. The proof can be split into two steps:
\begin{itemize}
\item
Step A: Mass will concentrate asymptotically near at $\phi_\infty$ and/or $\phi_\infty +\pi$:
\[ \lim_{t \to \infty} \int_{\mathbb T \setminus I_\delta} \rho(\theta, t)  ~ d\theta = 0. \]
\item
Step B: Mass in  the interval $I_\delta^-$ decays to zero exponentially fast:
\[  \lim_{ t \to \infty} \int_{I_\delta^-(t) } \rho(\theta, t)  ~ d\theta = 0, \]
where the intervals $I_\delta$ and $I_{\delta}^{\pm}$ are defined in \eqref{NE-0}. 
\end{itemize}
\subsubsection{Step A (concentration of mass in the interval $I_\delta$):} In this part, we show that mass will concentrate on the interval $I_\delta$ asymptotically.  For any $\delta \in (0, \frac{\pi}{2})$, we claim: 
\begin{equation} \label{Claim-1}
 \lim_{t \to \infty} \int_{\mathbb T \setminus I_\delta} \rho(\theta, t)  ~ d\theta = 0. 
\end{equation} 
{\it The proof of claim \eqref{Claim-1}}: It suffices to show that for any $\varepsilon > 0$, there exists a finite time $t_*(\varepsilon) >0$ such that 
\[
\int_{\mathbb T \setminus I_\delta} \rho(\theta, t)  ~ d\theta < \varepsilon, \quad t > t_*(\varepsilon).
\]
Due to Proposition \ref{P4.1}, we have ${\dot R} \to 0$ as $t \to \infty$, i.e., there exists a positive time $t_* = t_*(\varepsilon, \delta)$  such that
\begin{equation} \label{NN-1}
{\dot R}(t) = KR(t) \int_{\mathbb T} \sin^2 (\theta(t) - \phi(t)) \rho(\theta, t)  ~ d\theta <  K R_0 (\sin \delta)^2 \varepsilon, \quad t > t_*.
\end{equation}
By Lemma \ref{L4.1}, we have
\begin{equation} \label{NN-2}
 R(t) \geq R_0 \quad t \geq 0. 
\end{equation} 
Then, it follows from \eqref{NN-1} and \eqref{NN-2} that
\begin{equation} \label{NN-3}
\int_{\mathbb T} \sin^2 (\theta - \phi(t)) \rho(\theta, t)  ~ d\theta \leq (\sin \delta)^2\varepsilon, \quad t > t_*.
\end{equation}
On the other hand, since  
\[ |\sin(\theta -\phi(t))| > \sin \delta \quad \forall \theta \in \mathbb T \setminus  I_\delta, \]
 the estimate \eqref{NN-3} yield
\begin{align*}
 (\sin \delta)^2  \int_{\mathbb T \setminus I_\delta} \rho(\theta, t)  ~ d\theta &< \int_{\mathbb T \setminus I_\delta} \sin^2(\theta - \phi(t)) \rho(\theta, t)  ~ d\theta \cr
 &\leq \int_\mathbb T  \sin^2(\theta - \phi(t)) \rho(\theta, t)  ~ d\theta  \leq (\sin \delta)^2 \varepsilon, \quad t \geq t_*.
\end{align*}
Thus, we obtain the desired estimate \eqref{Claim-1}. 

\subsubsection{Step B (concentration of mass in the interval $I^+_\delta$):}In this part, we exclude the possibility of  bi-polar configuration as an asymptotic profile by showing that no mass concentration occurs in $I_\delta^-$ asymptotically, i.e., we claim:
\begin{equation} \label{Claim-2}
 \lim_{t \to \infty} \int_{I_\delta^-(t)} \rho(\theta, t)  ~ d\theta = 0. 
\end{equation}
For the proof of \eqref{Claim-2}, we use Cauchy-Schwarz inequality to see
\begin{equation} \label{NN-4}
 \int_{I_\delta^-(t)} \rho(\theta, t)  ~ d\theta \leq
 \Big( \int_{I_\delta^-(t)} |\rho(\theta, t)|^2  ~ d\theta \Big)^{\frac{1}{2}}  
 \Big( \int_{I_\delta^-(t)}   ~ d\theta \Big)^{\frac{1}{2}} \leq \sqrt{2\delta}  \Big( \int_{I_\delta^-(t)} |\rho(\theta, t)|^2  ~ d\theta 
 \Big)^{\frac{1}{2}}.   
\end{equation} 
Due to the relation \eqref{NN-4}, it suffices to show that there exist a positive number $T_1$ such that 
\begin{equation} \label{NN-5}
 \int_{I_\delta^-(t) } |\rho(\theta, t)|^{2} ~ d\theta\leq e^{-R(0) (\cos \delta) K(t-T_1)}
  \int_{I_\delta^-(T_1)} |\rho(\theta, T_1)|^{2}  ~ d\theta, \qquad \forall~t>T_1.
\end{equation}
Note that for $\delta \in (0, \frac{\pi}{2})$ and $t > 0$, 
\begin{equation} \label{D-2}
  \theta \in I_\delta^-
  \quad \Longrightarrow \quad \cos(\theta- \phi) < -\cos \delta.  
\end{equation}  
On the other hand, since $\displaystyle \lim_{t \to \infty} \dot{\phi} = 0$, for any $\varepsilon \in(0,K)$, there exist $T_1 = T_1(\varepsilon, \delta) >0$ such that
\begin{equation} \label{D-3}
\mid\dot{\phi}\mid< \varepsilon R_0 \sin \delta, \qquad \forall~t>T_1.
\end{equation}
We next introduce the Lyapunov functional:
\begin{equation} \label{D-4-0}
 \Lambda(t) :=
\int_{I_\delta^-(t)}
 |\rho(\theta, t)|^2  ~ d\theta = \int_{\phi + \pi - \delta}^{\phi + \pi + \delta}  |\rho(\theta, t)|^2  ~ d\theta, 
\end{equation} 
  and we show that it satisfies a Gronwall's inequality:
\begin{equation} \label{D-4}
 \frac{d \Lambda(t)}{dt} \leq -R(0)\cos \delta K \Lambda(t), \quad t \geq T_1.
\end{equation}
For the estimate \eqref{D-4}, we use \eqref{D-0-0} and \eqref{D-4-0} to see 
\begin{align}
\begin{aligned} \label{D-5}
\frac{d \Lambda}{dt} &= \dot{\phi}\left(\rho^{2}(\phi+\pi+\delta) 
- \rho^{2}(\phi+\pi-\delta) \right)+ 2 
\int_{I_\delta^-} \rho \partial_{t} \rho  ~ d\theta \\
 &= \dot{\phi} \left(\rho^{2}(\phi+\pi+\delta) - \rho^{2}(\phi+\pi-\delta)\right) + 2R K  \int_{I_\delta^-}\rho \partial_{\theta}\bigg[\sin(\theta-\phi) \rho \bigg] ~ d\theta \\
 &= \dot{\phi} \left(\rho^{2}(\phi+\pi+\delta) - \rho^{2}(\phi+\pi-\delta) \right) - 2R K \sin \delta \rho^{2}(\phi+\pi+\delta) \\
 &- 2R K \sin \delta \rho^{2}(\phi+\pi-\delta)- R K  \int_{I_\delta^-}\sin(\theta-\phi)\partial_{\theta}(\rho^{2}) ~ d\theta \\
 &= \dot{\phi} \rho^{2}(\phi+\pi+\delta) -\dot{\phi}\rho^{2}(\phi+\pi-\delta) -R K\sin \delta \rho^{2}(\phi+\pi+\delta) \\
 &- RK \sin \delta \rho^{2}(\phi+\pi-\delta)+  R K  \int_{I_\delta^-} \cos(\theta-\phi)\rho^{2} ~ d\theta \\
 &=  \Big(\dot{\phi} -  R K \sin \delta \Big) \rho^{2}(\phi+\pi+\delta) - \Big(\dot{\phi} + R K \sin \delta \Big)\rho^{2}(\phi+\pi-\delta) \\
 &+ R K \int_{I_\delta^-}  \cos(\theta-\phi)\rho^{2} ~ d\theta.
\end{aligned}
\end{align}
On the other hand, it follows from \eqref{D-2} and \eqref{D-3} that we have
\begin{align}
\begin{aligned} \label{D-6}
& \dot{\phi} -  R K \sin \delta \leq (\varepsilon R(0) - R K) \sin \delta \leq  (\varepsilon - K) R(0)  \sin \delta < 0, \\
& \dot{\phi} + R K \sin \delta > (-\varepsilon R(0) + R K ) \sin \delta > (K- \varepsilon) R(0) \sin \delta > 0, \\
& \int_{I_\delta^-} \cos(\theta-\phi)\rho^{2} ~ d\theta \leq -(\cos \delta) \Lambda(t).  
\end{aligned}
\end{align}
We combine \eqref{D-5}, \eqref{D-6} and the fact $R(t) \geq R_0$ in Proposition \ref{P4.1} to obtain \eqref{D-4}. Thus, we have \eqref{NN-5}.  Finally, we combine \eqref{NN-4} and \eqref{NN-5} to get
\begin{align*}
  \int_{I_\delta^-(t)} \rho(\theta, t)  ~ d\theta &\leq  \sqrt{2\delta}  \Big( \int_{I_\delta^-(t)} \rho^2(\theta, t)  ~ d\theta 
 \Big)^{\frac{1}{2}} \cr
 &\leq  \sqrt{2\delta} e^{-\frac{R(0) (\cos \delta)}{2} K(t-T_1)}
 \Big( \int_{I_\delta^-(T_1)}\rho(\theta, T_1)^{2}  ~ d\theta \Big)^{\frac{1}{2}}, \qquad \forall~t>T_1.  
\end{align*} 
This yields \eqref{Claim-2} and completes the proof of Theorem \ref{T3.1}.
\begin{remark}\label{rmk: r concentration}
Note that the result in Theorem \ref{T3.1} also implies
 \[  \lim_{t \to \infty} R(t) = 1 \quad \mbox{and} \quad  \lim_{t \to \infty}V_k(\rho) = 0. \]  
\end{remark}

In the next section, we study existence of positively invariant set for the K-S equation with distributed natural frequencies from well-prepared initial data, in which some significant fraction of mass is confined and it attracts a neighboring mass.

\section{Emergence of phase concentration} \label{sec:5}
\setcounter{equation}{0}
In this section, we study emergent phenomenon of phase concentration for the K-S equation with distributed natural frequencies, i.e., the non-identical case, from well-prepared initial configurations whose significant portion of mass is already concentrated on the average phase.  As we have seen in the previous section, the analysis on the dynamics of global order parameters $R$ and $\phi$ does play a key role in the proof of the first result in Theorem \ref{T3.1}. Likewise, we will introduce local order parameters for the sub-ensemble with the same natural frequencies and study the dynamics of these local order parameters. We also discuss a possible asymptotic behavior for the K-S equation with distributed natural frequencies. Finally, we present the proof of our second result Theorem \ref{T3.2} on the emergence of arc type attractors from a well aggregated initial datum in a large coupling strength regime. Note that in the next section, such a condition on the initial datum will be removed, but we still are able to show an asymptotic pattern of the mass and the amplitude order parameter which shows asymptotic emergence of complete synchronization, namely, a point cluster, as the coupling strength tends to infinity.
\subsection{Local order parameters} \label{sec:5.1} For a finite-dimensional Kuramoto model, all oscillators with the same natural frequency will aggregate to the same phase asymptotically. Thus, it is reasonable to consider order parameters for the sub-ensemble of oscillators with the same frequency, which we call {\em local} order parameters in the sequel. For a fixed $\omega \in \mbox{supp}~g(\cdot)$, let $\varrho(\theta, \omega, t)$ be the conditional probability density function corresponding to the natural frequency $\omega$:
\begin{equation}\label{L-1}
f(\theta, \omega, t) = g(\omega) \varrho(\theta, \omega, t) \quad \mbox{and} \quad \int_{\mathbb T} \varrho(\theta, \omega, t)  ~ d\theta = 1.
\end{equation}
Then, the local order parameters are defined as follows.
\begin{definition}
Let $\varrho$ be a conditional distribution function introduced in \eqref{L-1}. Then, for a given $\omega \in \mbox{supp}~g(\cdot)$ and $t \geq 0$,  the {\em local} order parameters $R_\omega$ and $\phi_{\omega}$ are defined by the following relation: 
\begin{equation}\label{L-2}
R_{\omega}(t) e^{i\phi_{\omega}(t)} := \int_{\mathbb T}  e^{i\theta} \varrho(\theta,\omega, t) ~ d\theta.
\end{equation}
\end{definition}
Then, the local order parameters satisfy the following estimates.
\begin{lemma}\label{L5.1}
Let $(R, \phi)$ and $(R_\omega, \phi_\omega)$ be global and local order parameters defined in \eqref{O-K} and \eqref{L-2}, respectively. Then, we have 
\begin{align*}
& (i)~R_{\omega} = \int_{\mathbb T} \varrho(\theta, \omega, t) \cos(\theta - \phi_{\omega})  ~ d\theta,
\qquad  
0 = \int_{\mathbb T} \varrho(\theta, \omega, t) \sin(\theta - \phi_{\omega})  ~ d\theta. \cr
& (ii)~R = \int_{\mathbb R} g(\omega) R_\omega \cos(\phi_\omega - \phi) ~ d\omega, \qquad 0 = \int_{\mathbb R} g(\omega) R_\omega \sin(\phi_\omega - \phi) ~ d\omega.
\end{align*}
\end{lemma}
\begin{proof} (i) We divide \eqref{L-2} by $e^{i\phi_{\omega}(t)}$ to get 
\begin{equation} \label{L-3}
 R_{\omega}(t) = \int_{\mathbb T}  e^{{\mathrm i}(\theta - \phi_{\omega}(t))} \varrho(\theta,\omega, t) ~ d\theta.
\end{equation} 
We now compare the real and imaginary parts of \eqref{L-3} to get the desired estimates. \newline

\noindent (ii)~We use the defining relation \eqref{O-K} for $R$ and $\phi$ to obtain
\begin{align}
\begin{aligned}\label{G-1}
R e^{i\phi} &:= \iint_{\mathbb T \times \mathbb R} f(\theta, \omega, t) e^{i\theta} ~ d\theta d\omega = \iint_{\mathbb T \times \mathbb R} g(\omega) \varrho(\theta, \omega, t) e^{i\theta} ~ d\theta d\omega\\
&= \int_{\mathbb R} g(\omega) \int_{\mathbb T} \varrho(\theta, \omega, t) e^{i\theta} ~ d\theta d\omega = \int_{\mathbb R} g(\omega) R_\omega e^{i\phi_\omega} ~ d\omega.
\end{aligned}
\end{align}\
This again yields
\[   R =  \int_{\mathbb R} g(\omega) R_\omega e^{{\mathrm i} (\phi_\omega - \phi)} ~ d\omega. \]
We compare real and imaginary parts of the above relation to get the desired estimates.
\end{proof}
We next derive an equation for the conditional probability density function $\varrho$ from the K-S equation \eqref{K-S}. Recall that $f$ satisfies 
\begin{align}
\begin{aligned} \label{G-2}
\displaystyle & \partial_t f + \partial_{\theta} (v[f] f) = 0, \qquad (\theta, \omega) \in \mathbb T \times \mathbb R,~~t > 0, \\
& v[f](\theta, \omega, t) = \omega - K \iint_{\mathbb T \times \mathbb R} \sin(\theta-\theta_*) f(\theta_*, \omega_*, t) ~ d\theta_*d\omega_* .
\end{aligned}
\end{align}
We now substitute the ansatz \eqref{L-1} into the above equation \eqref{G-2} to derive the equation for the conditional distribution $\varrho$:
\begin{align}
\begin{aligned} \label{KM-L}
& \partial_t \varrho (\theta, \omega, t)+ \omega  \partial_\theta \varrho(\theta, \omega, t) \\
& \hspace{2cm} -K \partial_\theta \Big[ \varrho (\theta, \omega, t) \iint_{
\mathbb T \times \mathbb R } \sin(\theta - \theta_*) g(\omega_*) \varrho(\theta_*, \omega_*, t)  ~ d\theta_* d\omega_*   \Big]  = 0.
\end{aligned}
\end{align}
As noticed  in \eqref{A-1}, we have
\[
 \iint_{\mathbb T \times \mathbb R} \sin(\theta - \theta_*) g(\omega_*) \varrho (\theta_*, \omega_*, t) ~ d\theta_* d\omega_*  = R\sin (\theta - \phi). 
\]

Thus \eqref{KM-L} can be  written as 
\begin{equation}\label{E-1}
\partial_t \varrho + \partial_\theta \Big[   \big( \omega - KR \sin(\theta - \phi)\big) \varrho \Big] = 0.
\end{equation}
This equation can also be obtained directly from the equation \eqref{KM-K-I-1}.

\begin{lemma} \label{L5.2} Let $f = f(\theta, \omega, t)$ be a solution to \eqref{G-2}, and $(R_\omega, \phi_\omega)$ and $(R, \phi)$ be local and global order parameters defined by \eqref{O-K} and \eqref{L-2}, respectively. Then, we have
\begin{align*}
(i)~{\dot R}_\omega &= K R  \int_{\mathbb T}   \varrho(\theta, \omega, t)   \sin(\theta - \phi_\omega) \sin(\theta - \phi) ~ d\theta, \cr
{\dot \phi}_\omega &=  \omega -  K \frac{R}{R_\omega} \int_{\mathbb T}  \varrho(\theta, \omega, t)  \cos(\theta - \phi_\omega)\sin(\theta - \phi) ~ d\theta.    \cr
(ii)~{\dot R} &= -\iint_{\mathbb T \times \mathbb R} \sin(\theta - \phi) \omega  f(\theta, \omega, t)   ~ d\theta d\omega + K R  \int_\mathbb T \sin^2 (\theta - \phi)   \rho(\theta, t) ~ d\theta. \cr
\dot \phi &=  \frac{1}{R} \iint_{\mathbb T \times \mathbb R} \cos (\theta - \phi) \omega  f(\theta, \omega, t) ~ d\theta d\omega - \frac{K}{2} \int_\mathbb T \sin\left(2(\theta - \phi)\right)  \rho(\theta, t) ~ d\theta . 
\end{align*}
\end{lemma}
\begin{proof} The estimates follow from the differentiation of the defining relations for order parameters and using the K-S equations \eqref{E-1} and \eqref{G-2} for $\rho$ and $f$. \newline

\noindent (i) We differentiate \eqref{L-2} with respect to $t$ and use Lemma \ref{L5.1} to obtain
\begin{align*} 
{\dot R}_\omega &= \int_{\mathbb T} \cos(\theta - \phi_\omega) \partial_t \varrho (\theta, \omega, t)  ~ d\theta\cr
&= - \int_{\mathbb T} \cos(\theta - \phi_\omega) \partial_\theta \Big[  \varrho(\theta, \omega, t) \big( \omega - KR \sin(\theta - \phi)\big) \Big]  ~ d\theta \cr
&=  - \int_{\mathbb T} \sin(\theta - \phi_\omega)   \varrho(\theta, \omega, t) \big( \omega - KR \sin(\theta - \phi)\big)  ~ d\theta \cr
&= KR  \int_{\mathbb T}   \varrho(\theta, \omega, t)   \sin(\theta - \phi_\omega) \sin(\theta - \phi)  ~ 
d\theta
\end{align*}
and
\begin{align*}
\dot \phi_\omega &= \frac{1}{R_\omega} \int_{\mathbb T} \sin(\theta - \phi_\omega) \partial_t \varrho(\theta, \omega, t)  ~ d\theta \\
&= - \frac{1}{R_\omega} \int_{\mathbb T} \sin(\theta - \phi_\omega) \partial_\theta \Big[  \varrho(\theta, \omega, t) \big( \omega - KR \sin(\theta - \phi) \big) \Big]  ~ d\theta \\
&=   \frac{1}{R_\omega} \int_{\mathbb T} \cos(\theta - \phi_\omega)    \varrho(\theta, \omega, t) \big( \omega - KR \sin(\theta - \phi) \big)   ~ d\theta \\
&= \omega -  K \frac{R}{R_\omega} \int_{\mathbb T}    \varrho(\theta, \omega, t)  \cos(\theta - \phi_\omega)\sin(\theta - \phi)  ~ d\theta.
\end{align*}

\vspace{0.2cm}

\noindent (ii) For the estimates on the global order parameters \eqref{G-1}, we perform similar computation as (i) to obtain desired estimates. 
\end{proof}
\begin{remark}
The dynamics of order parameters $(R_\omega, \phi_\omega)$ and $(R, \phi)$ coincide with the dynamics \eqref{B-5} of corresponding order parameters for the Kuramoto model \eqref{KM}.
\end{remark}

\subsection{Nonexistence of point attractors} In Section \ref{sec:4}, we have shown that point attractors can emerge from generic smooth initial data in a positive coupling strength regime for the identical natural frequency case. In this subsection, we will show that emergence of point attractors will not be possible in a general setting. Without loss of generality, we assume that average natural frequencies $\omega_c = \int_{\mathbb R} \omega g(\omega) d\omega$ is zero, otherwise, we can consider the rotating frame moving with $\omega_c$.  \newline

Suppose that $f^{\infty}$ is an equilibrium for the K-S equation \eqref{E-1}, whose  conditional probability density function $\varrho^{\infty}(\theta, \omega, t)$ is in the form $\varrho^{\infty}(\theta, \omega, t) \equiv  \delta_{\phi_\omega}$ for each $\omega \in \mbox{supp} \,g$, i.e., 
\begin{equation} \label{LSS}
 f^{\infty}(\theta, \omega, t) = g(\omega)\varrho^{\infty}(\theta, \omega, t) = g(\omega)  \delta_{\phi_\omega}.     
\end{equation} 
We call such  $f^{\infty}$ as a {\em locally synchronized state}, i.e., a locally synchronized state $f^{\infty}$ is a complete phase synchronization for a sub-ensemble with the same given frequency $\omega$. A complete phase synchronization is obviously of such a form. To distinguish these locally synchronized states, we use the notation $(R^{\infty}, \phi^{\infty})$ and $(R^{\infty}_\omega, \phi^{\infty}_\omega)$ for their global and local order parameters, respectively. 

Note that  for a locally synchronized state $f^{\infty}$ in \eqref{LSS},  it follows from Lemma \ref{L5.1} that for each $\omega \in \mbox{supp} \, g(\omega)$, 
\begin{equation}\label{r Omega 1}
 R^{\infty}_\omega = \int_\mathbb T \delta_{\phi^{\infty}_\omega} \cos (\theta- \phi^{\infty}_\omega) d \theta = \cos (
  \phi^{\infty}_\omega- \phi^{\infty}_\omega) = 1. 
\end{equation}

This and Lemma \ref{L5.2} imply  that for all $\omega \in \mbox{supp} ~ g(\omega)$,
\begin{align}
\begin{aligned}\label{E-7}
\dot{\phi}^{\infty}_\omega & =  \omega -  K \frac{R^{\infty}}{R^{\infty}_\omega} \int_{\mathbb T}  \varrho^{\infty} (\theta, \omega, t)  \cos(\theta - \phi^{\infty}_\omega)\sin(\theta - \phi^{\infty})  ~ d\theta \\
&=  \omega -  K R^{\infty} \int_{\mathbb T}  \delta_{\phi^{\infty}_\omega}  \cos(\theta - \phi^{\infty}_\omega)\sin(\theta - \phi^{\infty})  ~ d\theta \\
&= \omega - K R^{\infty} \sin(\phi^{\infty}_\omega - \phi^{\infty})\\
&=0,
\end{aligned}
\end{align}
where the last line is due to equilibrium state ${\dot \phi}_\omega^{\infty} = 0$. Thus, for all $\omega \in \mbox{supp} ~ g(\omega)$, we have
\begin{equation}\label{phi arcsin} 
\sin(\phi^{\infty}_\omega - \phi^{\infty}) = \frac{\omega}{K R^{\infty}}, \quad \mbox{i.e.,} \quad \phi^{\infty}_\omega - \phi^{\infty} = \arcsin \frac{\omega}{K R^{\infty}}.
\end{equation}
On the other hand, we use Lemma \ref{L5.1} and \eqref{r Omega 1} to see
\begin{equation}\label{condition r}
R^{\infty} = \int_{\mathbb R} g(\omega) R^{\infty}_\omega \cos(\phi^{\infty}_\omega - \phi^{\infty}) ~ d\omega = \int_{\mathbb R} g(\omega) \sqrt{1- \left(\frac{\omega}{K R^{\infty}}\right)^2} ~ d\omega. 
\end{equation}
Note that the condition in (ii) Lemma \ref{L5.1} is automatically satisfied from \eqref{r Omega 1} and \eqref{phi arcsin}: 
\[
 \int_{\mathbb R} g(\omega) R^{\infty}_\omega \sin(\phi^{\infty}_\omega - \phi^{\infty}) ~ d\omega =\frac{1}{K R^{\infty}} \int_\mathbb R \omega g(\omega) ~ d\omega = 0,
\]
where we used our assumption $\int_\mathbb R \omega g(\omega) ~ d\omega =0$. \newline

In summary, we have
\begin{align*}
&  f^{\infty}  = g(\omega) \delta_{\phi^{\infty}_\omega}~~\mbox{is an equilibrium} \cr
& \hspace{1cm} \iff \quad R^{\infty}  = \int_{\mathbb R} g(\omega) \sqrt{1- \left(\frac{\omega}{K R^{\infty}}\right)^2} ~ d\omega \quad \mbox{and} \quad \sin(\phi^{\infty}_\omega - \phi^{\infty}) = \frac{\omega}{K R^{\infty}}.
\end{align*}

\begin{proposition}\label{P5.1}
Suppose that the coupling strength and $g$ satisfy
\[ K > 0, \quad  \int_{\mathbb R} g(\omega) ~ d\omega  =1, \quad \int_{\mathbb R} \omega g(\omega) ~ d\omega  = 0, \quad g \not = \delta. \]
Then the K-S equation \eqref{G-2} may not have a complete phase synchronization.
\end{proposition}
\begin{proof} Suppose that the complete phase synchronization occurs, i.e., there exists an equilibrium $f^{\infty}$ which corresponds to $R^{\infty} = 1$. Then, the relation \eqref{condition r} yields
\begin{equation} \label{E-8}
 1= \int_{\mathbb R} g(\omega) \sqrt{1- \left(\frac{\omega}{K}\right)^2} ~ d\omega.     
\end{equation} 
However, there exist $g$ and $K$ such that the above relation does not hold. For example, we set
\[ g(\omega) = \frac{1}{2} {\mathbbm 1}_{[-1,1]}, \qquad K = 1. \]
then, the L.H.S. of \eqref{E-8} satisfies
\[ \label{E-8-1}
 \int_{\mathbb R} g(\omega) \sqrt{1- \left(\frac{\omega}{K}\right)^2} ~ d\omega = \int_{0}^1 \sqrt{1-\omega^2} ~ d\omega = \arctan \omega \Big|_{\omega = 0}^{\omega = 1} = \frac{\pi}{4} \not = 1.
\]
This is contradictory to the relation \eqref{E-8}. This shows that the complete phase synchronization may not occur.
\end{proof}

It follows from \eqref{r Omega 1}, \eqref{E-7} and \eqref{condition r} that for such a locally synchronized state we have 
\[ \dot{R}^{\infty} =0, \qquad \dot{\phi}^{\infty} = 0 \qquad \mbox{and} \qquad  R^{\infty} \to 1 \quad \mbox{as $K \to \infty$}.  \]
Thus, we can expect that for a large $K$,  equilibrium states are close to a complete phase synchronization. In the following proposition, we give a more quantified version of this. 

\begin{proposition} \label{P5.2}
Suppose that the probability density function $g = g(\omega)$ satisfies
\begin{equation} \label{E-9}
 [-m, m]\subset \mbox{\rm supp} ~  g(\omega) \subset [-M, M], \qquad \int_{\mathbb R} \omega g(\omega) ~ d\omega = 0, 
\end{equation} 
and let $f^{\infty}$ be an equilibrium to \eqref{K-S}. Then, we have 
\[  R^{\infty}  \geq  \sqrt{1- \left(\frac{M}{K R^{\infty}}\right)^2}, \qquad  R^{\infty} \geq     m \Big( \min_{\omega \in [-m, m]} g(\omega) \Big). \]
\end{proposition}
\begin{proof} (i)~It follows from \eqref{condition r} that we have
\begin{align*}
R^{\infty} &= \int_{\text{supp}\, g(\omega)} g(\omega) \sqrt{1- \left(\frac{\omega}{K R^{\infty}}\right)^2} ~ d\omega \\
&\geq \sqrt{1- \left(\frac{M}{K R^{\infty}}\right)^2} \int_{\text{supp}\, g(\omega)} g(\omega) ~ d\omega = \sqrt{1- \left(\frac{M}{K R^{\infty}}\right)^2}.
\end{align*}
(ii)~For $K \geq m$, we use \eqref{condition r}  and \eqref{E-9} to obtain
\begin{align}
\begin{aligned}\label{E-10}
R^{\infty} &\geq \int_{-m}^m g(\omega)\sqrt{1 - \left( \frac{\omega}{KR^{\infty}} \right)^2} ~ d\omega\\
&\geq \Big( \min_{\omega \in [-m, m]} g(\omega) \Big) \int_{-m}^m \sqrt{1 - \left( \frac{\omega}{K R^{\infty}} \right)^2} ~ d\omega\\
& = \Big( \min_{\omega \in [-m, m]} g(\omega) \Big) \Big( m \sqrt{1 - \left(\frac{m}{K R^{\infty}} \right)^2} + K R^{\infty} \arcsin \frac{m}{K R^{\infty}} \Big) \\
&\geq \Big( \min_{\omega \in [-m, m]} g(\omega) \Big) K R^{\infty} \arcsin \frac{m}{K R^{\infty}} \\
&\geq  m \Big( \min_{\omega \in [-m, m]} g(\omega) \Big). 
\end{aligned}
\end{align}
\end{proof}
\begin{remark} For $g(\omega) = \frac{1}{2 \ell} {\mathbbm 1}_{[-\ell, \ell]}$, if we choose
\[ m = M = \ell, \]
then we have $R^{\infty} \geq \frac{1}{2}.$
\end{remark}

\bigskip

\subsection{Proof of Theorem \ref{T3.2}}
As noted in Proposition \ref{P5.1}, a complete phase synchronization may not occur for the distributed natural frequencies and a complete phase synchronization can be regraded as a concentration phenomenon where full mass concentrates at a single point. Thus, it is still interesting to see  \newline
\begin{quote}
Under what conditions on parameters and initial data, when does a concentration around the average phase emerge? 
\end{quote}
This question will be addressed in the sequel. \newline

For $t \geq 0$, we consider the following time-dependent interval $L^+_\gamma$ (see Figure \ref{fig:Lgamma}) and mass on $S \subset \mathbb T$: 
\[
L^+_\gamma(t):= \big( \phi(t) - \frac{\pi}{2} + \gamma,  \phi(t) + \frac{\pi}{2} - \gamma\big), \qquad {\mathcal M}(S) : = \int_{\mathbb R} \int_S f(\theta, \omega, t)  ~ d\theta d\omega,
\]
where the constant $\gamma$ is to be determined later, and we assume that $g =g(\omega)$ is compactedly supported and 
\[ \mbox{\rm supp} ~  g(\omega) \subset (-M, M). \]
Note that the length of the time-dependent interval $L^+_\gamma(t)$ equals to $\pi - 2\gamma$ and $L_{\gamma}^+ = I^+_{\frac{\pi}{2} -\gamma}$.

Then, under an appropriate assumption on the initial configuration, we will show the following  two properties: for any solution $f = f(\theta, \omega, t)$ to \eqref{K-S},
\begin{equation} \label{E-10-0}
\frac{d}{dt}  {\mathcal M}(L^+_{\gamma(t)}) \geq 0 \quad \mbox{and} \quad  \lim_{t \to \infty} \int_{L^+_\gamma (t)} |f(\cdot, \omega, t)|^2 ~ d\theta = \infty, \quad \mbox{for each $\omega \in \mathbb R$},
\end{equation}

\begin{figure}
\centering
\subfigure[$L_\gamma^+(t)$]{\includegraphics[width=0.4\textwidth]{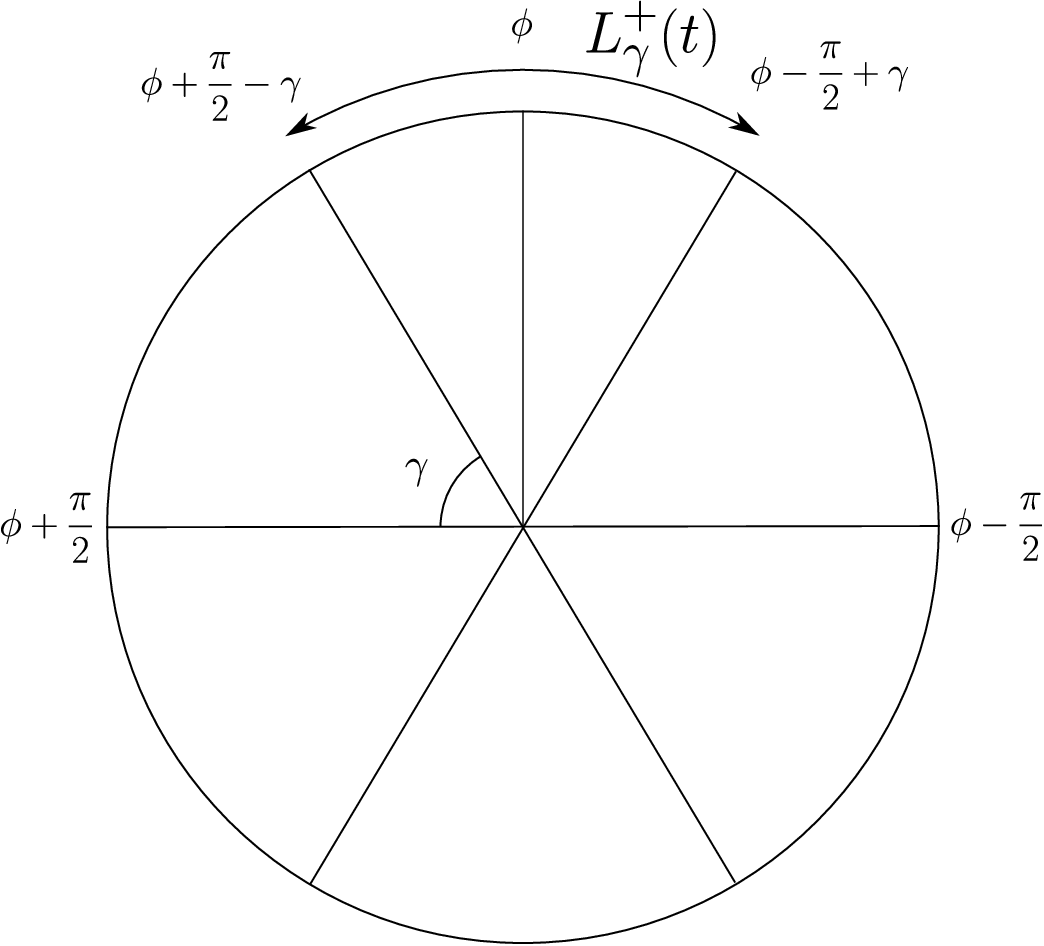}\label{fig:Lplus}}
\hspace{0.1\textwidth}
\subfigure[$L_\gamma^-(t)$]{\includegraphics[width=0.4\textwidth]{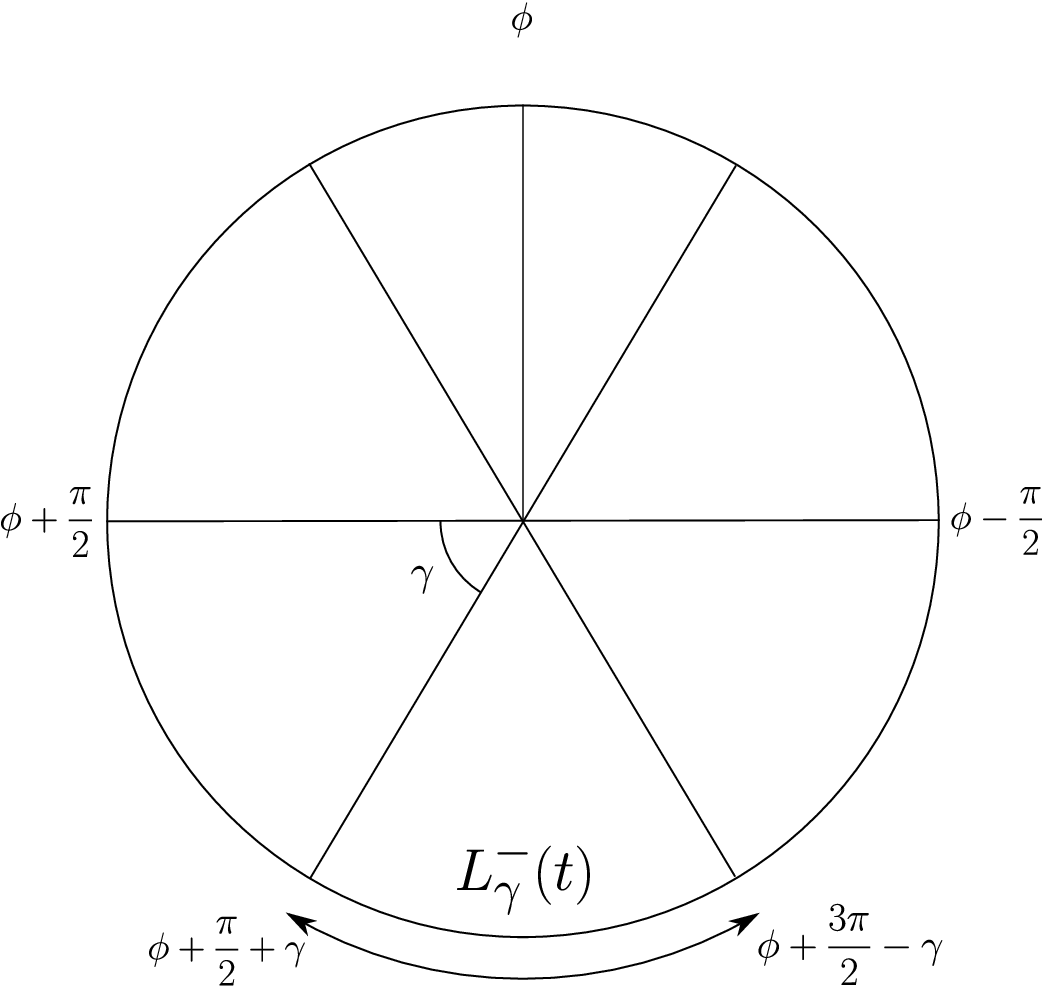}\label{fig:Lminus}}
\caption{Geometric descriptions of $L_\gamma^+(t)$ and $L_\gamma^-(t)$}
\label{fig:Lgamma}
\end{figure}

\subsubsection{Verification of the first estimate in \eqref{E-10-0}} Before we present the proof of Theorem \ref{T3.2}, we first establish several lemmata in the sequel.

We first study the bounds of ${\dot \phi}$ and $R$. 
\begin{lemma} \label{L5.3}
Let $f$ be a solution to \eqref{K-S}. Then, the order parameters $R$ and $\phi$ satisfy 
\begin{align}
\begin{aligned} \label{dot phi r}
& (i)~|\dot\phi | \leq \frac{M}{R} + K(1-R). \\
& (ii)~\max \{ 0, (1 + \sin \gamma ) {\mathcal M}(L^+_{\gamma}) - 1 \} \leq R \leq \min\{1, 
(1-\sin\gamma) {\mathcal M} \big(L^+_\gamma \big) + \sin \gamma\}. 
\end{aligned}
\end{align}
\end{lemma}
\begin{proof} (i) We use \eqref{L-1} and Lemma \ref{L5.2} to obtain
\begin{align}\label{E-10-1}
\begin{aligned}
{\dot \phi} &= \frac{1}{R} \int_{\mathbb R} \omega g(\omega) \int_{\mathbb T} \varrho(\theta, \omega, t) \cos(\theta -\phi)  ~ d\theta d\omega -\frac{1}{2} \int_{\mathbb R} Kg(\omega) \int_{\mathbb T} \varrho(\theta, \omega, t) \sin 2(\theta - \phi)  ~ d\theta d\omega \cr
& =: {\mathcal I}_{21} + {\mathcal I}_{22}.
\end{aligned}
\end{align}

$\bullet$ (Estimate of $ {\mathcal I}_{21}$): We use the fact $\text{supp}~g(\omega) \subset [-M , M]$ to see
\begin{equation} \label{E-10-2}
|{\mathcal I}_{21}|  \leq  \frac{1}{R} \int_{\mathbb R}  |\omega| g(\omega) \int_{\mathbb T} \varrho(\theta, \omega, t)   ~ d\theta d\omega \leq  \frac{1}{R} \int_{-M}^M  |\omega| g(\omega) ~ d\omega \leq \frac{M}{R}.
\end{equation}
$\bullet$ (Estimate of $ {\mathcal I}_{22}$): We use the same argument as in the proof 
of Proposition \ref{P4.1} to get
\begin{equation} \label{E-10-3}
|{\mathcal I}_{22}| \leq K(1-R). 
\end{equation}

Finally, in \eqref{E-10-1}, we combine \eqref{E-10-2} and \eqref{E-10-3} to obtain the desired estimate:
\[
|\dot\phi | \leq \frac{M}{R} + K(1-R).
\]
(ii)~For the lower bound estimate, we use a defining relation \eqref{O-K} for $R$ to obtain
\begin{align*}
R(t)&= \iint_{\mathbb T \times \mathbb R} f(\theta, \omega, t) \cos(\theta - \phi) ~ d\theta d\omega\\
&=\int_{\mathbb R}\int_{L^+_\gamma(t)} f(\theta, \omega, t) \cos(\theta - \phi) ~ d\theta d\omega  + \int_{\mathbb R} \int_{\mathbb T \setminus L^+_\gamma(t)} f(\theta, \omega, t) \cos(\theta - \phi) ~ d\theta d\omega\\
&\geq \sin\gamma \int_{\mathbb R} \int_{L^+_\gamma(t)} f(\theta, \omega, t) ~ d\theta d\omega - \int_{\mathbb R}\int_{\mathbb T \setminus L^+_\gamma(t)} f(\theta, \omega, t) ~ d\theta d\omega\\
&= \sin\gamma {\mathcal M}(L^+_\gamma(t)) - \Big(1 - {\mathcal M}(L^+_\gamma(t))\Big)\\
& = (1+\sin\gamma){\mathcal M}(L^+_\gamma(t)) - 1. 
\end{align*}
On the other hand, for the upper bound estimate, we use 
\[ R \leq \iint_{L^+_\gamma \times \mathbb R} f ~ d\theta d\omega + \sin \gamma \iint_{ (\mathbb T \backslash L^+_{\gamma})\times \mathbb R} f ~ d\theta d\omega \leq (1-\sin \gamma ) {\mathcal M}(L^+_\gamma) + \sin \gamma \]
to obtain the desired upper bound for $R$. 
\end{proof}

We now find appropriate constants $\varepsilon_0$ and $\gamma_0$ for the initial condition.

\begin{lemma}\label{L5.4} Suppose $\varepsilon_0$ and $\gamma_0$ are positive constants satisfying
\begin{equation} \label{gamma}
0<\varepsilon_0<\frac{3\sqrt{3}}{4}-1, \qquad 
\frac{\pi}{3} \leq \gamma_0 < \arcsin\Big(1 - \frac{2\varepsilon_0}{2\sqrt{3} + 1} \Big).
\end{equation}
Then, we have
\begin{equation} \label{E-12}
\frac{1+\varepsilon_0}{1+\sin\gamma_0} < {\mathcal M}_*(\varepsilon_0, \gamma_0) < 1,
\end{equation}
where ${\mathcal M}_*(\varepsilon_0, \gamma_0)$ is a positive constant defined by \eqref{MM}.
\end{lemma}
\begin{proof} 
\begin{figure}
\centering
\includegraphics[width=0.5\textwidth]{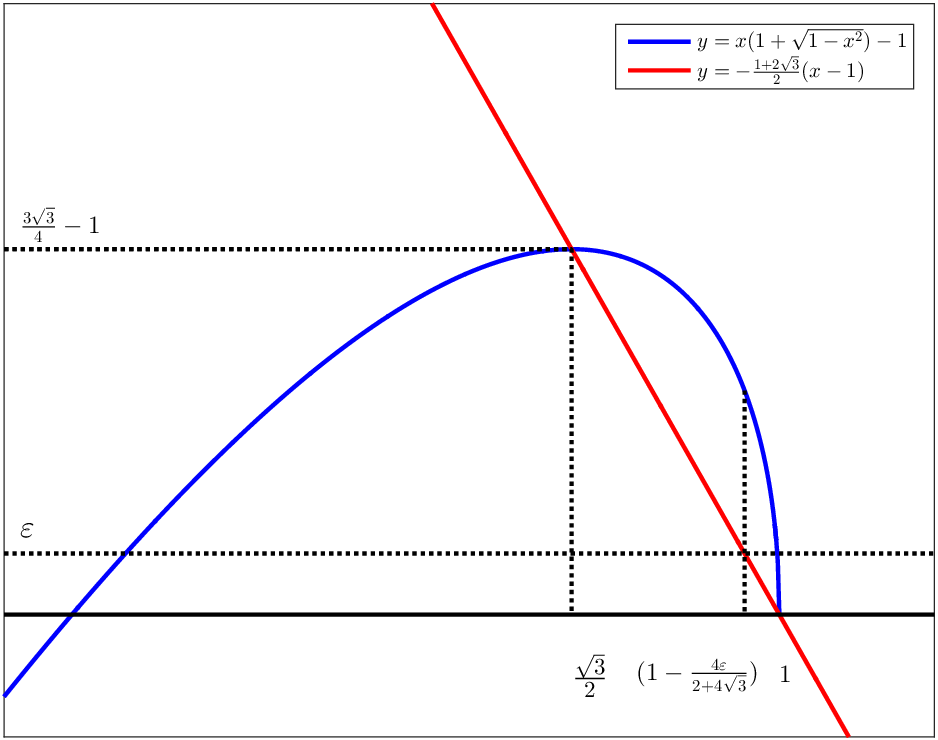}
\caption{$x(1 + \sqrt{1-x^2}) -1 > -\frac{1+2\sqrt{3}}{2}(x-1)$}
\label{fig:compare}
\end{figure}
(i) (First inequality): Since $1+\varepsilon_0 > \varepsilon_0(1+\cos\gamma_0)$, we have
\[ 2+\varepsilon_0 + \cos\gamma_0 = 1 + \cos \gamma_0 + 1 + \varepsilon_0  > (1 + \varepsilon_0)(1+\cos\gamma_0) \]
This yields the first inequality.
\vspace{0.2cm}

\noindent (ii) (Second inequality): For $\frac{\sqrt{3}}{2} < x < 1$, we have the following inequality (see Figure \ref{fig:compare}):
\[  \label{ex}
x(1 + \sqrt{1-x^2}) -1 > -\frac{1+2\sqrt{3}}{2}(x-1).
\]
Thus, we have
\begin{equation}\label{ex2}
x(1+\sqrt{1-x^2}) - 1 > \varepsilon_0, \quad \mbox{for} \quad \frac{\sqrt{3}}{2}< x < 1- \frac{2\varepsilon_0}{1+ 2\sqrt{3}}. 
\end{equation}
On the other hand, by the assumption \eqref{gamma}, 
\begin{equation*}
\frac{\sqrt{3}}{2} < \sin\gamma_0 < 1- \frac{2\varepsilon_0}{2\sqrt{3} + 1}
\end{equation*}
We use \eqref{ex2} to obtain
\[
\sin\gamma_0 ( 1 + \cos\gamma_0) > 1+ \varepsilon_0, \quad \mbox{equivalently} \quad 
(1+\sin\gamma_0)(1+\cos\gamma_0)> 2+\varepsilon_0 + \cos\gamma_0,
\]
which implies \eqref{E-12} for $\gamma_0> \frac{\pi}{3} $. The case $\gamma_0= \frac{\pi}{3}$
follows from the fact that 
\[
{\mathcal M}_*(\varepsilon_0, \gamma_0) =1,
\]
when $\gamma_0= \frac{\pi}{3} $ and  $\varepsilon_0=\frac{3\sqrt{3}}{4}-1.$
\end{proof}

We now see how the previoiusly chosen $\varepsilon_0$ and $\gamma_0$ are used to give an appropriate initial configuration.

\begin{lemma} \label{L5.5}
Suppose that the initial datum $f_0$ satisfies
\begin{align*}
& (i)~f_0(\theta,\omega) = 0  \hspace{1em}{\rm{in}}\hspace{1em}
\mathbb{T}\times (\mathbb{R}\backslash[-M,M]). \cr
\vspace{0.5cm}
& (ii)~\inf_{\omega \in  \mbox{supp}~g} \int_{L^+_{\gamma_0} (0)} f_0(\theta, \omega) ~ d\theta \geq {\mathcal M}_*(\varepsilon_0, \gamma_0),
\end{align*}
where $\varepsilon_0$ and $\gamma_0$ are positive constants as in Lemma \ref{L5.4}. Then we have 
\[
\sup_{\omega \in \mbox{supp}~g}|\phi(0) - \phi_{\omega}(0) |<\frac{\pi}{2} - \frac{1+\varepsilon_0}{1+\cos\gamma_0}.
\]
\end{lemma}
\begin{proof}
By definition of the local order parameter \eqref{L-2} and \eqref{MM}, we have
\begin{align*}
& R_{\omega}(0) \cos(\phi_{\omega}(0) - \phi(0)) \\
& \hspace{1cm} = \int_{L^+_{\gamma_0}(0)} \cos(\theta-\phi(0)) f_0(\theta, \omega)  ~ d\theta + \int_{\mathbb T \setminus L^+_{\gamma_0}(0)} \cos(\theta-\phi(0)) f_0(\theta, \omega)  ~ d\theta \\
& \hspace{1cm} \geq \cos (\frac{\pi}{2}-\gamma_0) \int_{L^+_{\gamma_0}(0)} f_0(\theta, \omega)  ~ d\theta - \int_{\mathbb T \setminus L^+_{\gamma_0}(0)} f_0(\theta, \omega)  ~ d\theta \\
& \hspace{1cm} \geq \sin\gamma_0 {\mathcal M_*}(\varepsilon_0, \gamma_0) -\Big(1- {\mathcal M_*}(\varepsilon_0, \gamma_0) \Big) \\
&\hspace{1cm} =\frac{2+\varepsilon_0 + \cos\gamma_0}{1+\cos\gamma_0} - 1 = \frac{1+\varepsilon_0}{1+\cos\gamma_0} >0.
\end{align*}
Since $R_{\omega}(0) \leq1$,
\[
\cos(\phi_{\omega} - \phi) > \frac{1+\varepsilon_0}{1+\cos\gamma_0} > \sin \Big(\frac{1+\varepsilon_0}{1+\cos\gamma_0} \Big) = \cos \Big(\frac{\pi}{2} - \frac{1+\varepsilon_0}{1+\cos\gamma_0} \Big).
\]
Therefore, we have
\[ |\phi(0) - \phi_{\omega}(0)| < \frac{\pi}{2} - \frac{1+\varepsilon_0}{1+\cos\gamma_0}. \]
\end{proof}

We are now ready to prove the first part of Theorem \ref{T3.2}.  Let $\gamma_0$ and $\varepsilon_0$ be positive constants satisfying the relations \eqref{gamma}, and suppose that $K, g =g(\omega)$ and the initial datum satisfy
\begin{align}
\begin{aligned} \label{ASP-1}
& (i)~ \mbox{supp}~g(\omega) \subset [-M, M], \qquad K>\frac{M}{\varepsilon_0}\bigg(1+\frac{1}{\varepsilon_0}\bigg), \\
& (ii)~ ||f_0||_{L^{\infty}} < \infty, \qquad  {\mathcal M}(L^+_{\gamma_0}(0)) \geq {\mathcal M}_*(\varepsilon_0, \gamma_0).
\end{aligned}
\end{align}
Then, for any classical solution $f$ to \eqref{K-S} we will show that
\[ \frac{d}{dt} \mathcal{M}(L^+_{\gamma_0} (t)) \geq 0.  \]
Since the proof is rather lengthy, we split it into several steps. We first note that 
\[ f(\theta, \omega,t) = 0, \quad \omega \not \in [-M, M]. \]

\vspace{0.5cm}

\noindent $\bullet$ Step A: When $R(t)>0$, we first establish 
\begin{align}
\begin{aligned} \label{H-1}
& \frac{d}{dt}\mathcal{M}(L^+_{\gamma_0}(t)) \\
& \hspace{1cm} \geq \Big [ K\big(R(1+\cos\gamma_0)-1\big)-M\left(1+\frac{1}{R}
 \right) \Big]  \int_{-M}^{M} \big|B_{-,\omega}(t)+B_{+,\omega}(t)
 \big| ~ d\omega, 
\end{aligned}
\end{align} 
where $B_{-,\omega}$ and $B_{+,\omega}$ denote the boundary values:
\begin{align}\label{H-1-1}
B_{-,\omega}(t):=f(\phi(t)-\frac{\pi}{2}+\gamma_0,\omega,t)\quad\mbox{and}\quad B_{+,\omega}(t):=f(\phi(t)+\frac{\pi}{2}-\gamma_0,\omega,t).	
\end{align}
For the estimate \eqref{H-1}, we use straightforward calculation to see
\begin{align}\label{dL}
\begin{aligned}
& \frac{d}{dt}\mathcal{M}(L^+_{\gamma_0}(t)) \\
& \hspace{0.5cm} = \frac{d}{dt} \int_{-M}^{M} \int_{\phi(t) - \frac{\pi}{2} + \gamma_0}^{\phi(t) + \frac{\pi}{2} - \gamma_0} f(\theta, \omega, t)  ~ d\theta d\omega \\
&  \hspace{0.5cm} =\dot{\phi}(t)\int_{-M}^{M}\big[f(\phi+\frac{\pi}{2}-\gamma_0,t)-
f(\phi-\frac{\pi}{2}+\gamma_0,t)\big] ~ d\omega+
\int_{-M}^{M} \int_{L_{\gamma_0}(t)}\partial_{t}f ~ d\theta d\omega\\
&  \hspace{0.5cm} =\dot{\phi}(t)\int\big[B_{+,\omega}(t)-B_{-,\omega}(t)
\big] ~ d\omega \\
& \hspace{0.5cm} -\int_{-M}^{M}\int_{L^+_{\gamma_0}(t)}\partial_{\theta}
\Big[f(\theta,\omega,t)\big(\omega-KR\sin(\theta-\phi)
\big)\Big] ~ d\theta d\omega\\
& \hspace{0.5cm} =\dot{\phi}(t)
\int_{-M}^{M}\big[B_{+,\omega}(t)-B_{-,\omega}(t)\big] ~ d\omega\\
& \hspace{0.5cm}+\int_{-M}^{M}\big[-B_{+,\omega}(t)\big(\omega-KR\sin(\frac{\pi}{2}-
\gamma_0)\big)+B_{-,\omega}(t)\big(\omega-KR\sin(-\frac{\pi}{2}+
\gamma_0)\big)\big] ~ d\omega\\
\end{aligned}
\end{align}
By rearranging the terms, we have
\begin{align*}
&\frac{d}{dt}\mathcal{M}(L^+_{\gamma_0}(t)) \\
& \hspace{0.5cm}=\dot{\phi}(t)\int_{-M}^{M} \big[B_{+,\omega}(t)-
B_{-,\omega}(t)\big] ~d\omega\\
& \hspace{0.5cm}+\int_{-M}^{M}\big[-B_{+,\omega}(t)
\big(\omega-KR\cos\gamma_0\big)+B_{-,\omega}(t)\big(\omega+
KR\cos\gamma_0\big)\big] ~ d\omega\\
&  \hspace{0.5cm} =KR\cos\gamma_0\int_{-M}^{M}\big[B_{-,\omega}(t)+B_{+,\omega}(t)\big] ~ d\omega+\int_{-M}^{M}\big(\dot{\phi}(t)-\omega\big)\big[B_{+,\omega}(t)-B_{-,\omega}(t)\big] ~ d\omega.
\end{align*}
In \eqref{dL}, we use \eqref{dot phi r} in Lemma \ref{L5.3} to obtain \eqref{H-1}:
\begin{align*}
&\frac{d}{dt}\mathcal{M}(L^+_{\gamma_0}(t)) \\
& \hspace{0.5cm} \geq\int_{-M}^{M} \big(KR\cos\gamma_0-|\dot{\phi}(t)-\omega|\big)\big|B_{-,\omega}(t)+B_{+,\omega}(t)\big| ~ d\omega\\
& \hspace{0.5cm} \geq\left(KR\cos\gamma_0-M\left(1+\frac{1}{R}\right)-K(1-R)\right)\int_{-M}^{M}\big|B_{-,\omega}(t)+B_{+,\omega}(t)\big| ~ d\omega\\
&\hspace{0.5cm} =\left(K\big(R(1+\cos\gamma_0)-1\big)-M\left(1+\frac{1}{R}
 \right)\right)\int_{-M}^{M}\big|B_{-,\omega}(t)+B_{+,\omega}(t)
 \big| ~ d\omega.
\end{align*}
Now observe that 
\begin{align*}
& K\big(R(1+\cos\gamma_0)-1\big)-M\left(1+\frac{1}{R} \right) \\
 &\geq K\Big\{\Big((1+\sin\gamma_0)\mathcal{M}
\big(L^+_{\gamma_0}(t)\big)-1\Big)(1+\cos\gamma_0)-1\Big\} \\
& -M\Big(1+\frac{1}{(1+\sin\gamma_0)\mathcal{M}\big(L^+_{\gamma_0}(t)
\big)-1}\Big)\\
&=: \Delta(t).
\end{align*}
In the next steps, we will show that $\Delta \geq 0$. 

\vspace{0.5cm}

\noindent $\bullet$ Step B: Due to the assumption \eqref{ASP-1} (i), we can choose sufficiently small $\eta >0$ satisfying 
\begin{equation}\label{econ2}
 K\Big(\varepsilon_0 -\eta(1+\sin\gamma_0)(1+\cos\gamma_0)\Big)-M\Big(1+\frac{1}{\varepsilon_0-\eta(1+\sin\gamma_0)}\Big)>0. 
\end{equation}
Note that for $\eta = 0$, the relation \eqref{econ2} reduces to 
\[ K \varepsilon_0 -M\Big(1+\frac{1}{\varepsilon_0}\Big) = \varepsilon_0 \Big[ K - \frac{M}{\varepsilon_0}  \Big(1+\frac{1}{\varepsilon_0}\Big)     \Big] > 0 \quad \mbox{by assumption (i) in \eqref{ASP-1}}.   \]
Thus, such $\eta$ satisfying \eqref{econ2} exists. 

\vspace{0.5cm}

\noindent $\bullet$ Step C: We claim that for $t \in [0, \infty)$, 
\begin{align}
\begin{aligned} \label{claim-1}
& \Delta(t) \geq K\Big(\varepsilon_0-\eta(1+\sin\gamma_0)(1+\cos\gamma_0)\Big)-M\Big(1+\frac{1}{\varepsilon_0-\eta(1+\sin\gamma_0)}\Big) \\
& \quad \mbox{and} \quad R(t) \geq \varepsilon_0.
\end{aligned}
\end{align}
{\it The proof of claim \eqref{claim-1}}: we now define a set  ${\mathcal T}_{\eta}$ and its supremum $T^*_\eta := \sup \mathcal{T_{\eta}}$:
\[
{\mathcal T}_{\eta} :=\bigg\{T \in [0, \infty): 
\mathcal{M}\big(L^+_{\gamma_0}(t)\big) > {\mathcal M}_*(\varepsilon_0, \gamma_0) -\eta 
\quad \mbox{for all} \quad  t \in [0, T) \bigg\}.
\]
It follows from \eqref{dot phi r}, \eqref{econ2}, and definition of $\mathcal{T}_{\eta}$ that for $t \in [0, T_\eta^*),$
\begin{equation}\label{rloe}
R(t) \geq (1+\sin\gamma_0){\mathcal M}\big(L^+_{\gamma_0}(t)\big) - 1 \geq \frac{1+\varepsilon_0}{1+\cos\gamma_0}- \eta(1+\sin\gamma_0) >0.
\end{equation}
where we used $\eta \ll 1$. By the assumption (ii) in \eqref{ASP-1} and Lipschitz continuity of $\mathcal{M}\big(L^+_{\gamma_0}(t)\big)$ in $t$ (see Appendix \ref{App-B}), we can see that the set ${\mathcal T}_{\eta}$ is nonempty, hence $T^*_\eta \in (0, \infty]$.
Suppose that $T^*_{\eta}<\infty.$ Then, we have
\begin{equation} \label{H-2}
\lim_{ t \to T^*_{\eta}-} \mathcal{M}\big(L^+_{\gamma_0}(t)\big) = {\mathcal M}_*(\varepsilon_0, \gamma_0) -\eta.
\end{equation}
 Again by \eqref{dot phi r}, \eqref{econ2} and \eqref{rloe}, for $t \in [0, T_\eta^*)$ we have
\begin{equation} \label{H-3}
\Delta \geq K\Big(\varepsilon_0-\eta(1+\sin\gamma_0)(1+\cos\gamma_0)\Big) -M\Big(1+\frac{1}{\varepsilon_0-\eta(1+\sin\gamma_0)}\Big) \geq 0,
\end{equation}
where we used an inequality $1 + \varepsilon_0 > \varepsilon_0 (1 + \cos \gamma_0)$. Thus, the relation \eqref{H-3} yields
\[ \mathcal{M}(L^+_{\gamma_0}(t)) \geq \mathcal{M}(L^+_{\gamma_0}(0)), \quad t \in [0, T_\eta^*). \]
We let $t \to T_\eta^*$ and use \eqref{H-2}, assumption (ii) in \eqref{ASP-1} to obtain
\[ {\mathcal M}_*(\varepsilon_0, \gamma_0)-\eta = \lim_{ t \to T_\eta^*} \mathcal{M}(L^+_{\gamma_0}(t)) \geq \mathcal{M}(L^+_{\gamma_0}(0)) > {\mathcal M}_*(\varepsilon_0, \gamma_0) \]
which is contradictory.  Therefore, we have $T_\eta^* = \infty$ and 
\[ \mathcal{M}\big(L^+_{\gamma_0}(t)\big) > {\mathcal M}_*(\varepsilon_0, \gamma_0)-\eta,  \qquad R(t) \geq \varepsilon_0, \quad \forall~t \in [0, \infty).    \]
In fact, the above inequality holds for any $\eta \ll 1$, thus, we have
\[ \mathcal{M}\big(L^+_{\gamma_0}(t)\big) \geq {\mathcal M}_*(\varepsilon_0, \gamma_0), \quad  \forall~t \in [0, \infty). \]
We substitute the above relation again into \eqref{H-3} to obtain the desired estimate:
\[ \Delta(t) \geq K\varepsilon_0 -
M\left(1+\frac{1}{\varepsilon_0}\right), \quad \forall~~t \in [0, \infty), \]
which then shows 
\[
\frac{d}{dt}\mathcal{M}\big(L^+_{\gamma_0}(t)\big) \geq 0,  \quad \forall~t \in [0, \infty).
\]
This completes the proof.

\subsubsection{The second part of the proof of Theorem \ref{T3.2}} \label{sec:5.3.2}
 In this part, we control the $L^2$-integral of $f_{\omega}$
on the arc $L_{\gamma_0}$ and show that concentration of mass occurs
on $L_{\gamma_0}(t)$ as time goes to infinity, when the coupling strength is large enough, i.e.,
\[  \lim_{t \to \infty} \int_{L^+_{\gamma_0} (t)} |f(\theta, \omega, t)|^2 ~ d\theta = \infty, \quad \mbox{for each $\omega \in \text{supp} ~ g(\omega)$}. \]
More precisely, under the same assumptions as in the previous part, we have
\[  \label{concen}
\int_{L^+_{\gamma_0}(t)} |f(\theta,\omega, t)|^2  ~ d\theta \geq \int_{L^+_{\gamma_0}(t)} |f_0(\theta, \omega, 0)|^2  ~ d\theta  ~ e^{(K\varepsilon_0\sin\gamma_0)  t}, \quad \forall~\omega \in \mbox{supp}~g(\omega).
\]
First, for each $t\geq0$ and each $\omega$ in $[-M,M],$ we define
\[
\Gamma_{\gamma_0, \omega}^+(t):=\int_{L^+_{\gamma_0}(t)} |f(\theta, \omega, t)|^2  ~ d\theta = \int_{\phi(t) - \frac{\pi}{2} + \gamma_0}^{\phi(t) + \frac{\pi}{2} - \gamma_0} |f(\theta,\omega, t)|^2  ~ d\theta.
\]
Then, by direct computation,  
\begin{align} 
\begin{aligned} \label{H-4}
\frac{d}{dt}\Gamma_{\gamma_0, \omega}^+(t) &= \dot\phi(t) \big(B_{+, \omega}(t)\big)^2 -  \big(B_{-, \omega}(t)\big)^2 + 2\int_{L^+_{\gamma_0}(t)} f \partial_t f  ~ d\theta\\
 &=: {\dot \phi}(t) \Big[ \big(B_{+, \omega}(t)\big)^2 - \big(B_{-, \omega}(t)\big)^2 \Big] + {\mathcal I}_3, 
\end{aligned}
\end{align}
where $B_{\pm, \omega}$ is the boundary value defined in \eqref{H-1-1}.  \newline

For the estimate of ${\mathcal I}_3$, we use  \eqref{E-1} to obtain
\begin{align}
\begin{aligned} \label{H-5}
{\mathcal I}_3(t) &= - 2 \int_{L^+_{\gamma_0} (t)} f\partial_\theta \Big[  f \big(\omega - KR\sin(\theta - \phi)\big) \Big]  ~ d\theta \\
&= -2 \int_{L^+_{\gamma_0} (t)} ( f \partial_\theta f) \big( \omega - KR\sin(\theta - \phi) \big) - f^2 KR \cos(\theta - \phi)  ~ d\theta \\
& = - \int_{L^+_{\gamma_0} (t)} ( \partial_\theta f^2\big) \big( \omega - KR\sin(\theta - \phi) \big) ~ d\theta 
+2 KR \int_{L^+_{\gamma_0}(t)}  f^2  \cos(\theta - \phi)  ~ d\theta \\ 
&=: {\mathcal I}_{31}(t) + {\mathcal I}_{32}(t).
\end{aligned}
\end{align}
Below, we estimate the terms ${\mathcal I}_{3i},~i=1,2$ separately. \newline 

\noindent $\bullet$ (Estimate on ${\mathcal I}_{31}$): Integration by parts yields
\begin{align}
\begin{aligned} \label{H-6}
{\mathcal I}_{31}(t) &= - \Big[  \big(B_{+, \omega}(t)\big)^2 \big(\omega - KR\sin(\frac{\pi}{2} - \gamma_0)\big) -  \big(B_{-, \omega}(t)\big)^2 \big(\omega - KR\sin(-\frac{\pi}{2} + \gamma_0)\big) \Big]\\
&- KR \int_{L^+_{\gamma_0}(t)}  \big( f(\theta,\omega, t)\big)^2 \cos(\theta - \phi)    ~ d\theta \\
&=: {\mathcal I}_{311}(t) + {\mathcal I}_{312}(t).
\end{aligned}
\end{align}
\vspace{0.2cm}

\noindent $\diamond$ (Estimate on ${\mathcal I}_{311}$): By rearranging terms, we have
\begin{align}
\begin{aligned} \label{H-7}
& {\mathcal I}_{311}(t) =-  \omega \Big[ \big(B_{+, \omega}(t)\big)^2  - \big(B_{-, \omega}(t)\big)^2 \Big]  + KR\cos\gamma_0 \Big[ \big(B_{+, \omega}(t)\big)^2  + \big(B_{-, \omega}(t)\big)^2 \Big], \\
& {\mathcal I}_{32}(t) + {\mathcal I}_{312}(t) \geq KR \sin\gamma_0 \Gamma_{\gamma_0, \omega}^+(t).
\end{aligned}
\end{align}
In \eqref{H-4}, we combine all estimates \eqref{H-5}, \eqref{H-6}, \eqref{H-7} and use $R \geq \varepsilon_0$ to obtain
\begin{align}
\begin{aligned} \label{H-8}
& \frac{d}{dt} \Gamma_{\gamma_0, \omega}^+(t) \\
& \hspace{0.5cm} \geq (\dot\phi(t) - \omega) \Big[  \big(B_{+, \omega}(t)\big)^2 - \big(B_{-, \omega}(t)\big)^2 \Big] \\
& \hspace{0.5cm} + KR\cos\gamma_0 \Big[  \big(B_{+, \omega}(t)\big)^2 + \big(B_{-, \omega}(t)\big)^2 \Big] +  KR \sin\gamma_0  \Gamma_{\gamma_0, \omega}^+(t)  \\
&\hspace{0.5cm} \geq \big(KR\cos\gamma_0  - |\dot\phi(t) - \omega| \big) \Big[  \big(B_{+, \omega}(t)\big)^2 + \big(B_{-, \omega}(t)\big)^2 \Big] +  KR \sin\gamma_0  \Gamma_{\gamma_0,\omega}^+(t) \\
&\hspace{0.5cm} \geq  \underbrace{\big(KR\cos\gamma_0  - |\dot\phi(t) - \omega| \big)}_{ =:{\tilde \Delta}(t)} \Big[  \big(B_{+, \omega}(t)\big)^2 + \big(B_{-, \omega}(t)\big)^2 \Big] +  K \varepsilon_0 \sin\gamma_0  \Gamma_{\gamma_0, \omega}^+(t).
\end{aligned}
\end{align}
We next estimate the sign of ${\tilde \Delta}$. It follows from \eqref{dot phi r} and $|\omega| \leq M$ that we have
\begin{equation} \label{H-9}
|{\dot \phi} -\omega| \leq |{\dot \phi}| + M \leq M \Big(1 + \frac{1}{R} \Big) + K (1-R).
\end{equation}
Then, \eqref{H-9} and \eqref{dot phi r}(ii) imply
\begin{align} 
\begin{aligned} \label{H-10}
{\tilde \Delta} &\geq  K\Big[ R(1+\cos\gamma_0)-1 \Big]  - M \Big(1+\frac{1}{R} \Big) \\
 &\geq K\Big(\big( (1+\sin\gamma_0){\mathcal M}\big(L^+_{\gamma_0}(t)\big) - 1\big)(1+\cos\gamma_0)-1\Big)  \\
 &- M\bigg(1+\frac{1}{ (1+\sin\gamma_0) {\mathcal M} \big(L^+_{\gamma_0}(t)\big) - 1}\bigg) \\
 &\geq K\Big(\big( (1+\sin\gamma_0){\mathcal M}\big(L^+_{\gamma_0} (0)\big) - 1\big)(1+\cos\gamma_0)-1\Big)  \\
 &- M\bigg(1+\frac{1}{ (1+\sin\gamma_0) {\mathcal M} \big(L^+_{\gamma_0}(0)\big) - 1}\bigg) \\
 & \geq K\varepsilon_0  - M\Big(1+\frac{1}{ \varepsilon_0}\Big) > 0.
\end{aligned}
\end{align}
In the last line, we used the same argument as in the proof of Step B and Step C in Theorem \ref{T3.2}.  Finally, we use \eqref{H-8} and \eqref{H-10} to obtain a Gronwall's inequality:
\[
\frac{d}{dt} \Gamma_{\gamma_0, \omega}^+(t) \geq K \sin\gamma_0 \varepsilon_0  \Gamma_{\gamma_0, \omega}^+(t). \]
This yields the desired estimate. This completes the proof of Theorem \ref{T3.2}.

\section{Lower bounds for the amplitude order parameter} \label{sec:6}
\setcounter{equation}{0}

In this section, we prove Theorem \ref{T3.3}, by establishing the promised asymptotic lower bound on the order parameter $R$ (assuming $R_0 >0$). 
A first key step is the existence of a positive lower bound $\underline{R}$ of the order parameter $R$ for the system \eqref{K-S} with $R(0)=R_0$. 
For such a lower bound, we need a large coupling strength $K$, depending on $\frac{1}{R_0}$, and under this assumption, we will first show that if $R_0>0$,
then we can guarantee that the mass in the sector $L^+_{\pi/3}(t)$  will remain above a universal value, whenever $\dot R\le0$. This will enable us to establish the lower bound $\underline{R}$ of $R$. This is the result of Corollary \ref{C6.1}, where we also prove that $\dot R$ will remain below a small positive constant after some time (see Figure \ref{fig:R}). This lower bound will induce in Section \ref{sec:6.3}.  a rough asymptotic lower bound 2/3 for $R$; see Proposition \ref{P6.2}. And we finally improve this rough lower bound to our desired one, namely, $R_\infty$ in Theorem \ref{T3.3}, which tends to $1$ as $K\to \infty$. 

Throughout this section, we will assume that the natural frequency density function $g = g(\omega)$ is compactly supported on the interval $[−M,M]$. It is important to note that for the results in this section, we do not require the previous assumption \eqref{ASP-1}$(ii)$ given in Section \ref{sec:5} on the initial configuration.

\begin{figure}
\centering
\includegraphics[width=0.5\textwidth]{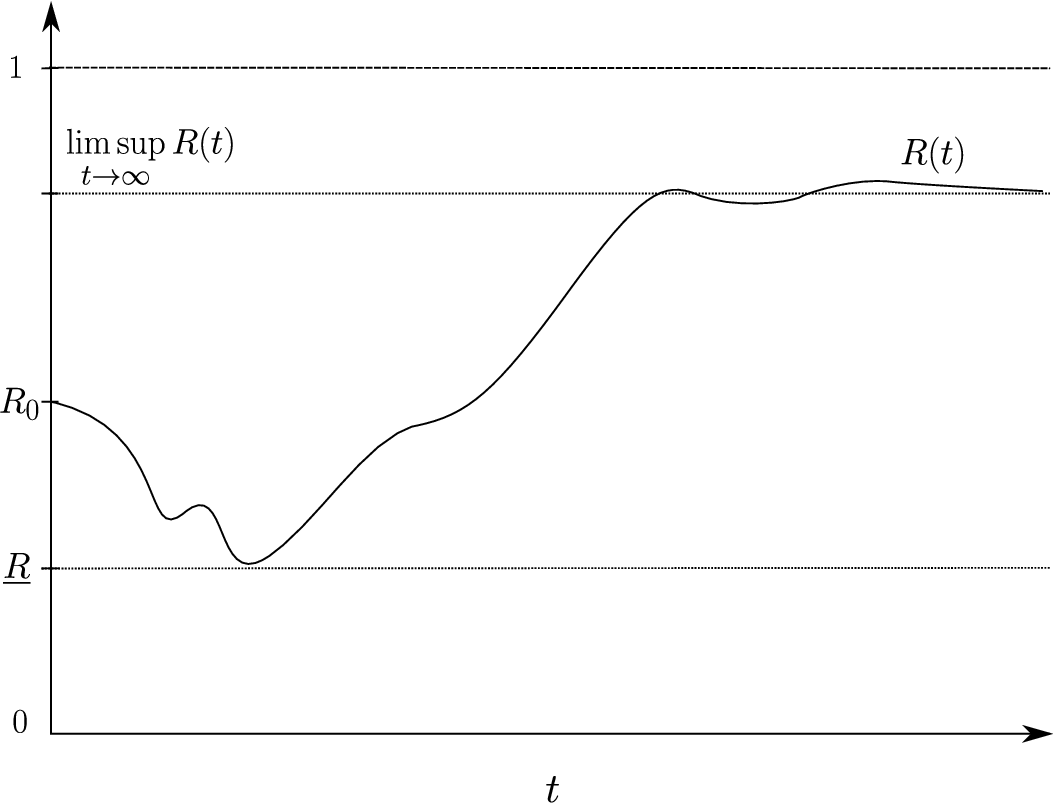}
\caption{Schematic diagram on the dynamics of $R(t)$}
\label{fig:R}
\end{figure}  
  
\subsection{Several lemmata} In this subsection, for a uniform lower bound of $R$, we will first present several lemmata. The constants appearing in all computations are not necessarily optimal. Our strategy to find a uniform lower bound is as follows. We will first assume that such a uniform lower bound $\underline{R}$ exists a priori and then later, by the choice of large $K$, we will remove this a priori assumption and obtain the uniform positive lower bound. We first begin a series of lemmata with the growth estimate for $f$. We assume that there exists a uniform positive lower bound for $R$ in the following lemmata.
\begin{lemma}\label{L6.1}
Let $f = f(\theta, \omega, t)$ be the solution to \eqref{K-S}. Then, we have 
\[  \label{ll-cond}
\|f(t)\|_{L^\infty(\mathbb{T}\times[-M,M])} \leq  ||f_0||_{L^{\infty}} e^{Kt},  \quad t \in [0, \infty).
\]
\end{lemma}
\begin{proof} For $(\theta_0, \omega_0) \in \mathbb T \times \mathbb R$, we define a forward characteristics $(\theta(t), \omega_0)$  issued from $(\theta_0, \omega_0)$ at time $t= 0$ as a solution to the following Cauchy problem:
\[
\begin{dcases}
{\dot \theta}(t) = \omega_0-KR(t)\sin(\theta(t)-\phi(t)), \quad {\dot \omega}(t)=0, \quad t >0, \cr
(\theta(0), \omega(0)) = (\theta_0, \omega_0). 
\end{dcases}
\]
Since the right-hand side of the ODE is Lipschitz continuous and sub-linear in $(\theta, \omega)$, we have a global solution and $\omega(t) = \omega_0$. On the other hand, we use 
\begin{align*}
(\partial_{t}f)(\theta,\omega,t)&=-\partial_{\theta}\big([\omega-KR\sin(\theta-\phi)]f\big)\\
&=-(\omega-KR\sin(\theta-\phi))\partial_{\theta}f+KR\cos(\theta-\phi)f(\theta,\omega,t).
\end{align*}
to see the time-rate of change of $f$ along the characteristics $(\theta(t), \omega_0, t)$:
\begin{align}
\begin{aligned}\label{f}
& \frac{d}{dt}f(\theta(t),\omega_0,t) \\
& \hspace{1cm} = -[\omega_0 -KR(t)\sin(\theta(t)-
\phi(t))]\partial_{\theta}f(\theta(t), \omega_0,t)\\
&\hspace{1cm}+KR(t)\cos(\theta(t)-\phi(t))f(\theta(t),\omega_0,t) +{\dot \theta}(t)\partial_{\theta}f(\theta(t),\omega_0,t)\\
&\hspace{1cm}=KR(t)\cos(\theta(t)-\phi(t))f(\theta(t), \omega_0,t)\\
&\hspace{1cm}\leq K f(\theta(t), \omega_0,t).
\end{aligned}
\end{align}
This yields
\[
f(\theta(t),\omega_0,t)\leq e^{Kt}f_0(\theta_0, \omega_0) \leq \|f_0\|_{L^{\infty}(\mathbb T \times \mathbb R)} e^{Kt}, \quad \mbox{in}~[0, T].
\]
This yields the desired $L^{\infty}$-estimate of $f$. 
\end{proof}

\begin{lemma}\label{L6.2}
The following assertions hold.
\begin{enumerate}
\item
$R =R(t)$ is Lipschitz continuous in $[0, \infty)$. 
\item
Suppose that the initial density $f_0$ and the order parameter $R$ satisfy the following conditions:
\[ \|f_0 \|_{L^{\infty}} < \infty, \qquad \inf_{0 \leq t \leq T} R(t) \geq \underline{R}, \quad \mbox{for some $T \in (0, \infty]$ and $\underline{R} > 0$},
\]
Then, there exists $\eta^{\prime}>0$ such that the functions $R(\cdot), {\dot R}(\cdot)$ and  $\mathcal{M}(L^+_{\gamma}(\cdot))$ are Lipschitz continuous in $[0,T+\eta^{\prime}),$ for any $\gamma$ in $(-\frac{\pi}{2},\frac{\pi}{2}).$ The Lipschitz constants for $R(\cdot), \dot R(\cdot)$ and $\mathcal (L^+_\gamma(\cdot))$ are given by
\begin{align*}
&|\dot R| < M + K,\\
&|\ddot R| \leq \frac{1}{\underline R}\bigg[4(K^2 + KM) + 2KM + 2M^2 + 2(M+K)^2 \bigg],\\
&\bigg|\frac{d}{dt}\mathcal M(L^+_\gamma(t))\bigg| \leq 2M\bigg[K + \frac{2M}{\underline R} + K\Big(1-\frac{\underline R}{2}\Big) + M \bigg]\|f_0\|_{L^\infty(\mathbb T \times \mathbb R)} e^{K(T+\eta)}.
\end{align*}
\end{enumerate} 
\end{lemma}
\begin{proof}
Since the proof is very lengthy, we postpone it to Appendix \ref{App-B}.
\end{proof} 

\vspace{0.2cm}

In the next Lemma we will show how the values $R$ and ${\dot R}$ can be used to control the mass in $L^+_{\frac{\pi}{3}}$. For $t\geq0$ and $\gamma \in (\frac{\pi}{3},\frac{\pi}{2})$,  we set
\begin{align*}
& L^-_{\gamma}(t) :=\big(\phi(t)+\frac{\pi}{2}+\gamma,\phi(t)+\frac{3\pi}{2}-\gamma\big) = I^-_{\frac{\pi}{2}-\gamma},  \cr
& \tilde{B}_{-,\omega}(t) :=f \Big(\phi(t)+\frac{\pi}{2}+\gamma, \omega, t \Big), \quad \tilde{B}_{+,\omega}(t) :=f \Big(\phi(t)+\frac{3\pi}{2}-\gamma, \omega, t \Big), \cr
& E_{1}(K, M, \mu, {\underline R}, \gamma) := \frac{\sin\gamma}{\cos^{2}\gamma}\frac{M}{K\underline{R}}+\frac{1-\sin\gamma}{2} \quad  \mbox{and} \cr
& E_{2}(K, M, \mu, {\underline R},\gamma) := 1-\sin\gamma+(1+\sin\gamma)\frac{M}{K\underline{R}\cos^{2}\gamma}+\frac{\sin\gamma}{\cos^{2}\gamma}\frac{M}{K\underline{R}}. 
\end{align*}
In the subsequent three lemmata, we will study the relationships between $(R(t_0), {\dot R}(t_0))$ and ${\mathcal M}(L^+_{\frac{\pi}{3}}(t_0))$ under the following three situations:
\begin{align*}
& \mbox{Case 1}:~ R(t_0) \geq \underline{R}, \qquad \dot{R}(t_0) \leq 0, \cr
& \mbox{Case 2}:~ R(t_0)\geq\underline{R}, \qquad \dot{R}(t_0)<K \mu, \quad K \gg 1,~~ 0< \mu \ll 1,   \cr
& \mbox{Case 3}:~\inf_{0 \leq t \leq T} R(t) \geq\underline{R}, \qquad \dot{R}(T)=K \mu.
\end{align*}

In the sequel, to simplify presentation appearing in the messy computations, we will consider the sector $L_{\frac{\pi}{3}}^+$ and to emphasize $\gamma$ dependence in $E_1$ and $E_2$, we suppress other dependence, i.e. $E_i(\gamma) := E_i(K, M, \mu, {\underline R}, \gamma)$. 
\begin{lemma}\label{L6.3} 
Let $\gamma \in (\frac{\pi}{3} , \frac{\pi}{2} )$ and suppose that there exists $t_0 \geq0$ and~$\underline{R}>0$ such that 
\begin{equation} \label{Apr-2}
 R(t_0) \geq \underline{R}, \qquad \dot{R}(t_0) \leq 0. 
\end{equation} 
Then, $R(t_0)$ is controlled by the mass ${\mathcal M}(L^+_{\frac{\pi}{3}}(t_0))$, and vice verse:
\begin{equation} \label{Extra-A}
2{\mathcal M}(L^+_{\frac{\pi}{3}}(t_0)) -E_{2}(\gamma) -1 \leq R(t_0) \leq 2 {\mathcal M}(L^+_{\frac{\pi}{3}}(t_0)) + 2 E_1(\gamma)-1.
\end{equation}
\end{lemma}
\begin{proof} (i) (Proof of the upper bound): For derivation of the second inequality, we first estimate how the mass ${\mathcal M}(L_{\gamma}^-)$ can be controlled by the mass ${\mathcal M}(L_{\gamma}^+), K, M, \underline{R}$ and $\gamma$. In the sequel, all quantities will be evaluated at $t = t_0$.  \newline

\noindent $\bullet$ Step A (Controlling the mass ${\mathcal M}(L_{\gamma}^-)$): By a priori condition \eqref{Apr-2} and Lemma \ref{L5.2}, 
\begin{align}
\begin{aligned} \label{H-11}
0 \geq \dot{R} &= -\iint_{\mathbb T \times \mathbb R}\sin(\theta-\phi)\omega f(\theta,\omega) ~ d\theta d\omega \\
&+KR \iint_{\mathbb T \times \mathbb R} \sin^{2}(\theta-\phi)f(\theta,\omega) ~ d\theta d\omega\\
&\geq-M+K\underline{R}\cos^{2}\gamma ( 1 - {\mathcal M}(L_{\gamma}^+) - {\mathcal M}(L_{\gamma}^-)),
\end{aligned}
\end{align}
where we used 
\begin{align*}
& \iint_{\mathbb T \times \mathbb R} \sin^{2}(\theta-\phi)f(\theta,\omega) ~ d\theta d\omega \cr
& \hspace{1cm} \geq  \iint_{(\mathbb T \backslash (L_{\gamma}^+ \cup L_{\gamma}^-)) \times \mathbb R} \sin^{2}(\theta-\phi)f(\theta,\omega) ~ d\theta d\omega \cr
& \hspace{1cm} \geq \cos^2 \gamma \iint_{(\mathbb T \backslash (L_{\gamma}^+ \cup L_{\gamma}^-)) \times \mathbb R} f(\theta,\omega) ~ d\theta d\omega \cr
& \hspace{1cm}  = \cos^2 \gamma \Big( 1 - {\mathcal M}(L_{\gamma}^+) - {\mathcal M}(L_{\gamma}^-) \Big).
\end{align*}
Then the relation \eqref{H-11} yields
\begin{align}
\begin{aligned} \label{midc}
&1-{\mathcal M}(L^+_{\gamma}(t_0)) - {\mathcal M}(L^-_{\gamma}(t_0)) \leq \frac{M}{K\underline{R}\cos^{2}\gamma}, \quad \mbox{or} \\
& {\mathcal M}(L^-_{\gamma}(t_0)) \geq 1-{\mathcal M}(L^+_{\gamma}(t_0))-\frac{M}{K\underline{R}\cos^{2}\gamma}.
 \end{aligned}
 \end{align}

\vspace{0.5cm} 
 
\noindent $\bullet$ Step B (Bounding $R$ by ${\mathcal M}(L_{\gamma}^+)$):  Since $\frac{\pi}{3} < \gamma$, we have
\[ L^+_{\gamma} \subset L^+_{\frac{\pi}{3}}. \]
Then, we use the above relation, \eqref{O-K-R} and \eqref{midc} to obtain
\begin{align}
\begin{aligned} \label{H-12}
R & =\iint_{L^+_{\gamma}\times \mathbb R}\langle e^{{\mathrm i}\theta},e^{{\mathrm i}\phi}\rangle f ~ d\theta d\omega+ \iint_{(\mathbb{T}\backslash (L^+_{\gamma} \cup L^-_{\gamma})) \times \mathbb R}\langle e^{{\mathrm i}\theta},e^{{\mathrm i}\phi}\rangle f ~ d\theta d\omega \\
&+\iint_{L^-_{\gamma} \times \mathbb R}\langle e^{{\mathrm i}\theta},e^{{\mathrm i}\phi}\rangle f ~ d\theta d\omega\\
 & \leq {\mathcal M}(L^+_{\gamma})+\sin\gamma \Big(1- {\mathcal M}(L^-_{\gamma})- {\mathcal M}(L^+_{\gamma}) \Big)-\sin\gamma {\mathcal M}(L^-_{\gamma}) \\
 & \leq {\mathcal M}(L^+_{\gamma}) + \frac{\sin\gamma}{\cos^{2}\gamma}\frac{M}{K\underline{R}}-\sin\gamma 
 {\mathcal M}(L^-_{\gamma}) \\
 & \leq {\mathcal M}(L^+_{\gamma}) +2\frac{\sin\gamma}{\cos^{2}\gamma}\frac{M}{K\underline{R}}+\sin\gamma {\mathcal M}(L^+_{\gamma}) - \sin\gamma \\
 & \leq2 {\mathcal M}(L^+_{\gamma})-1+2\frac{\sin\gamma}{\cos^{2}\gamma}\frac{M}{K\underline{R}}+(\sin\gamma-1) {\mathcal M}(L^+_{\gamma}) +1-\sin\gamma \\
 & \leq 2{\mathcal M}(L^+_{\frac{\pi}{3}})-1+2\frac{\sin\gamma}{\cos^{2}\gamma}\frac{M}{K\underline{R}}+1-\sin\gamma \\
 & = 2 \Big({\mathcal M}(L^+_{\frac{\pi}{3}}) - \frac{1}{2} + \frac{\sin\gamma}{\cos^{2}\gamma}\frac{M}{K\underline{R}} + \frac{1-\sin\gamma}{2}                     \Big) \\
 & = 2 \Big( {\mathcal M}(L^+_{\frac{\pi}{3}}) - \frac{1}{2} + E_1 \Big).
 \end{aligned}
 \end{align}
 This verifies the upper bound. \newline

\noindent (ii) (Proof of the lower bound):  We use the similar argument to the first part of \eqref{H-12} to find 
\begin{align}
\begin{aligned} \label{H-13}
R &=\iint_{L^+_{\gamma}\times \mathbb R}\langle e^{{\mathrm i}\theta},e^{{\mathrm i}\phi}\rangle f~ d\theta d\omega+ \iint_{(\mathbb{T}\backslash (L^-_{\gamma}\cup L^+_{\gamma}))\times \mathbb R}\langle e^{{\mathrm i}\theta},e^{{\mathrm i}\phi}\rangle f ~ d\theta d\omega \\
  &+\iint_{L^-_{\gamma}\times \mathbb R}\langle e^{{\mathrm i}\theta},e^{{\mathrm i}\phi}\rangle f ~ d\theta d\omega \\
  &\geq \sin\gamma {\mathcal M}(L^+_{\gamma}) -\sin\gamma \Big(1-{\mathcal M}(L^+_{\gamma}) - {\mathcal M}(L^-_{\gamma}) \Big)- {\mathcal M}(L^-_{\gamma}).
\end{aligned}
\end{align}
On the other hand, we again use the same arguments as in \eqref{H-11} to find 
\begin{align*}
0 &\geq {\dot R} = -\iint_{\mathbb T \times \mathbb R}\sin(\theta-\phi)\omega f(\theta,\omega,t) ~ d\theta d\omega+KR \iint_{\mathbb T \times \mathbb R} \sin^{2}(\theta-\phi)\big)\rho(\theta,t) ~ d\theta d\omega\\ 
&\geq-M+K\underline{R}\cos^{2}\gamma(1-{\mathcal M}(L^+_{\gamma}) - {\mathcal M}(L^-_{\gamma})).\\
\end{align*}
This yields
\[
1- {\mathcal M}(L^+_{\gamma}) - {\mathcal M}(L^-_{\gamma}) \leq\frac{M}{K\underline{R}\cos^{2}\gamma}. 
\]
We use \eqref{H-13} and the fact that ${\mathcal M}(L^+_{\gamma}) + {\mathcal M}(L^-_{\gamma}) \leq 1$ to obtain
\begin{align*}
R &\geq \sin\gamma {\mathcal M}(L^+_{\gamma})-\frac{\sin\gamma}{\cos^{2}\gamma}\frac{M}{K\underline{R}}-
{\mathcal M}(L^-_{\gamma}) \cr
&\geq \sin\gamma {\mathcal M}(L^+_{\gamma})-\frac{\sin\gamma}{\cos^{2}\gamma}\frac{M}{K\underline{R}}-1+ {\mathcal M}(L^+_{\gamma}) \cr
&= (1+\sin\gamma) {\mathcal M}(L^+_{\gamma})-1-\frac{\sin\gamma}{\cos^{2}\gamma}\frac{M}{K\underline{R}}.
\end{align*}
Similarly, Lemma \ref{L5.2} and the fact that $\dot{R}\leq0$ imply
\[
{\mathcal M}(L^+_{\frac{\pi}{3}})-\frac{M}{K\underline{R}\cos^{2}\gamma}\leq {\mathcal M}(L^+_{\gamma}). 
\]
Hence,
\[
R\geq(1+\sin\gamma) {\mathcal M}(L^+_{\frac{\pi}{3}}) -1-(1+\sin\gamma)\frac{M}{K\underline{R}\cos^{2}\gamma}-\frac{\sin\gamma}{\cos^{2}\gamma}\frac{M}{K\underline{R}}. 
\]
Thus, we again use the fact that $\mathcal{M}(L^+_{\frac{\pi}{3}})\leq 1$ to get
\[
R\geq 2 {\mathcal M}(L^+_{\frac{\pi}{3}}) -1-(1-\sin\gamma)-(1+\sin\gamma)\frac{M}{K\underline{R}\cos^{2}\gamma}-\frac{\sin\gamma}{\cos^{2}\gamma}\frac{M}{K\underline{R}}.
\] 
This yields the desired result.
\end{proof}
\begin{remark}
Note that the estimate \eqref{Extra-A} can be rewritten as follows.
\[ \frac{R(t_0) + 1}{2} -E_1(\gamma) \leq  {\mathcal M}(L^+_{\frac{\pi}{3}}(t_0)) \leq \frac{R(t_0) + 1}{2} + \frac{E_2(\gamma)}{2}. \]
\end{remark}

\vspace{0.5cm}

We next show that there exists a positive constant $\mu,$ such that when $\dot{R}$ is below $K \mu,$  the mass in 
$L^+_{\frac{\pi}{3}}$ is nondecreasing. 
\begin{lemma}\label{L6.4}
Let $K, \underline{R}$ and $\mu$ satisfy the relation:
\begin{equation} \label{H-13-1}
\frac{1}{2}>\frac{M}{K\underline{R}}+\frac{M}{K\underline{R}^{2}}+\frac{1}{\underline{R}\sqrt{\underline{R}}}\sqrt{\frac{M}{K}+\mu},
\end{equation}
that is, $K$ is sufficiently large and $\mu$ is sufficiently small relative to $\underline{R}$. Suppose that at $t = t_0$, the oder parameter $R$ satisfies 
\[  R(t_0)\geq\underline{R}, \qquad \dot{R}(t_0)<K \mu.   \]
Then, we have
\[
\frac{d}{dt} \Big|_{t = t_0}\mathcal{M}(L^+_{\frac{\pi}{3}}(t)) \geq0. 
\]
\end{lemma}
\begin{proof} In the sequel, for notational simplicity, we will assume  that all the time dependent expression are evaluated at $t=t_0$. By \eqref{dL}, we have
\begin{align}\label{LdL}
\begin{aligned}
\frac{d}{dt}\mathcal{M}(L^+_{\frac{\pi}{3}}(t))&=KR \cos\frac{\pi}{3}\int_{-M}^{M} \big[B_{-,\omega}(t)+B_{+,\omega}(t)\big] ~ d\omega \\
&+\int_{-M}^{M}\big(\dot{\phi}(t)-\omega\big)\big[B_{+,\omega}(t)-B_{-,\omega}(t)\big] ~ d\omega\\
& \geq\bigg(\frac{K\underline{R}}{2}-M-|\dot{\phi}|\bigg) \int_{-M}^{M} \big[B_{+,\omega}(t)+B_{-,\omega}(t)\big] ~ d\omega.
\end{aligned}
\end{align}
Thus, we need to show
\begin{equation} \label{H-14}
 \frac{K\underline{R}}{2}-M-|\dot{\phi}| \geq 0. 
\end{equation} 
To check \eqref{H-14}, we use Lemma \ref{L5.2} to find
\begin{align}
\begin{aligned} \label{H-15}
\dot{\phi} &= \frac{1}{R}\iint_{\mathbb T \times \mathbb R}\cos(\theta-\phi)\omega f(\theta,\omega,t) ~ d\theta d\omega-\frac{K}{2}\iint_{\mathbb T \times \mathbb R}\sin\left(2(\theta-\phi)\right)f(\theta,\omega,t) ~ d\theta d\omega, \\
{\dot{R}} &= -\iint_{\mathbb T \times \mathbb R}\sin(\theta-\phi)\omega f(\theta,\omega,t) ~ d\theta d\omega+KR\iint_{\mathbb T \times \mathbb R}\sin^{2}(\theta-\phi)\big)f(\theta,\omega,t) ~ d\theta d\omega . 
\end{aligned}
\end{align}
Then, we use \eqref{H-15}, $f = 0$ for $|\omega| > M$ and Cauchy-Schwarz inequality to get
\begin{align}
\begin{aligned} \label{H-16}
|\dot{\phi}|&<\frac{M}{R}+K\iint_{\mathbb T \times \mathbb R}|\sin(\theta-\phi)|f(\theta,\omega,t) ~ d\theta d\omega.\\ 
&<\frac{M}{R}+\sqrt{\frac{K}{R}}\sqrt{KR\iint_{\mathbb T \times \mathbb R}\sin^{2}(\theta-\phi)f(\theta,\omega,t) ~ d\theta d\omega}\\
&<\frac{M}{R}+\sqrt{\frac{K}{R}}\sqrt{M+\dot{R}}.
\end{aligned}
\end{align}
Note that the conditions \eqref{H-13-1} yield
\[\frac{1}{2}> \frac{M}{K\underline{R}}+\frac{M}{K\underline{R}^{2}}+\frac{1}{\underline{R}\sqrt{\underline{R}}}\sqrt{\frac{M}{K}+
\mu} > \frac{M}{K\underline{R}}+\frac{M}{K\underline{R}^{2}}+\frac{1}{\underline{R}\sqrt{\underline{R}}}\sqrt{\frac{M}{K}+\frac{\dot{R}}{K}}.\]
By multiplying $K \underline{R}$, we have
\begin{equation} \label{H-17}
\frac{K \underline{R}}{2} - M > \frac{M}{\underline{R}}+\sqrt{\frac{K}{\underline{R}}} \sqrt{M+ \dot{R}}.
\end{equation}
We now combine \eqref{H-16} and \eqref{H-17} to get
\[  |\dot\phi(t_0)| < \frac{K \underline{R}}{2} - M.   \]
This and \eqref{LdL} implies the desired estimate.
\end{proof}
In the following Lemma, under the assumption that $\dot R(t)\geq K\mu $ for some $t,$  we quantify
the increase of $R,$ at some later time.
\begin{lemma}\label{L6.5} 
Suppose that $f_0, R$ and $K$ satisfy
\[ \|f_0 \|_{L^{\infty}(\mathbb T \times \mathbb R)} < \infty, \qquad \inf_{0 \leq t \leq T} R(t) \geq\underline{R}, \qquad \dot{R}(T)=K \mu,\qquad K^2 \mu > \frac{M^2}{2\underline R^2} - \frac{3M^2}{4\underline R}
\]
for some $T\geq0$ and some positive constants $\underline{R},$ and $\mu$. Then there exist positive constants $d:=d(K,M,\underline{R},\mu)$ and $E_{3}:=E_{3}(K,M,\underline{R},\mu)$ satisfying 
\[
\dot{R}>0\hspace{1em} {\rm{in}} \hspace{1em} [T,T+d), \qquad 
R(T+d)-R(T)\geq\frac{1}{12}\underline{R}\mu-E_{3}.\] 
where
\begin{align*}
E_{3} &:=\bigg|\frac{\underline{R}}{12}\mu\frac{\frac{M^{2}}{4K}+\frac{M^{2}}{2\underline{R}K}}{\frac{1}{3}\underline{R}K\mu-\bigg(\frac{M^{2}}{6\underline{R}K}-\frac{M^{2}}{4K}\bigg)} +\frac{1}{4K}\bigg(\frac{M^{2}}{6\underline{R}K}-\frac{M^{2}}{4K}\bigg)\bigg[1-\frac{\frac{M^{2}}{4K}+\frac{M^{2}}{2\underline{R}K}}{\frac{1}{3}\underline{R}K\mu-\bigg(\frac{M^{2}}{6\underline{R}K}-\frac{M^{2}}{4K}\bigg)}\bigg]\\
 &+\bigg(\frac{M^{2}}{4K}+\frac{M^{2}}{2\underline{R}K}\bigg)\bigg[\frac{1}{4K}\log\frac{\frac{1}{3}\underline{R}K\mu-\bigg(\frac{M^{2}}{6\underline{R}K}-\frac{M^{2}}{4K}\bigg)}{\frac{M^{2}}{4K}+\frac{M^{2}}{2\underline{R}K}}\bigg]\bigg|.
\end{align*}
\end{lemma}
\begin{proof} Since the proof is rather lengthy, we leave it to Appendix C.
\end{proof}

\vspace{0.5cm}

\subsection{A framework for the asymptotic lower bound of $R$} \label{sec:6.2} In this subsection, we study sufficient framework $({\mathcal H})$ for the lower bounds of order parameter, and then present a rough estimate for the lower bound for $R$ in Proposition \ref{P6.2}:
\[ \liminf_{t\rightarrow\infty}R(t) > \frac{2}{3}. \]
and then in the proof of Theorem \ref{T3.3} presented in next subsection, we will improve the above uniform lower bound by showing 
\[ \liminf_{t\rightarrow\infty}R(t) \geq 1 - \frac{|{\mathcal O}(1)|}{\sqrt{K}}.    \]
\newline
We first list our main framework $({\mathcal H})$ for the lower bound estimate of $R$ as follows. \newline

\begin{itemize}
\item
$({\mathcal H}1)$: The ${\mathcal C}^1$ initial data $f_0$ satisfies $ R_0>0$, and $g$ is compactly supported on the interval $[-M,M]$.

\vspace{0.5cm}

\item
$({\mathcal H}2)$: The constant $\mu$ is sufficiently small and $K$ is sufficiently large so that
\begin{align*}
& \frac{1}{2}>2\frac{M}{KR_0}+4\frac{M}{KR_0^{2}}+\frac{2\sqrt{2}}{R_0 \sqrt{R_0}}\sqrt{\frac{M}{K}+\mu}, \qquad K^2\mu - \Big(\frac{2M^2}{R_0^2}-\frac{3M^2}{2R_0}\Big)>0, \cr
& \mbox{and} \quad K> \max \Big\{\frac{64M}{\sqrt{3}}, \frac{64\sqrt{3}M}{3-(\sqrt{3} - 2R_0)^2} \Big \}.
\end{align*}

\vspace{0.5cm}

\item
$({\mathcal H}3)$: The constant $\gamma$ is contained in $(\frac{\pi}{3},\frac{\pi}{2})$, and $E_i = E_i(K, M, \frac{R_0}{2}, \gamma),~i=1,2$ and $E_3 = E_3(K, M, \frac{R_0}{2}, \mu)$ satisfy
\[ R_0 -2 E_{1}-E_{2}> \frac{R_0}{2}, \qquad \mu\frac{R_0}{24}> 2E_{1}+E_{2}+E_{3},
\]
where the positive constants $E_{1},$ $E_{2}$ and $E_{3}$ can be made sufficiently small by taking $K$ sufficiently large and $\gamma$ sufficiently close to $\frac{\pi}{2}.$

\vspace{0.5cm}

\item
$({\mathcal H}4)$: For a positive constant  $\kappa \in (\frac{2}{3},\frac{\sqrt{3}}{2}),$ the coupling strength strength $K$ is sufficiently large such that 
\[
\kappa < \frac{\sqrt{3}}{4} + \frac{1}{4}\sqrt{3 - \frac{64\sqrt{3}M}{K}}, \qquad 
\varepsilon_{\kappa} :=\frac{\kappa +1}{\kappa^{2}}\frac{M}{K}+\frac{(1-\kappa)}{\kappa} <1.
\]
\end{itemize}

\vspace{0.5cm}

Under the above assumptions, we will derive
\begin{equation} \label{Order-KM}
 \liminf_{t\rightarrow\infty}R(t)\geq R_\infty := 1+\frac{M}{K}-\sqrt{\frac{M^{2}}{K^{2}}+4\frac{M}{K}}. 
\end{equation} 
Thus, letting $K \to \infty$, the above estimate and $R \leq 1$ yield
\[  \lim_{K \to \infty} \liminf_{t \to \infty} R(t) = 1. \]
Hence we obtain a complete phase synchronization in this asymptotic limit. Thus, the estimate \eqref{Order-KM} indicates the kinetic analogue of the practical synchronization estimates. 

\vspace{0.5cm}

Now, we are ready to state the first proposition of this subsection. In this proposition, we show that the mass in $L^+_{\frac{\pi}{3}}$ remains above a constant, depending only on $R_0$ and $E_1$ in the set of times $t$ where $R$ is in non-increasing mode. 
\begin{proposition}\label{P6.1}
Suppose that the assumptions $({\mathcal H}1)$ - $({\mathcal H}3)$ hold, and assume that there exist $t_{0}\geq0$ and $\underbar{R}>0$ such that 
\[
 R(t_{0})\geq R_0, \quad  \inf_{0 \leq t \leq t_0} R(t) \geq \underbar{R}.
\]
Then, for any $t\geq t_{0}$ satisfying $\dot{R}(t)\leq0$, we have
\begin{equation} \label{N-M}
\mathcal{M}(L^+_{\frac{\pi}{3}}(t))\geq\frac{1}{2}(R(t_{0})+1)-E_{1}.
\end{equation}
\end{proposition}
\begin{proof} Since the proof is rather lengthy, we postpone it to Apppendix \ref{App-D}.
\end{proof} 
\begin{remark} The estimate \eqref{N-M} looks similar to that appearing in Lemma \ref{L6.3}, however in \eqref{N-M}, we are comparing $\mathcal{M}(L^+_{\frac{\pi}{3}}(t))$ at any non-increasing instant $t$ after $t_0$ with the fixed constant $\frac{1}{2}(R(t_{0})+1)-E_{1}$. Thus the estimate in Lemma \ref{L6.3} corresponds to a special situation of \eqref{N-M}.
\end{remark}

Below, we present three corollaries followed by Proposition \ref{P6.1}. 
\begin{corollary}\label{C6.1}
Suppose that assumptions $({\mathcal H}1)$ - $({\mathcal H}3)$ hold and there exists $t_0 \geq 0$ such that 
\[
R(t_{0})\geq R_0.
\]
Then, we have
\[ R(t)\geq R(t_{0})-2 E_{1}-E_{2}, \quad t \geq t_0 \quad \mbox{and} \quad  \limsup_{t\rightarrow\infty}\dot{R}(t)\leq K\mu.  \]
\end{corollary}
\begin{proof}
(i)~First, assume $t_{0}=0.$ Then, our hypotheses guarantee that the
assumptions of Proposition \ref{P6.1} are satisfied. 
In the course of the proof of Proposition \ref{P6.1} in Appendix D, we have already shown that 
\[
R(t)\geq R(t_{0})-2 E_{1}-E_{2},  \quad \mbox{in $\mathcal{T}(t_0)$}. \]
where

\[
\mathcal{T}(t_0) := \bigg\{ t \in [t_{0}, \infty)~:~\mathcal{M}(L^+_{\frac{\pi}{3}}(t^{*}))\geq\frac{1}{2}(R(t_{0})+1)-E_{1}\hspace{1em}\forall t^{*}\in [t_{0},t] \cap {\mathcal N}(t_0)\bigg\},
\]
and we also showed $\mathcal{T}=[t_{0},\infty).$ Thus, we have the desired estimate for $t_0 = 0$, and it follows from $({\mathcal H}1)$ and $({\mathcal H}3)$ that 
\begin{equation} \label{N-Pa}
 R(t) > R_0 -2 E_{1}-E_{2}>\frac{R_0}{2}>0,  \quad t \in [0, \infty). 
\end{equation} 
For the case $t_{0}>0,$  note that by the above inequality, the assumptions $({\mathcal H}1)$, $({\mathcal H}2)$ and $({\mathcal H}3)$, one can apply for Proposition \ref{P6.1} for $t_{0}>0.$ Thus, the desired result follows by the same argument. \newline

\noindent (ii)~By definition of the order parameter,  $R$ is uniformly bounded, and 
\[ 0 \leq \liminf_{t\rightarrow\infty}R(t) \leq 1.
\]
Suppose that we have
\begin{equation} \label{H-18}
\limsup_{t\rightarrow\infty}\dot{R}>K\mu. 
\end{equation}
Let $\varepsilon^{\prime}>0$ be sufficiently small so that
\begin{equation}\label{con}
\frac{R_0}{24}\mu-\varepsilon^{\prime}-2E_{1}-E_{2}-E_{3}>0.
\end{equation}
Such $\varepsilon^{\prime}$ exists by $({\mathcal H}3)$. By the result in (i) we have
\begin{equation}\label{starlo0}
\liminf_{t\rightarrow\infty}R(t) \geq R_0 -2E_{1}-E_{2}. 
\end{equation}
On the other hand, by definition of $\liminf_{t\rightarrow\infty}R(t)$, there exists $t_{0}\geq0,$ such that
\begin{equation}\label{starlo}
R(t) \geq \liminf_{t\rightarrow\infty}R(t)-\varepsilon^{\prime}, \quad t \in  [t_{0},\infty).
\end{equation}
Thanks to \eqref{H-18}, boundedness of $R$ and continuity of ${\dot R}$ in Lemma \ref{L6.2}, there exists $t_{1}, t_2 \geq t_{0}$ such that
\[
\dot{R}(t_{1})>K\mu \quad \mbox{and} \quad \dot{R}(t_{2})=K\mu. 
\]
Thus, it follows from Lemma \ref{L6.5}, \eqref{starlo}, \eqref{starlo0} and \eqref{con} that there exists $d>0$ such
that 
\begin{align*}
R(t_{2}+d)&\geq R(t_{2})+\frac{R_0}{24}\mu-E_{3} \\
&\geq \liminf_{t\rightarrow\infty}R(t) +\frac{R_0}{24} \mu -\varepsilon^{\prime}-E_{3}\\
&\geq R_0+\frac{R_0}{24}\mu-\varepsilon^{\prime}-2E_{1}-E_{2}-E_{3} \\
&\geq R_0.
\end{align*}
Then, by the third inequality as above and \eqref{N-Pa}, we have
\[ R(t)\geq \liminf_{t\rightarrow\infty}R(t) +\frac{R_0}{24}\mu-\varepsilon^{\prime}-2E_{1}-E_{2}-E_{3},  \quad \mbox{in $[t_{2}+d,\infty)$}. \]
We use \eqref{con} to see that the above expression
contradicts the definition of $\displaystyle R_* = \liminf_{t \to \infty} R$. This yields
the desired result.
\end{proof}

\vspace{0.5cm}

We next show that the $L^{2}$ norm of $f$ in an interval
of length $\frac{\pi}{3},$ centered at $-\phi(t),$ decays exponentially after some time. This is  analogous to the phenomenon in   \eqref{NN-5}.
\begin{corollary}\label{C6.2} 
Suppose that assumptions $({\mathcal H}1)$ - $({\mathcal H}3)$ hold. Then, there exists $T\geq0$ such that   
\begin{align}
\begin{aligned} \label{decay2}
& \iint_{L^-_{\frac{\pi}{3}}(t) \times \mathbb R} |f|^2 ~ d\theta d\omega\leq  e^{-\frac{KR(0)}{4}(t-T)} \iint_{L^-_{\frac{\pi}{3}}(T) \times \mathbb R} |f(T)|^2 ~ d\theta d\omega , \quad t \geq T, \\
& \mathcal{M}\big(L^-_{\frac{\pi}{3}}(t)\big)\leq \Big(\frac{\pi}{3}\Big)^{\frac{1}{2}}  e^{-\frac{KR(0)}{8}(t-T)} \sqrt{\iint_{L^-_{\frac{\pi}{3}}(T) \times \mathbb R}f^{2}(\theta,\omega,T) ~ d\theta d\omega}.
\end{aligned}
\end{align}
\begin{proof} We postpone its proof in Appendix E.
\end{proof}
\end{corollary}
Before we present the last proposition showing that
$R$ will remain above $\frac{2}{3}$ after some time, we present a preparatory result.

\begin{lemma} \label{L6.6}
Suppose that assumptions $({\mathcal H}1)$ - $({\mathcal H}3)$ hold. Then, the following estimate holds.

\[
\frac{2R(t)-\sqrt{3} + 2\sqrt{3} {\mathcal M}(L^-_{\frac{\pi}{3}}(t))}{2-\sqrt{3}}\leq {\mathcal M}(L^+_{\frac{\pi}{3}}(t)) \leq \frac{2R(t) + \sqrt{3} + {\mathcal M}(L^-_{\frac{\pi}{3}}(t))(2-\sqrt{3})}{2\sqrt{3}}, \quad t \in [0, \infty).
\]
\end{lemma}
\begin{proof}
Note that
\begin{align}\label{par}
\begin{aligned}
R(t) &=\iint_{\mathbb T \times \mathbb R} \langle e^{{\mathrm i}\theta},e^{{\mathrm i}\phi}\rangle f ~ d\theta d\omega \\
&=\iint_{\big(\mathbb{T}\backslash (L^+_{\frac{\pi}{3}}(t)\cup L^-_{\frac{\pi}{3}}(t))\big)\times \mathbb R}
\langle e^{{\mathrm i}\theta},e^{{\mathrm i}\phi}\rangle f~ d\theta d\omega+\iint_{L^+_{\frac{\pi}{3}}(t) \times \mathbb R}
\langle e^{{\mathrm i}\theta},e^{{\mathrm i}\phi}\rangle f~ d\theta d\omega \\
&+
\iint_{L^-_{\frac{\pi}{3}}(t) \times \mathbb R}\langle e^{{\mathrm i}\theta},
e^{{\mathrm i}\phi}\rangle f~ d\theta d\omega. \\
\end{aligned}
\end{align}

\vspace{0.5cm}

\noindent $\bullet$ (Lower bound estimate): We use similar arguments in Lemma \ref{L6.2} to see that 
\begin{align*}
R(t) &\geq-\sin(\frac{\pi}{3})(1-{\mathcal M}(L^+_{\frac{\pi}{3}}(t)) - {\mathcal M}(L^-_{\frac{\pi}{3}}(t))+
\sin\frac{\pi}{3} {\mathcal M}(L^+_{\frac{\pi}{3}}(t))- {\mathcal M}(L^-_{\frac{\pi}{3}}(t)) \\
&=\sqrt{3}  {\mathcal M}(L^+_{\frac{\pi}{3}}(t))  - \frac{\sqrt{3}}{2} -  {\mathcal M}(L^-_{\frac{\pi}{3}}(t)) \Big (1-\frac{\sqrt{3}}{2} 
\Big ).
\end{align*}

\vspace{0.2cm}

\noindent $\bullet$ (Upper bound estimate):  We use again \eqref{par} to obtain
\begin{align*}
R(t) &\leq\sin\frac{\pi}{3}(1-{\mathcal M}(L^+_{\frac{\pi}{3}}(t))- {\mathcal M}(L^-_{\frac{\pi}{3}}(t))+ {\mathcal M}(L^+_{\frac{\pi}{3}}(t))-\sin\frac{\pi}{3} 
{\mathcal M}(L^-_{\frac{\pi}{3}}(t)) \\
&=  \frac{\sqrt{3}}{2} +    {\mathcal M}(L^+_{\frac{\pi}{3}}(t)) \Big (1-\frac{\sqrt{3}}{2} \Big) - \sqrt{3}{\mathcal M}(L^-_{\frac{\pi}{3}}(t)).
\end{align*}  
This yields the desired inequalities.
\end{proof}

\vspace{0.5cm}

Before we state our proposition, we introduce several quantities: for every $T\geq0$ and  $\eta>0$, we define $\beta_{T,\eta}~:[T, \infty) \to \mathbb R$ as a solution of the inhomogeneous Riccati ODE:
\begin{align}
\begin{aligned} \label{RIC}
& \dot{\beta}_{T,\eta}= \frac{K}{4\sqrt{3}}\bigg(-\beta_{T,\eta}^{2}+\frac{\sqrt{3}}{2} \beta_{T,\eta}- \frac{4\sqrt{3}M}{K} -\eta \bigg), \quad t > T, \\
& \beta_{T,\eta}(T)=R(T),
\end{aligned}
\end{align}
and we set the solutions of the following quadratic equation as $r_{\pm}(\eta)$:
\[
x^{2} -\frac{\sqrt{3}}{2}  ~  x + 2\frac{4 \sqrt{3}M}{K} + \eta=0.
\]
More precisely, we have
\[
r_{-}(\eta) := \frac{\sqrt{3}}{4}-\frac{1}{2}\sqrt{\frac{3}{4} - \frac{16\sqrt{3}M}{K} -4\eta}, \qquad 
r_{+}(\eta) := \frac{\sqrt{3}}{4} + \frac{1}{2}\sqrt{\frac{3}{4} - \frac{16\sqrt{3}M}{K} -4\eta}.
\]
Note that if  $\beta_{T,\eta}(T)> r_-(\eta),$ then we have
\begin{equation}\label{bb}
\lim_{t\rightarrow\infty}\beta_{T,\eta}(t)= r_+(\eta).
\end{equation}

\begin{proposition}\label{P6.2} 
Suppose that the assumptions $({\mathcal H}1)$ - $({\mathcal H}4)$ hold. Then, we have
\begin{equation} \label{R-L}
\liminf_{t\rightarrow\infty}R(t)\geq r_+(0) > \kappa > \frac{2}{3}.
\end{equation}
\end{proposition}
\begin{proof}
Let $\varepsilon_{0}$ be a small positive constant satisfying
\begin{equation}\label{lcritical}
\frac{R_0}{2}> r_-(\varepsilon_0).
\end{equation}
Note that the assumption $({\mathcal H}2)$ on $K$ implies
\[ K > \max \frac{64\sqrt{3}M}{3-(\sqrt{3} - 2R_0)^2} \quad  \Rightarrow  \quad \frac{R_0}{2}> r_-(0). \]
Thus, the continuity of $r_-(\cdot)$ with respect to its argument implies the existence of $\varepsilon_0$.  We set  $\eta$ and $d$  be positive numbers satisfying 
\begin{equation} \label{H-19}
 \eta <\varepsilon_{0}, \qquad 
d>\frac{8}{KR_0}\log \left[ \frac{1}{\eta} \Big(1+\frac{\sqrt{3}}{2} \Big) \Big(\frac{\pi}{3}\Big)^{\frac{1}{2}}\sqrt{\iint_{L^-_{\frac{\pi}{3}}(T)\times \mathbb R}f^{2}(\theta,\omega,T) ~ d\theta d\omega} \right]. 
\end{equation}
Here $T$ is the one given by Corollary \ref{C6.2}.\\

By Lemma \ref{L5.2}, we have 
\begin{align*}
\dot{R} &= -\iint_{\mathbb T \times \mathbb R}\sin(\theta-\phi)\omega 
f(\theta,\omega,t) ~ d\theta d\omega+KR \iint_{\mathbb T \times \mathbb R} \sin^{2}(\theta-\phi)
f(\theta,\omega,t) ~ d\theta d\omega \cr
&\geq -M+ \frac{KR}{4} \Big( 1 - {\mathcal M}(L^+_{\frac{\pi}{3}})- {\mathcal M}(L^-_{\frac{\pi}{3}}) \Big).
\end{align*}
Then, we use Lemma \ref{L6.6} to get 
\begin{align*}
{\dot R}&\geq -M + \frac{KR}{4} \Big[  1-\frac{2R + \sqrt{3} + {\mathcal M}(L^-_{\frac{\pi}{3}}) (2-\sqrt{3})}{2\sqrt{3}} -  {\mathcal M}(L^-_{\frac{\pi}{3}}) \Big] \cr
&= -M + \frac{KR}{8} - \frac{KR^2}{4\sqrt{3}} - \frac{KR}{4} \Big (  \frac{2 + \sqrt{3}}{2\sqrt{3}} \Big) {\mathcal M}(L^-_{\frac{\pi}{3}}) \cr
&= \frac{K}{4\sqrt{3}} \Big[  -R^2 + \frac{\sqrt{3}}{2} R - \frac{4\sqrt{3} M}{K} - R \Big(1 +  \frac{\sqrt{3}}{2} \Big)   {\mathcal M}(L^-_{\frac{\pi}{3}})   \Big] \cr
&\geq \frac{K}{4\sqrt{3}} \Big[  -R^2 + \frac{\sqrt{3}}{2} R - \frac{4\sqrt{3} M}{K} - \Big(1 +  \frac{\sqrt{3}}{2} \Big)   {\mathcal M}(L^-_{\frac{\pi}{3}})   \Big].
\end{align*}
Hence, it follows from \eqref{decay2} and \eqref{H-19} that we have 
\begin{equation}\label{beta3}
\dot{R}\geq  \frac{K}{4\sqrt{3}} \Big( -R^2 + \frac{\sqrt{3}}{2} R - \frac{4\sqrt{3} M}{K} -\eta \Big) \hspace{1em} {\rm{in}} \hspace{1em} [T+d,\infty). 
\end{equation}
Moreover, thanks to Corollary \ref{C6.1},
we also see that 
\[R(t) > \frac{R_0}{2}, \qquad \forall~t > 0. \]
We  use \eqref{lcritical} to obtain
\[
 R(T+d)> r_-(\eta). \]
Thus, \eqref{RIC} and \eqref{beta3} yield
\[
 R\geq \beta_{T+d,\eta} \hspace{1em} {\rm{in}} \hspace{1em} [T+d,\infty).
\]
Since
\[
R(T + d) \geq \beta_{T+d,\eta}(T+d)> r_-(\eta),
\]
we get the desired result from \eqref{bb} and the fact that 
$\eta$ was arbitrary small.
\end{proof} 
\begin{remark}
The lower bound for $\liminf_{t \to \infty} R(t)$ in \eqref{R-L} will be improved in Theorem \ref{T3.3} so that $\displaystyle \liminf_{t \to \infty} R(t) \to 1$, as $K \to \infty$.
\end{remark}

\vspace{0.5cm}

\subsection{Proof of Theorem \ref{T3.3}} \label{sec:6.3}
In this subsection, we  will provide the proof for Theorem \ref{T3.3}. For this purpose, for $\kappa, \varepsilon$ and $\delta$ in $(0,1)$ and $t\geq0$, we define the sets:
\begin{align*}
 Y_{\kappa,\varepsilon}(t)&:=\bigg\{ \theta\in\mathbb{T} ~ : ~ \cos(\theta-\phi(t))>\sqrt{1-\varepsilon{}_{\kappa}^{2}}-\varepsilon \bigg\}, \cr
\tilde{Y}_{\delta}(t) &:=\bigg\{\theta\in\mathbb{T} ~ : ~ \cos(\theta-\phi(t))\leq-\delta \bigg\},
\end{align*}
where  for $\kappa \in (0,1)$, $\varepsilon_\kappa$ and $T_\kappa$ are  positive constants satisfying the following relations, respectively:
\[ \varepsilon_{\kappa} :=\frac{\kappa +1}{\kappa^{2}}\frac{M}{K}+\frac{(1-\kappa)}{\kappa} <1 \quad \mbox{and} \quad R > \kappa \quad \mbox{in}~~[T_\kappa, \infty). \]

\begin{lemma}\label{L6.7} 
Suppose that for some positive constants $\kappa$ and  $\delta$ in $(0,1),$ the order parameter $R$ satisfies
\[
 R> \underline{R} \quad {\rm{in}}~~
 [0,\infty), \quad 
 R>\kappa \quad {\rm{in}}~~ 
 [T_{\kappa},\infty), \quad 
\delta<\sqrt{1-\varepsilon_{\kappa}^{2}},
\]
and we also assume that $f_0$ is bounded and $\text{\rm supp}~f_0(\cdot, \omega) \subset [-M, M]$. Then we have
\[
\lim_{t\rightarrow0}||f(t) {\mathbbm 1}_{\tilde{Y}_{\delta}(t)}||_{L^{\infty}(\mathbb{T}\times[-M,M])}=0, 
\]
where the positive constant $\varepsilon_{\kappa}$ is defined in $({\mathcal H}4)$. 
\end{lemma}
\begin{proof}
By Lemma \ref{L6.1}, we have
\begin{equation}\label{h}
||f(T_{\kappa})||_{L^{\infty}(\mathbb{T}\times[-M,M])}\leq ||f_0||_{L^{\infty}(\mathbb T \times [-M, M])} e^{KT_{\kappa}}.
\end{equation}
Choose constants $\theta^{*},$ $\omega^{*},$ and $t^{*}$ satisfying  
\[
\omega^{*}\in[-M,M] \quad \mbox{and} \quad 
\cos(\theta^{*}-\phi(t^{*}))\leq-\delta.
\]
By Corollary \ref{CG1} in the appendix, we know that the characteristic $(\theta(t), \omega^*)$ through $\theta^{*}$
and $\omega^{*}$ at $t^*$ satisfies
\[
\cos(\theta(t)-\phi(t))\leq-\delta, \quad \forall~t \in [T_{\kappa},t^{*}). 
\]
Thus, proceeding as in \eqref{f}, we get
\begin{align*}
\frac{d}{dt}f(\theta(t),\omega^*,t) &= KR(t)\cos(\theta(t)-\phi(t))f(\theta(t), \omega^*,t) \cr
&\leq -K\kappa\delta f(\theta(t),\omega^*,t), \qquad  t \in [T_{\kappa},t^{*}).
\end{align*}
We use \eqref{h} and Gronwall's inequality to obtain
\[
f(\theta^{*},\omega^{*},t^{*})=f(\theta(t^{*}),\omega(t^{*}),t^{*})\leq ||f_0||_{L^{\infty}(\mathbb T \times [-M, M])} e^{KT_{\kappa}}e^{-K\kappa\delta(t^{*}-T_{\kappa})}
\]
Since $(\theta^{*},\omega^{*})$ was an arbitrary point in $\tilde{Y}_{\delta}(t^{*})\times[-M,M],$
we have
\[
||f(t^{*}){\mathbbm 1}_{\tilde{Y}_{\delta}(t^{*})}||_{L^{\infty}(\mathbb T \times[-M,M])}\leq  ||f_0||_{L^{\infty}(\mathbb{T}\times[-M,M])} e^{-K\kappa\delta t^{*}}e^{K(1+\kappa\delta)T_{\kappa}}.
\]
By letting $t^{*}\rightarrow\infty,$ we obtain the desired result.
\end{proof}

\begin{proposition}\label{P6.3} 
Suppose that the following conditions hold.
\begin{enumerate}
\item
The order parameter satisfies
\[
 R> \underline{R} \hspace{1em}  {\rm{in}} \hspace{1em} 
 [0,\infty), \quad 
 R>\kappa  \hspace{1em} {\rm{in}} \hspace{1em} 
 [T_{\kappa},\infty) \quad  \mbox{and} \quad 
 \varepsilon_{\kappa}<1,
\]
for some positive constants $\kappa$ and $\varepsilon_{\kappa}$.

\vspace{0.2cm}

\item
$f_0$ is bounded and $\mbox{supp}~f_0(\cdot, \omega) \subset [-M, M]$.
\end{enumerate}

\vspace{0.5cm}

Then, we have
\[
\lim_{t\rightarrow0}||f(t){\mathbbm 1}_{\mathbb{T}\backslash Y_{\kappa,\varepsilon}(t)}||_{L^{\infty}(\mathbb{T}\times[-M,M])}=0, \quad \forall~~\varepsilon \in (0,\sqrt{1-\varepsilon_{\kappa}^{2}}).
\]
\end{proposition}
\begin{proof}
Let $\varepsilon$ and $\delta$  be positive constants such that 
\[
\varepsilon<\sqrt{1-\varepsilon_{\kappa}^{2}} \qquad \mbox{and} \qquad
\delta<\sqrt{1-\varepsilon_{\kappa}^{2}}-\varepsilon.
\]
By Lemma \ref{L6.7}, for any $\varepsilon^{\prime}>0,$ there
exists $T_{\varepsilon^{\prime}}$ such that
\begin{equation} \label{ep}
 R(t)> \kappa \quad \mbox{and} \quad  \|f_{t}{\mathbbm 1}_{\tilde{Y}_{\delta}(t)} \|_{L^{\infty}(\mathbb{T}\times[-M,M])}<\varepsilon^{\prime}, \quad t \in [T_{\varepsilon^{\prime}},\infty).
\end{equation}
For given $\varepsilon$ and $\kappa$, we define
\begin{equation} \label{FF}
 F_{\kappa}(q) := \kappa K \Big(\sqrt{1-q^{2}}-\varepsilon_{\kappa} \Big)\sqrt{1-q^{2}}, \quad q \in [-1, 1], \quad
 D(\varepsilon,\kappa) :=\frac{2\big(\sqrt{1-\varepsilon_{\kappa}^{2}}-\varepsilon\big)}{F_{\kappa}\bigg(\sqrt{1-\varepsilon_{\kappa}^{2}}-\varepsilon\bigg)}. 
 \end{equation}

Let $\theta^{*}$, $\omega^{*},$ and $t^{*}$ have the property that
\[ \omega^{*}\in [-M,M], \qquad t^{*}\geq T_{\varepsilon^{\prime}}+ D(\varepsilon,\kappa), \quad \cos(\theta^{*}-\phi(t^{*}))\leq\sqrt{1-\varepsilon_{\kappa}^{2}}-\varepsilon.
\]
Let $(\theta(t), \omega^*, t)$ be the forward characteristic through $(\theta^*, \omega^*,t^*)$ at time $t^*$. By Corollary \ref{CG1}, we have
\[
\cos(\theta(t^{*}-d)-\phi(t^{*}-d))\leq-\sqrt{1-\varepsilon_{\kappa}^{2}}+\varepsilon \leq-\delta, \quad \mbox{where} \quad d< D(\varepsilon,\kappa). \]
Proceeding as in \eqref{f}, we get
\begin{align*}
\frac{d}{dt}f(\theta(t),\omega^*,t) &= KR(t)\cos(\theta(t)-\phi(t))f(\theta(t),\omega^*,t) \cr
&\leq K f(\theta(t), \omega^*,t), \quad t \in [T_\kappa, t^*).
\end{align*}
By Gronwall's lemma, we obtain
\[
f(\theta^{*},\omega^{*},t^{*}) =f(\theta(t^{*}),\omega(t^{*}),t^{*}) \leq e^{Kd}f(\theta(t^{*}-d),\omega(t^{*}-d),t^{*}) \leq e^{KD(\varepsilon,\kappa)}\varepsilon^{\prime},
\]
where in the last line we have used \eqref{ep} and the fact that $\theta(t^{*}-d)$ contained in 
$\tilde{Y}_{\delta}(t^{*}-d).$ \newline

Since $(\theta^{*},\omega^{*})$ was an arbitraty point in $\mathbb{T}\backslash Y_{\kappa,\varepsilon}(t^{*})$,
where 
\[
Y_{\kappa,\varepsilon}(t^{*}):=\big\{ \theta\in\mathbb{T} ~ : ~ \cos(\theta-\phi(t^*))>\sqrt{1-
\varepsilon_{\kappa}^{2}}-\varepsilon^{\prime} \big\},
\]
we conclude
\[
||f(t^{*}){\mathbbm 1}_{\mathbb{T\backslash} Y_{\kappa,\varepsilon}(t^{*})}||_{L^{\infty}(\mathbb{T}\times [-M,M])}\leq e^{KD(\varepsilon,\kappa)}\varepsilon^{\prime}.
\]
Additionally, since $t^{*}$ was an arbitrary point in $[T_{\varepsilon^{\prime}}+D(\varepsilon,\kappa),\infty),$ we get
\[
||f(t){\mathbbm 1}_{\mathbb{T\backslash} Y_{\kappa,\varepsilon}(t)}||_{L^{\infty}(\mathbb{T}\times [-M,M])}\leq e^{KD(\varepsilon,\kappa)}\varepsilon^{\prime},
\hspace{1em} \forall t\in[T_{\varepsilon^{\prime}}+D(\varepsilon,\kappa),\infty).
\]
Since $\varepsilon^{\prime}$ is an arbitrary positive number, we have the desired estimate:
\[
\lim_{t\rightarrow\infty}||f(t){\mathbbm 1}_{\mathbb{T\backslash}Y_{\kappa,\varepsilon}(t)}||_{L^{\infty}
(\mathbb{T}\times [-M,M])}=0.
\]
\end{proof}

\vspace{0.5cm}

We are now ready to provide the proof of Theorem \ref{T3.3}. Suppose that the assumptions $({\mathcal H}1)$ to $({\mathcal H}4)$ hold, and we assume
\begin{equation}\label{kaa2}
 \frac{4}{15}\bigg(\sqrt{1-\frac{\sqrt{2}}{2}}-\frac{1}{2}\bigg)>\frac{M}{K}  \quad \mbox{or equivalently} \quad K > \frac{15M}{2(\sqrt{4-2\sqrt{2}}-1)}.
\end{equation}
Then, Theorem \ref{T3.3} follows once we prove the following claims:
\begin{enumerate}
\item
\[ \liminf_{t\rightarrow\infty}R(t)\geq R_\infty :=1+\frac{M}{K}-\sqrt{\frac{M^{2}}{K^{2}}+4\frac{M}{K}}. \]
\item
There exists a time dependent interval $L_{\infty},$ centered at $\phi(t),$ with constant  width 
\[ 2 ~ \arccos\bigg(\sqrt{1-\bigg[\frac{M}{K}\frac{(1+ R_\infty)}{R_\infty^{2}}+\frac{1-R_\infty}{R_\infty}\bigg]^{2}}\bigg)+\varepsilon,\]
such that
\[
 \lim_{t\rightarrow\infty}||f(t){\mathbbm 1}_{\mathbb{T}\backslash L_{\infty}(t)}||_{L^{\infty}(\mathbb{T\times}[-M,M])}=0, \quad \mbox{for every $\varepsilon$ in $(0,1).$}
\]
As $K\rightarrow \infty$, $R_\infty$ tends to 1 and
the width of $L_{\infty}$ can be made arbitrarily small. 
\end{enumerate}

\vspace{1cm}

\noindent {\it Proof of claim:} Note that the second item (2) follows from the item (1) by apply Proposition \ref{P6.3}. To prove the item (1), by the assumptions $({\mathcal H}1)$ - $({\mathcal H}3)$ and Corollary \ref{C6.1}, we have 
\[  ||f_0||_{L^{\infty}(\mathbb{T}\times[-M,M])} < \infty \quad \mbox{and} \quad R(t)>\frac{R_0}{2}\hspace{1em}\forall t \in [0,\infty).  \]
Second, Proposition \ref{P6.2} and $({\mathcal H}4)$ yield
\[
R > \frac{2}{3}, \quad \mbox{in $[T,\infty)$~~for some $T\geq0$}.
\]
Now, we improve this lower bound. Let ${\bar \varepsilon}_{0}>0$ be sufficiently small so that
\begin{equation}\label{econ3}
 \frac{4}{15}\bigg(\sqrt{1-\frac{\sqrt{2}}{2}-{\bar \varepsilon}_{0}}-\frac{1}{2}\bigg)>\frac{M}{K}.
\end{equation}
Such ${\bar \varepsilon}_{0}$ exists by assumption \eqref{kaa2}. Moreover, let $\eta$ be 
an arbitrary positive constant in $(0,1)$. By Proposition \ref{P6.3} and the assumption $({\mathcal H}4)$,
\[
 \lim_{t\rightarrow0}||f(t){\mathbbm 1}_{\mathbb{T}\backslash Y_{\frac{2}{3},\eta}(t)}||_{L^{\infty}(\mathbb{T}\times[-M,M])}=0.
\]
From this, we also see 
\begin{equation}\label{F-1}
\lim_{t \to \infty} \mathcal{M}\bigg(\bigg\{\theta\in\mathbb{T} ~ : ~ \cos(\theta-\phi(t))\leq\sqrt{1-\bigg(\frac{2/3+1}{(2/3)^{2}}\frac{M}{K}+\frac{1-2/3}{2/3}\bigg)^{2}}-\eta \bigg\} \bigg) = 0,
\end{equation}
or equivalently, 
\[
\lim_{t \to \infty} \mathcal{M} \bigg(\bigg\{\theta\in\mathbb{T}:~ \cos(\theta-\phi(t))\leq\sqrt{1-\bigg(\frac{15}{4}\frac{M}{K}+\frac{1}{2}\bigg)^{2}}-\eta \bigg\} = 0.
\]
Then, using \eqref{econ3} in the previous limit, we deduce
\[
 \lim_{t \to \infty} \mathcal{M}(A(t)) = 0,
\]
where
\begin{equation}\label{F-2}
 A(t) :=\bigg\{\theta\in\mathbb{T}:~ \cos(\theta-\phi(t))\leq\sqrt{\frac{\sqrt{2}}{2}+ {\bar \varepsilon}_{0}}-\eta \bigg\}.
\end{equation}
Since
\begin{align}\label{rfinal}
\begin{aligned}
R(t) & =\iint_{(\mathbb{T}\backslash A(t)) \times \mathbb R}\langle e^{{\mathrm i}\phi(t)},e^{{\mathrm i}\theta}\rangle f ~ d\theta d\omega+ \iint_{A(t) \times \mathbb R}\langle e^{{\mathrm i}\phi(t)},e^{{\mathrm i}\theta}\rangle f ~ d\theta d\omega\\
 & \geq \Big(1-\mathcal{M}(A(t)) \Big)\bigg(\sqrt{\frac{\sqrt{2}}{2}+{\bar \varepsilon}_{0}}-\eta \bigg)-\mathcal{M}(A(t)),
\end{aligned}
\end{align}
and $\eta$ can be made arbitrarely small, we obtain
\begin{equation}\label{r1}
R_*  := \liminf_{t\rightarrow\infty} R(t) \geq\sqrt{\frac{\sqrt{2}}{2}+{\bar \varepsilon}_{0}}. 
\end{equation}
Then, there exists $T_{0}$ such that 
\[
 R(t) \geq \sqrt{\frac{\sqrt{2}}{2}} \hspace{1em} \forall t\in[T_{0},\infty). 
\]
We next claim:
\[ R_* \geq1+\frac{M}{K}-\sqrt{\frac{M^{2}}{K^{2}}+4\frac{M}{K}} =: \alpha_-(M, K).
\]
Suppose not, i.e.,
\[
R_* <1+\frac{M}{K}-\sqrt{\frac{M^{2}}{K^{2}}+4\frac{M}{K}}.
\]
Then, by \eqref{r1} we have
\begin{equation}\label{kbound}
R_* \in\bigg(\sqrt{\frac{\sqrt{2}}{2}},~1+\frac{M}{K}-\sqrt{\frac{M^{2}}{K^{2}}+4\frac{M}{K}}\bigg). 
\end{equation}
Since the roots of the polynomial
\[
 Q(x)=x^{2}-2\bigg(1+\frac{M}{K}\bigg)x+1-2\frac{M}{K},
\]
are
\[
\alpha_-(K, M) = 1+\frac{M}{K}-\sqrt{\frac{M^{2}}{K^{2}}+4\frac{M}{K}},  \qquad 
\alpha_+(K, M) :=1+\frac{M}{K}+\sqrt{\frac{M^{2}}{K^{2}}+4\frac{M}{K}}, 
\]
we have
\[
 Q(x)>0 \hspace{1em} {\rm{in}} \hspace{1em} ( -\infty, \alpha_-(K, M) ).
\]
It follows from \eqref{kbound} that 
\[
 0< R_*^{2}-2(1+\frac{M}{K}) R_*+1-2\frac{M}{K}, \quad \mbox{or equivalently} \quad 
 0<\frac{1}{2} R_*^{2}- \Big(1+\frac{M}{K} \Big) R_*+\frac{1}{2}-\frac{M}{K}.
\]
This yields 
\begin{equation}\label{ineq1}
\frac{1}{2}(R_*^{2}-1)<-R_*(1-R_*)-\frac{M}{K}(1+R_*). 
\end{equation}
By \eqref{kaa2} and \eqref{kbound} we have
\[
0<R_*(1-R_*)+\frac{M}{K}(1+R_*)<\frac{1}{4}+2\frac{M}{K}<1. 
\]
We again use \eqref{ineq1} to see
\begin{align*}
\bigg(\sqrt{\frac{\sqrt{2}}{2}}\bigg)^{4}\big(R_*^{2}-1)&=\frac{1}{2}(R_*^{2}-1) <-R_*(1-R_*)-\frac{M}{K}(1+R_*)\\ 
&<-\bigg[R_*(1-R_*)+\frac{M}{K}(1+R_*)\bigg]^{2}.
\end{align*}
By the above expression and  \eqref{kbound},
\[
R_*^{4}\big(R_*^{2}-1)<\bigg(\sqrt{\frac{\sqrt{2}}{2}}\bigg)^{4}\big(R_*^{2}-1)<-\bigg[ R_*(1-R_*)+\frac{M}{K}(1+R_*)\bigg]^{2}. 
\]
Hence, we have
\begin{equation}\label{rec1}
R_*<\sqrt{1-\bigg[\frac{M}{K}\frac{(1+R_*)}{R_*^{2}}+\frac{1-R_*}{R_*}\bigg]^{2}}. 
\end{equation}
We set 
\[
R_{*,\varepsilon^{\prime}} := R_*-\varepsilon^{\prime}.
\]
By construction, there exists $T_{\varepsilon^{\prime}}$ such that
\[
 R(t) \geq R_{*,\varepsilon^{\prime}} \hspace{1em} {\rm{in}} \hspace{1em} [T_{\varepsilon^{\prime}},\infty).
\]
Moreover, it follows from Proposition \ref{P6.3} that  we have
\[
\lim_{t \to \infty} \mathcal{M} \big(\mathbb{T}\backslash Y_{R_{*,\varepsilon^{\prime}},\varepsilon}(t)\big) = 0,  \quad \mbox{for any $\varepsilon>0$}.
\]
Then, by the same arguments as in \eqref{rfinal}, we get
\[
R(t)\geq\big(1-\mathcal{M}(\mathbb{T}\backslash Y_{R_{*,\varepsilon^{\prime}},\varepsilon}(t)\big)
\bigg[\sqrt{1-\bigg[\frac{M}{K}\frac{(1+R_{*,\varepsilon^{\prime}})}{R_{*, \varepsilon^{\prime}}^{2}}+\frac{1-R_{*,\varepsilon^{\prime}}}{R_{*,\varepsilon^{\prime}}}\bigg]^{2}}-\varepsilon\bigg]-\mathcal{M}(\mathbb{T}\backslash Y_{R_{*,\varepsilon^{\prime}},\varepsilon}(t)).\]
Thus, we have
\[ R_* \geq\sqrt{1-\bigg[\frac{M}{K}\frac{(1+R_{*,\varepsilon^{\prime}})}{R_{*,\varepsilon^{\prime}}^{2}}+\frac{1-R_{*,\varepsilon^{\prime}}}{R_{*,\varepsilon^{\prime}}}\bigg]^{2}}-\varepsilon. 
\]
Since $\varepsilon$ and $\varepsilon^{\prime}$ can be made arbitrarily small, using inequality \eqref{rec1} and the above expression we get
\[
R_* \geq\sqrt{1-\bigg[\frac{M}{K}\frac{(1+R_*)}{R_*^{2}}+
\frac{1-R_*}{R_*}\bigg]^{2}}> R_*, 
\]
which yields a contradiction. Therefore, we have
\[
R_* \geq1+\frac{M}{K}-\sqrt{\frac{M^{2}}{R_*^{2}}+4\frac{K}{M}}. 
\]
This completes the proof of Theorem \ref{T3.3}.

\begin{remark}
We now briefly discuss how the lower bound on $\kappa >\frac{2}{3}$ in $(\mathcal H3)$ was determined. In the above proof, we used the property
\[
A(t) \subset \bigg\{\theta\in\mathbb T: \cos(\theta - \phi(t)) \leq \sqrt{1 - \Big(\frac{\kappa + 1}{\kappa^2}\frac{M}{K} + \frac{1-\kappa}{\kappa}\Big)^2} - \eta \bigg\}
\]
in \eqref{F-1} and \eqref{F-2}, which is equivalent to
\begin{equation}\label{F-3}
1 - \Big(\frac{\kappa + 1}{\kappa^2}\frac{M}{K} + \frac{1-\kappa}{\kappa}\Big)^2 > \frac{\sqrt{2}}{2} + \bar\varepsilon_0.
\end{equation}
Thus, if the following inequality holds,
\begin{equation}\label{F-4}
1 - \Big(\frac{1-\kappa}{\kappa}\Big)^2 > \frac{\sqrt{2}}{2},
\end{equation}
\eqref{F-3} is satisfied for small $\bar\varepsilon_0$ and sufficiently large $K$(depending on $\bar\varepsilon_0$).
Since the inequality \eqref{F-4} is equivalent to
\[
\sqrt{2} - \sqrt{2-\sqrt{2}} < \kappa < \sqrt{2} + \sqrt{2-\sqrt{2}},
\]
and $\sqrt{2} - \sqrt{2-\sqrt{2}} \approx 0.6488$, we choose $\kappa > \frac{2}{3}$ for the simplicity.
\end{remark}
\section{Conclusion} \label{sec:7}
\setcounter{equation}{0}
In this paper, we have presented several results on the asymptotic dynamics of the Kuramoto-Sakaguchi equation which is  obtained from the Kuramoto model in the mean-field limit. For a large ensemble of Kuramoto oscillators, it is very expensive to study the dynamics of the oscillators directly via the Kuramoto model. So, from the beginning of the study on Kuramoto oscillators, the corresponding mean-field model, namely the Kuramoto-Sakaguchi equation has been widely used in the physics literature for the phase transition phenomena of large ensembles of Kuramoto oscillators. For example, Kuramoto himself employed a self-consistent theory based on the linearized Kuramoto-Sakaguchi equation, to derive a critical coupling strength for the phase transition from disordered states to partially ordered states (see \cite{A-B}). However, existence of steady states and chimera states, as well as their nonlinear stability are still far from complete understanding. In this long paper, we have studied phase concentration in a large coupling regime, for a large ensemble of oscillators. First, in the identical natural frequency case, we showed that mass of the ensemble concentrates exponentially fast at the average phase. In particular, the mass on each interval containing the average phase is nondecreasing over time, whereas the mass outside the interval decays to zero asymptotically. This illustrates the formation of a point cluster for the large ensemble of Kuramoto oscillators, which is a stable solution. It is interesting to note that, on the other hand, the Kuramoto model allows the unstable bi-polar state as an asymptotic pattern. Second, for the non-identical natural frequencies, i.e., the general case, we showed that the phases of a large ensemble of Kuramoto oscillators will aggregate inside a small interval around the average phase as the coupling strength increases. This is a similar feature as in the finite-dimensional Kuramoto model. Our third result is a quantitative lower bound for the amplitude order parameter. From a series of technical lemmata, we obtain an asymptotic formula for the amplitude order parameter in a large coupling strength regime, which also shows that a point cluster pattern arises as the coupling strength becomes sufficiently large. There are still lots of issues to be resolved on the large-time dynamics of the Kuramoto-Sakaguchi equation. To name a few, we mention three outstanding problems. First, we have not yet shown the existence of stationary solutions for the Kuramoto-Sakaguchi equation. Thus, our present results can be a first foot step toward this direction. Second, our estimates on the ensemble of Kuramoto oscillators with distributed natural frequencies are strongly relying on the large coupling strength. In particular, we have not optimized the size of coupling strength. Thus, one interesting question is to find the critical coupling strength for the phase transition from the partially ordered state to the fully ordered state (complete synchronization). Third, it will be also interesting to investigate the intermediate regime where the coupling strength is not too small nor too large, especially, regarding existence and stability of partially synchronized states. These issues will be addressed in future work.

\newpage

\appendix 

\section{Otto calculus} \label{App-A}
\setcounter{equation}{0}
In this section, we review the Otto calculus dealing with gradient flows on the Wasserstein space, and explain how the K-S equation can be regarded as a gradient flow on the Wasserstein space. 

\subsection{The K-S equation as a  gradient flow} We now formulate the gradient flow for the potential $V_k$ via the Otto calculus, and see that it coincides with the K-S equation \eqref{KM-K-I}, equivalently \eqref{D-0-0} for the identical oscillators case. This is a rather standard procedure (see \cite{Vi}, for example). For this, we first recall the Otto calculus introduced in \cite{Otto}, which gives a formal Riemannian metric on the space of absolutely continuous probability measures  $\mathbb P_{ac}(\mathbb M)$ on a Riemannian manifold ${\mathbb M}$.  \newline

Consider two curves $\rho^1, \rho^2: (-\varepsilon, \varepsilon) \to  \mathbb P_{ac}({\mathbb M})$ with  the common value $\rho$ at $t=0$, i.e  $\rho^1(0)=\rho^2(0) =\rho$.  Assume that they are differentiable and the Riemmanian product between the time derivatives  $\dot \rho^1 (0), \dot \rho^2(0)$ is given by 
\begin{align*}
 \langle \dot \rho^1 (0), \dot \rho^2(0)\rangle_{W_2} = \int_{\mathbb M} \langle \nabla \varphi^1\big|_{t=0} , \nabla \varphi^2 \big|_{t=0} \rangle \rho \,   ~  ~  d \mbox{vol}, 
\end{align*}
where  the bracket $\langle \cdot, \cdot \rangle$ is the Riemannian product on the underlying space of $\mathbb M$, and the functions $\varphi^1$ and  $\varphi^2$ are determined by solving the continuity equation:
\begin{align*}
 \frac{\partial }{\partial t} \rho ^i + {\rm div} (\rho^i \nabla \varphi^i ) =0, \quad i=1,2. 
\end{align*}
With respect to this metric $\langle \cdot, \cdot \rangle_{W_2}$, the gradient of a given functional $F: {\mathbb P}_{ac}({\mathbb M}) \to \mathbb R,$ can be considered as a vector field, denoted by ${\rm grad}_\rho F$ on $\mathbb M$, such that for a one-parameter differentiable family $\tau \mapsto g_\tau \in {\mathbb P}_{ac}({\mathbb M})$ with $g_0 = \rho$, satisfies
\begin{align*}
 \frac{d}{dt}\bigg|_{\tau=0} F (g_{\tau})  =\int_{\mathbb M} \langle \nabla \varphi\big|_{\tau=0}, {\rm grad}_\rho F \rangle \rho \, ~  d \mbox{vol},
\end{align*}
where the vector field $\nabla \varphi$ solves the continuity equation:
\begin{align}\label{eq: g}
 \frac{\partial g_t}{\partial t} + {\rm div} (g_t \nabla \varphi) = 0. 
\end{align}
Then, the gradient flow of $F$ is a one-parameter family $t \mapsto \rho_t \in {\mathbb P}_{ac}({\mathbb M})$ satisfying 
\begin{align}\label{eq: grad flow}
 \frac{\partial}{\partial t} \rho_t + {\rm div} \left(\rho_t (-{\rm grad}_{\rho}F) \right) = 0. 
\end{align}
The equation \eqref{eq: grad flow} can be written as a weak form:
\[
\frac{d}{dt}\bigg|_{t=0}\int_{\mathbb M}\zeta \rho_{t} ~ d\theta=- \int_{\mathbb M}  \langle \nabla\zeta ,  {\rm grad}_{\rho} F \rangle \, \rho_t d \theta, \quad \forall \zeta \in C_c^\infty({\mathbb M}).
\]
We now verify that the equation \eqref{D-0-0} is the gradient flow in the above sense, of the potential $V_k$ from \eqref{eq:V}. 
In our case the underlying Riemannian manifold ${\mathbb M}$ is $\mathbb T$, with the metric $ ~ d\theta^2$.
 Given $\rho$ in $\mathbb{P}(\mathbb{T}),$ recall the notation for $J$ specified in \eqref{eq:J f}. Then, we have
\begin{align}\label{eq: grad V}
{\rm grad}_{\rho}V_k:=- K \nabla (\sigma\cdot J). 
\end{align}
where the inner product is the Euclidean one in $\mathbb T \subset \mathbb R^2$, and $\nabla$ is the gradient of the Riemannian manifold $\mathbb T$; of course, $\nabla =\frac{d}{ ~ d\theta}$, but we keep the notation $\nabla$ to be more consistent with the general formulation.  
In order to see \eqref{eq: grad V},   we note that for each one-parameter differentiable family, $t \mapsto g_t \in {\mathbb P}_{ac}({\mathbb M})$  satisfying \eqref{eq: g} with $g_0 =\rho$, the derivative of $V_k$ given in \eqref{eq:V} is 

\begin{align*}
\frac{d}{dt}\bigg|_{t=0}V_k(g_{t}) &=& - K  J \cdot \frac{d}{dt}  \bigg|_{t=0}\int_\mathbb T \sigma g_{t} ~ d\theta 
=-K\frac{d}{dt}\bigg|_{t=0}\int_\mathbb T (\sigma\cdot J)g_{t} ~ d\theta \cr &=& - \int_\mathbb T  \langle \nabla\varphi, K\nabla(\sigma\cdot J)\rangle \rho  ~ d\theta. 
\end{align*}
This yields \eqref{eq: grad V}.  On the other hand, note that since $J =Re^{i\phi},$
\begin{equation} \label{NEW-1}
\sigma(\theta)\cdot J =R\cos(\theta-\phi),
\end{equation}
Therefore, 
\[
{\rm grad}_\rho V_k = -K  \nabla (\sigma(\theta)\cdot J)=K R\sin(\theta-\phi). 
\]
We substitute this into \eqref{eq: grad flow} to get the gradient flow of $V_k$ as the 
 one-parameter family $t \in [0, \varepsilon) \mapsto \rho_t \in \mathbb P(\mathbb T)$
satisfying 
\[
\partial_{t}\rho= \partial_\theta \left(\rho K R \sin (\theta-\phi) \right),
\]
which is the same as \eqref{D-0-0}, verifying the gradient flow structure. 
Moreover, this immediately implies  
\begin{align*}
 \frac{d}{dt}\bigg|_{t=0}V_k(\rho_t) & = - (KR)^2 \int_\mathbb T  \sin(\theta-\phi)^2 \rho  ~ d\theta. 
\end{align*}
Of course, this expression can be directly obtained from the formula of $V_k$ \eqref{eq: V simple} and the definitions of the order parameters $R$ and $\phi$. 

Additionally, using the Riemannian inner product on $(\mathbb{P}(\mathbb{T}),W_{2}),$
we have that the metric slope is given by
\[
 \int_{\mathbb{T}} K|\nabla(\sigma\cdot J)|^2 \rho ~  
 d\theta=  KR^2 \int_{\mathbb{T}}  \sin^{2}(\theta-\phi) \rho 
  ~  d\theta.
\]
\subsection{The Hessian of the potential $V_k$} In this subsection, we explicitly compute the Hessian of the potential $V_k$ via the Otto calculus. By direct calculations, the Hessian of $V_k$ is given by 
\begin{align}\label{eq: Hess}
\mbox{\ensuremath{\langle}Hess}_{\rho}\nabla\varphi,\nabla\varphi\rangle=-K\int_{\mathbb T} \nabla\varphi D^{2}(\sigma\cdot J)\nabla\varphi \rho(\theta)
 ~ d\theta
-K\bigg|\int_{\mathbb T}\nabla\varphi\rho ~ d\theta\bigg|^{2}.
\end{align}
Using the Jensen inequality, this expression can be bounded from below
by 
\[
-K\int_{\mathbb T}\nabla\varphi\big[D^{2}(\sigma\cdot J)+1\big]\nabla\varphi \rho(\theta) ~ d\theta.
\]
Consequently, the functional $V_k(\rho)$ is $\lambda$-convex with $\lambda=-2K,$ as a functional on the formal Riemannian space $\mathbb P (\mathbb T)$.  In order to obtain \eqref{eq: Hess}, we proceed as follows. Let, 
$t \mapsto g_t \in \mathbb P (\mathbb T),$ be the  differentiable one-parameter family, with $g_0=\rho$, associated with $\varphi : \mathbb T \times \mathbb R \to \mathbb R$. Moreover, suppose 

\begin{align}\label{eq: geodesic}
\begin{dcases}
 \frac{d}{dt}\int_{\mathbb T}\zeta g_{t} ~ d\theta
  =\int_{\mathbb T} \langle \nabla\zeta, \nabla\varphi \rangle g_t ~ d\theta, \quad \forall \zeta \in C^\infty(\mathbb T),   & \text{ } \\
\partial_{t}\varphi+\frac{\mid\nabla\varphi\mid^{2}}{2}=0.     & \text{}.
\end{dcases}
\end{align}
This family is a Riemannian geodesic in the formal Riemannian manifold $\mathbb P (\mathbb T),$ in the sense of the Otto calculus. 
Then, we compute 
\[
\frac{d^{2}}{dt^{2}}\bigg|_{t=0}V_k(g_t) 
 =\frac{d^{2}}{dt^{2}}\bigg|_{t=0}\bigg[K-\frac{K}{2} |J|^{2}\bigg] =-K\frac{d}{dt}\bigg|_{t=0}\int_{\mathbb T} \langle \nabla\varphi, \nabla(\sigma\cdot J)\rangle g_{t} ~ d\theta . 
\]
Here, the last expression can be computed as 
\begin{align*}
 & K\int_{\mathbb T} \big\langle \nabla \left(\frac{1}{2}\mid \nabla\varphi\mid^{2}\right),  \nabla(\sigma\cdot J)\big\rangle \rho  ~ d\theta  - K \int_{\mathbb T} \big\langle \nabla\varphi ,  \nabla\left[ \sigma \cdot \frac{d}{dt}\bigg|_{t=0}  J \right]\big\rangle  \rho   ~ d\theta \\
 &\hspace{1cm} -K\int_{\mathbb T} \bigg\langle \nabla \left\langle \nabla \varphi , \nabla\big(\sigma \cdot J \big) \right\rangle, \nabla \varphi \bigg\rangle \rho  ~ d\theta\\
 &\hspace{1cm} =: {\mathcal I}_{11}  + {\mathcal I}_{12} + {\mathcal I}_{13}.
\end{align*}
 where we used the system \eqref{eq: geodesic} in the first integral ${\mathcal I}_{11}$.\\
 
\noindent $\bullet$ (Estimate of ${\mathcal I}_{11} + {\mathcal I}_{13}$): We apply the identity
\begin{align*}
-\nabla \varphi \nabla^2 \varphi \nabla\Psi+\langle \nabla\varphi, \nabla[\langle \nabla\varphi, \nabla \Psi\rangle]\rangle=\nabla \varphi D^{2}\Psi \nabla \varphi.
\end{align*}
to obtain
\begin{equation} \label{NE-1}
 {\mathcal I}_{11} + {\mathcal I}_{13} =  -K\int_{\mathbb T}  \nabla \varphi [D^2\big(\sigma \cdot J\big)] \nabla \varphi  \rho  ~ d\theta. 
\end{equation}
\noindent $\bullet$ (Estimate of ${\mathcal I}_{12}$): In order to simplify ${\mathcal I}_{12}$ we use the identity
\begin{align*}
e\cdot\frac{d}{dt}\bigg|_{t=0}J &=\frac{d}{dt}\bigg|_{t=0}e\cdot\int_{\mathbb T}\sigma g ~ d\theta=\int_{\mathbb T}(e\cdot\sigma)\partial_{t}g_{t}(\theta,0) ~ d\theta\\
 &=\int_{\mathbb T}\langle\nabla(e\cdot\sigma),\nabla\varphi\rangle\rho ~ d\theta =e\cdot\int_{\mathbb T}\nabla\varphi\rho ~ d\theta,
\end{align*}
 which holds for any $e$ in $\mathbb{R}^{2}$. In such an identity, we have used 
the fact that for any vector $v$ in $\mathbb{R}^{2},$
  $\nabla (v\cdot\sigma)$
  is the orthogonal projection of $v$
  in the orthogonal subspace to $\sigma(\theta).$\\
Consequently, by the same reason, ${\mathcal I}_{12}$ can be computed as 
\begin{equation} \label{NE-2}
{\mathcal I}_{12} = -K\int_{\mathbb T}\langle\nabla\varphi,\nabla\bigg[\sigma\cdot\int_{\mathbb T}\nabla\varphi\rho ~ d\theta_* \bigg] \rangle\rho ~ d\theta = -K\bigg|\int_{\mathbb T}\nabla\varphi\rho ~ d\theta\bigg|^{2}.
\end{equation}

Finally, we combine \eqref{NE-1} and \eqref{NE-2} to obtain \eqref{eq: Hess}.

\vspace{1cm}

\section{Proof of Lemma \ref{L6.2}} \label{App-B}
\setcounter{equation}{0}
Below, we study the Lipschitz continuity of $R, {\dot R}$ and $\mathcal{M}(L^+_{\gamma}(t))$. \newline
 
\subsection{Lipschitz continuity of $R$} It follows from Lemma \ref{L5.2} and the facts
\[ f = 0 \quad \mbox{for}~ \omega \not \in [-M, M], \qquad R \leq 1 \]
that  ${\dot R}$ is uniformly bounded:
\begin{equation} \label{R-1}
|{\dot R}| \leq M \iint_{\mathbb T \times [-M, M]}  f(\theta, \omega, t) ~ d\theta d\omega + K R  \int_\mathbb T \rho(\theta, t) ~ d\theta \leq M + K.
\end{equation}
This yields the Lipschitz continuity of $R$. 

\subsection{Lipschitz continuity of ${\dot R}$}:  Note that the a priori condition $\min_{0 \leq t \leq T} R(t) \geq \underline{R} > 0$ and the continuity of $R$ yield that there exists a positive constant $\eta$ such that 
\begin{equation}\label{R-2}
R(t)>\frac{\underline{R}}{2},   \quad t \in [0,T+\eta).
\end{equation}
We next show that ${\dot R}$ is Lipschitz continuous on $[0, T+ \eta)$. For this, it suffices to show that ${\ddot R}$ is uniformly bounded. Since 
\begin{equation} \label{R-2-1}
 \frac{d^2}{dt^2} R^2 = 2 ({\dot R})^2 + 2 R {\ddot R}, 
\end{equation} 
once we can show that $\frac{d^2}{dt^2} R^2$ is bounded, then, it follows from \eqref{R-1} and \eqref{R-2} that ${\ddot R}$ is bounded.  \newline

\noindent $\bullet$ Step A (uniform boundedness of  $\frac{d^2}{dt^2}R^2$): We set 
 \[
J(t) :=\iint_{\mathbb T \times \mathbb R} e^{{\mathrm i}\theta^{*}}f(\theta^{*},\omega,t) ~ d\theta^{*} d\omega, \qquad \sigma(\theta) :=e^{{\mathrm i}\theta}.
 \]
Then, we have
\[ \sigma(\theta^*) \cdot \sigma(\theta) =\cos(\theta-\theta^{*}), \quad  \sigma(\theta) \cdot J(t) =\iint_{\mathbb T \times \mathbb R} \cos(\theta-\theta^{*})f ~ d\theta^{*} d\omega.
\]
This yields
\begin{equation} \label{R-3}
\partial_{\theta}(\sigma(\theta)\cdot J)=-\iint_{\mathbb T \times \mathbb R} \sin(\theta-\theta^{*})f ~ d\theta^{*} d\omega. 
\end{equation}
Note that the K-S equation \eqref{K-S} can be written on $\mathbb{T\times R}$ as
\begin{equation}\label{R-3-1}
\partial_{t}f+\mbox{div}_{\theta}(\omega \sigma^{\perp} f)+K\mbox{div}_{\theta}(\partial_{\theta}(\sigma\cdot J) f)=0. 
\end{equation}
Where ${\rm{div}}_{\theta}$ and $\partial_{\theta}$
denote the divergence operator and gradient operator
on $\mathbb{T}\subset \mathbb{R}^{2},$ endowed with angle metric, and for each $\sigma$ in $\mathbb T$, $\sigma^{\perp}$ denotes the vector obtained by rotating $\sigma$ by $\frac{\pi}{2}$ radians counterclockwise. We will denote the Laplacian on $\mathbb{T}$ by $\partial_{\theta \theta}$. We use \eqref{R-3}, \eqref{R-3-1} and  
  \[
   R^{2}=\iint_{\mathbb T \times \mathbb R} (\sigma\cdot J) f(\theta, \omega, t) ~ d\theta d\omega,
  \]
to obtain
\begin{equation}\label{rto2}
\frac{dR^{2}}{dt} =2 \Big[ \iint_{\mathbb T \times \mathbb R} \partial_{\theta}(\sigma\cdot J)\cdot \omega \sigma^{\perp} f ~ d\theta d\omega+ K  \iint_{\mathbb T \times \mathbb R} (\partial_{\theta}(\sigma\cdot J))^{2}f ~ d\theta d\omega \Big]. 
\end{equation}
Then, we claim:  
\begin{align}
\begin{aligned} \label{Z}
\frac{d^2}{dt^2} R^2 &=4K^{2}\iint_{\mathbb T \times \mathbb R}\partial_{\theta}(\sigma\cdot J)\partial_{\theta\theta}(\sigma\cdot J)\partial_{\theta}(\sigma\cdot J) ~ f ~ d\theta d\omega \\
&+4K\iint_{\mathbb T \times \mathbb R}\omega \sigma^{\perp} \cdot (\partial_{\theta\theta}\sigma\cdot J)\partial_{\theta}(\sigma\cdot J)f ~ d\theta d\omega+4K^{2}\bigg|\iint_{\mathbb T \times \mathbb R}\partial_{\theta}(\sigma\cdot J)f ~ d\theta d\omega\bigg|^{2} \\
&+4K\iint_{\mathbb T \times \mathbb R}\partial_{\theta}(\sigma\cdot J)f ~ d\theta d \omega \cdot \iint_{\mathbb T \times \mathbb R}\omega \sigma^{\perp} f ~ d\theta d\omega \\
&+2K\iint_{\mathbb T \times \mathbb R}\partial_{\theta}(\sigma\cdot J)\partial_{\theta\theta}(\sigma\cdot J)\omega \sigma^{\perp} f ~ d\theta d\omega \\
&+2\iint_{\mathbb T \times \mathbb R}\omega \sigma^{\perp} \partial_{\theta\theta}(\sigma\cdot J)\omega\sigma^{\perp} f ~ d\theta d\omega \\
&+2K\iint_{\mathbb T \times \mathbb R}\omega \sigma^{\perp}f ~ d\theta d\omega\cdot\iint_{\mathbb T \times \mathbb R}\partial_{\theta}(\sigma\cdot J)f ~ d\theta d\omega \\
&+2\bigg|\iint_{\mathbb T \times \mathbb R}\omega \sigma^{\perp} f ~ d\theta d\omega \bigg|^{2}.
\end{aligned}
\end{align}

\noindent {\it Proof of claim \eqref{Z}}: It follows from \eqref{rto2} that we have
\begin{equation} \label{Z-1}
\frac{d^2}{dt^2} R^2=2\frac{d}{dt}\bigg[\iint_{\mathbb T \times \mathbb R}\partial_{\theta}(\sigma\cdot J)\omega \sigma^{\perp} f ~ d\theta d\omega+K\iint_{\mathbb T \times \mathbb R}(\partial_{\theta}(\sigma\cdot J))^{2}f ~ d\theta d\omega\bigg].
\end{equation}
$\bullet$ (The second integral in the R.H.S. of \eqref{Z-1}): By direct calculations, we have
\begin{align}
\begin{aligned} \label{Z-2}
& \frac{d}{dt} \iint_{\mathbb T \times \mathbb R}|\partial_{\theta}(\sigma\cdot J)|^{2}f ~ d\theta d\omega \\
& \hspace{1cm} =K\iint_{\mathbb T \times \mathbb R}\partial_{\theta}|\partial_{\theta}(\sigma\cdot J)|^{2}\partial_{\theta}(\sigma\cdot J)f ~ d\theta d\omega+\iint_{\mathbb T \times \mathbb R}\omega\sigma^{\perp}\partial_{\theta}|\partial_{\theta}(\sigma\cdot J)|^{2}f ~ d\theta d\omega\\
 & \hspace{1cm}+2\iint_{\mathbb T \times \mathbb R}\partial_{\theta}(\sigma\cdot J)\partial_{\theta}(\sigma\cdot\partial_{t}J)f ~ d\theta d\omega\\
 & \hspace{1cm} =2K\iint_{\mathbb T \times \mathbb R}\partial_{\theta}(\sigma\cdot J)\partial_{\theta\theta}(\sigma\cdot J)\partial_{\theta}(\sigma\cdot J)f ~ d\theta d\omega \\
 & \hspace{1cm} +2\iint_{\mathbb T \times \mathbb R}\omega \sigma^{\perp} \partial_{\theta\theta}(\sigma\cdot J)\partial_{\theta}(\sigma\cdot J)f ~ d\theta d\omega\\
 & \hspace{1cm}+2K\iint_{\mathbb T \times \mathbb R}\partial_{\theta}(\sigma\cdot J)\partial_{\theta}\Big[\sigma\cdot\big(\iint_{\mathbb T \times \mathbb R}\partial_{\theta}(\sigma\cdot J)fd\theta d\omega\big)\Big]f ~ d\theta d\omega\\
 & \hspace{1cm}+2\iint_{\mathbb T \times \mathbb R}\partial_{\theta}(\sigma\cdot J)\partial_{\theta}\Big[\sigma\cdot\big(\iint_{\mathbb T \times \mathbb R}\omega \sigma^{\perp} f ~ d\theta d\omega)\Big]f ~ d\theta d\omega,
\end{aligned}
\end{align} 
where in the second equality, we used the identity:
\begin{align*}
e\cdot \partial_{t} J&=\frac{d}{dt}e\cdot\iint_{\mathbb T \times \mathbb R}\sigma f ~ d\theta d\omega=\iint_{\mathbb T \times \mathbb R}(e\cdot\sigma)\partial_{t}f ~ d\theta d\omega\\
&=K\iint_{\mathbb T \times \mathbb R}\partial_{\theta}(e\cdot\sigma)\partial_{\theta}(\sigma\cdot J)f ~ d\theta d\omega+\iint_{\mathbb T \times \mathbb R}\omega \sigma^{\perp} \partial_{\theta}(e\cdot\sigma)f ~ d\theta d\omega\\ 
&=Ke\cdot\iint_{\mathbb T \times \mathbb R}\partial_{\theta}(\sigma\cdot J)f ~ d\theta d\omega+e\cdot\iint_{\mathbb T \times \mathbb R}\omega\sigma^{\perp} f ~ d\theta d\omega,
\end{align*}
which holds for any $e$ in $\mathbb{R}^{2}$.
In the above identity, we have used 
the fact that for any vector $v$ in $\mathbb{R}^{2},$
  $\partial_{\theta}(v\cdot\sigma)$
  is the orthogonal projection of $v$
  in the orthogonal subspace to $\sigma(\theta),$
  and the 
fact  that, by definition, $\omega\sigma^\perp(\theta)$
  is contained in the same subspace as well. By similar arguments, we also get
 \begin{align}
 \begin{aligned} \label{dz1}
& \frac{d}{dt}\iint_{\mathbb T \times \mathbb R}(\partial_{\theta}
 (\sigma\cdot J))^{2}f ~ d\theta 
 d\omega \\
 & \hspace{1cm} =2K\iint_{\mathbb T \times \mathbb R}\partial_{\theta}(\sigma
 \cdot J_{f})\partial_{\theta\theta}
 (\sigma\cdot J)\partial_{\theta}
 (\sigma\cdot J)f ~ d\theta d\omega
  \\
  & \hspace{1cm} +2\iint_{\mathbb T \times \mathbb R}\omega \sigma^{\perp} \partial_{\theta\theta}
 (\sigma\cdot J)\partial_{\theta}(\sigma
 \cdot J)f ~ d\theta d\omega
 \\
 & \hspace{1cm} +2K\bigg|\iint_{\mathbb T \times \mathbb R}\partial_{\theta}
 (\sigma\cdot J)f ~ d\theta d\omega\bigg|^{2}
 +2\iint_{\mathbb T \times \mathbb R}\partial_{\theta}(\sigma\cdot J)f ~  d\theta d\omega\cdot \iint_{\mathbb T \times \mathbb R}\omega \sigma^{\perp} f ~ d\theta d\omega. 
 \end{aligned}
 \end{align}
 
\bigskip 
 
\noindent $\bullet$ (The first integral in the R.H.S. of \eqref{Z-1}): By direct calculation, we have
\begin{align}
\begin{aligned}\label{dz2}
& \frac{d}{dt} \iint_{\mathbb T \times \mathbb R}\partial_{\theta}(\sigma\cdot J)\omega f ~ d\theta d\omega \\
& \hspace{1cm} =K\iint_{\mathbb T \times \mathbb R}\partial_{\theta}(\sigma\cdot J)\partial_{\theta}[\partial_{\theta}(\sigma\cdot J) \cdot \omega \sigma^{\perp} ]f ~ d\theta d\omega\\
& \hspace{1cm}+\iint_{\mathbb T \times \mathbb R}\omega \sigma^{\perp} \cdot \partial_{\theta}[\partial_{\theta}(\sigma\cdot J)\omega \sigma^{\perp}]f ~ d\theta d\omega+\iint_{\mathbb T \times \mathbb R}\partial_{\theta}(\sigma\cdot\partial_{t}J) \cdot \omega \sigma^{\perp} f ~ d\theta d\omega\\
& \hspace{1cm} =K\iint_{\mathbb T \times \mathbb R}\partial_{\theta}(\sigma\cdot J)\partial_{\theta\theta}(\sigma\cdot J)\omega \sigma^{\perp} f ~ d\theta d\omega+\iint_{\mathbb T \times \mathbb R}\omega \sigma^{\perp} \cdot \partial_{\theta\theta}(\sigma\cdot J)\omega \sigma^{\perp} f ~ d\theta d\omega\\
&\hspace{1cm}+K\iint_{\mathbb T \times \mathbb R}\partial_{\theta}[\sigma\cdot\iint_{\mathbb T \times \mathbb R}\partial_{\theta}(\sigma\cdot J)f ~ d\theta d\omega]\omega \sigma^{\perp} f ~ d\theta d\omega \\
& \hspace{1cm} +\iint_{\mathbb T \times \mathbb R}\partial_{\theta}[\sigma \cdot\iint_{\mathbb T \times \mathbb R}\omega \sigma^{\perp} f ~ d\theta d\omega]\omega \sigma^{\perp} f ~ d\theta d\omega\\
& \hspace{1cm} =K\iint_{\mathbb T \times \mathbb R}\partial_{\theta}(\sigma  \cdot J)\partial_{\theta\theta}(\sigma\cdot J)\omega \sigma^{\perp}  f ~ d\theta d\omega\\
& \hspace{1cm}+\iint_{\mathbb T \times \mathbb R}\omega \sigma^{\perp} \cdot \partial_{\theta\theta}(\sigma\cdot J)\omega \sigma^{\perp} f ~ d\theta d\omega \\
& \hspace{1cm} +K\iint_{\mathbb T \times \mathbb R}\partial_{\theta}(\sigma\cdot J)f ~ d\theta d\omega\cdot\iint_{\mathbb T \times \mathbb R}\omega \sigma^{\perp} f ~ d\theta d\omega +\bigg|\iint_{\mathbb T \times \mathbb R}\omega \sigma^{\perp} f ~ d\theta d\omega\bigg|^{2},
\end{aligned}
\end{align}
where in the second and third equalities above, we have used the same
tools we used to obtain \eqref{dz1}. Finally, in \eqref{Z-1}, we combine \eqref{dz1} and \eqref{dz2} to prove the claim. \newline

\noindent $\bullet$ Step B (uniform boundedness of  ${\ddot R}$): It follows from the relation \eqref{Z} that 
\begin{equation} \label{R-4}
\Big|  \frac{d^2}{dt^2} R^2 \Big| \leq4(K^{2}+KM)+2KM+2M^{2},
\end{equation}
Then, we use the relations \eqref{R-2} and \eqref{R-2-1} to see that for $t \in [0, T+ \eta)$,
\begin{align*}
\underline{R} |{\ddot R}| &\leq 2 |R {\ddot R}| \leq \Big|  \frac{d^2}{dt^2} R^2 \Big| + 2 |{\dot R}|^2 \cr
&\leq 4(K^{2}+KM)+2KM+2M^{2} + 2(M + K)^2.
\end{align*}
This yields the desired bound estimate for ${\ddot R}$:
\[  |{\ddot R}| \leq \frac{1}{  \underline{R}} \Big[  4(K^{2}+KM)+2KM+2M^{2} + 2(M + K)^2 \Big], \quad t \in [0, T+ \eta). \]
and implies the Lipschitz continuity of ${\dot R}$ on $[0, T+ \eta)$. 

\subsection{Lipschitz continuity of $\mathcal{M}(L^+_{\gamma}(t))$} It follows from \eqref{dL} that we have
\[
\frac{d}{dt}\mathcal{M}(L^+_{\gamma}(t)) = KR \cos\gamma\int_{-M}^M\big[B_{-,\omega}(t)+B_{+,\omega}(t)\big] ~ d\omega + \int_{-M}^M\big(\dot{\phi}(t)-\omega\big)\big[B_{+,\omega}(t)-B_{-,\omega}(t)\big] ~ d\omega.
\]
Hence, we use \eqref{dot phi r}, $R \leq 1$, Lemma \ref{L5.3}, \eqref{R-2} and 
\[ \int_{-M}^{M} \big[B_{-,\omega}(t)+B_{+,\omega}(t)\big] ~ d\omega  \leq 2M \|f_0\|_{L^{\infty}(\mathbb T \times [-M, M])} e^{K(T+\eta)}   \]
to obtain
\begin{align*}
\bigg|\frac{d}{dt}
\mathcal{M}(L^+_{\gamma}(t))\bigg|
&\leq(K R +|\dot{\phi}|+M)\int_{-M}^M\big[B_{-,\omega}(t)+B_{+,\omega}(t)\big] ~ d\omega\\
&\leq 2M \bigg[K+\frac{2M}{\underline{R}}+K\bigg(1-\frac{\underline{R}}{2}\bigg)+M \bigg] \|f_0\|_{L^{\infty}(\mathbb T \times [-M, M])} e^{K(T+\eta)},\\
\end{align*}
on $[0,T+\eta)$. This concludes the proof of Lemma \ref{L6.2}.

\vspace{1cm}

\section{Proof of Lemma \ref{L6.5}} \label{App-C}
\setcounter{equation}{0}
Suppose that $f_0, R$ and $K$ satisfy
\begin{equation} \label{N-K}
\|f_0 \|_{L^{\infty}(\mathbb T \times [-M, M])} < \infty, \quad  \min_{0 \leq t \leq T} R(t) > \underline{R}, \quad \dot{R}(T)=K \mu \quad \mbox{and} \quad K^2 \mu > \frac{M^2}{2\underline R^2} - \frac{3M^2}{4\underline R}
\end{equation}
for some $T \in (0, \infty)$ and some positive constants $\underline{R},$ and $\mu.$ Then, we claim: there exist positive constants $d:=d(K,M,\underline{R},\mu)$ and $E_{3}:=E_{3}(K,M,\underline{R},\mu)$ satisfying 
\[ \dot{R}>0\hspace{1em} {\rm{in}} \hspace{1em} [T,T+d), \qquad 
R(T+d)-R(T)\geq\frac{1}{12}\underline{R}\mu-E_{3}.\] 
Note that Lemma \ref{L5.2} yields
\[ {\dot R} = -\iint_{\mathbb T \times \mathbb R} \sin(\theta - \phi) \omega  f(\theta, \omega, t)   ~ d\theta d\omega + K R  \int_\mathbb T \sin^2 (\theta - \phi)   \rho(\theta, t)  ~ d\theta. \]

We define a function
$S ~ : ~ 
[0,\infty)\rightarrow\mathbb{R}$
given by 
\begin{equation} \label{lin-0}
S(t)=KR(t)\iint_{\mathbb T \times \mathbb R} \sin^{2}(\theta-\phi)f(\theta,\omega,t) ~ d\theta d\omega.
\end{equation}
We remark that for identical oscillator case, this coincides with metric slope. \newline

\noindent $\bullet$ Step A (Derivation of differential inequalities for $R$): For  any $t\geq T$ satisfying $R(t)\geq \underline{R}$, we claim:
\begin{equation}\label{lin}
-\frac{M^{2}}{2K\underline{R}}+\frac{1}{2}S(t)\leq\dot{R}(t)\leq\frac{3}{2}S(t)+\frac{M^{2}}{2\underline{R}K}.
\end{equation}
{\it Proof of claim \eqref{lin}}:  Suppose that 
\[ R(t)\geq \underline{R}. \]
Then, it follows from Lemma \ref{L5.2} that we have  
\begin{align}
\begin{aligned} \label{lin-1}
\dot{R}(t) & =-\iint_{\mathbb T \times \mathbb R}\sin(\theta-\phi)\omega f(\theta,\omega,t) ~ d\theta d\omega \\
&+KR(t) \iint_{\mathbb T \times \mathbb R} \sin^{2}(\theta-\phi)f(\theta,\omega,t) ~ d\theta d\omega.
\end{aligned}
\end{align}
For the first term in \eqref{lin-1}, we use Young's inequality:
\[ |ab|\leq\frac{a^{2}}{2\varepsilon}+\frac{b^{2}}{2}\varepsilon, \quad \mbox{with}~~\varepsilon=KR \]
and the relations $|\omega| \leq M, \quad R(t)\geq \underline{R}$ to see 
\begin{align}
\begin{aligned} \label{lin-2}
& \Big|-\iint_{\mathbb T \times \mathbb R}\sin(\theta-\phi)\omega f(\theta,\omega,t) ~ d\theta d\omega \Big| \\
& \hspace{0.5cm} \leq \frac{1}{2KR(t)}\iint_{\mathbb T \times \mathbb R}\omega^{2}f ~ d\theta d\omega + \frac{KR(t)}{2}\int_{\mathbb T \times \mathbb R} \sin^{2}(\theta-\phi)f(\theta,\omega,t) ~ d\theta d\omega \leq \frac{M^2}{2K\underline{R}} + \frac{1}{2}S(t).
\end{aligned}
\end{align}
We combine \eqref{lin-1} and \eqref{lin-2} to verify the claim \eqref{lin}.  \newline

Note that the relations  \eqref{lin-0} and \eqref{NEW-1} imply
\begin{align}
\begin{aligned} \label{lin-3}
R(t)S(t) &= KR^2(t)\iint_{\mathbb T \times \mathbb R} \sin^{2}(\theta-\phi)f(\theta,\omega,t) ~ d\theta d\omega \\
   &= K\iint_{\mathbb T \times \mathbb R} (-R\sin(\theta-\phi))^2 f(\theta,\omega,t) ~ d\theta d\omega \\
   &= K\iint_{\mathbb T \times \mathbb R}|\partial_{\theta}(\sigma\cdot J)|^{2}f ~ d\theta d\omega. 
\end{aligned}
\end{align}
Then, we use \eqref{lin-3} and \eqref{dz1} to obtain
\begin{align}\begin{aligned}\label{rs}
\frac{d}{dt}(R(t)S(t))&=2K^{2}\iint_{\mathbb T \times \mathbb R}\partial_{\theta}(\sigma\cdot J)\partial_{\theta\theta}(\sigma\cdot J)\partial_{\theta}(\sigma\cdot J)f ~ d\theta d\omega\\
 & +2K\iint_{\mathbb T \times \mathbb R}\omega \sigma^{\perp} \partial_{\theta\theta}(\sigma\cdot J)\partial_{\theta}(\sigma\cdot J)f ~ d\theta d\omega\\
 &+2K^{2}\bigg|\iint_{\mathbb T \times \mathbb R}\partial_{\theta}(\sigma\cdot J)f ~ d\theta d\omega\bigg|^{2}+2K\iint_{\mathbb T \times \mathbb R}\partial_{\theta}(\sigma\cdot J)f d\theta d\omega \iint_{\mathbb T \times \mathbb R}\omega \sigma^{\perp} f ~ d\theta d\omega.
\end{aligned}
\end{align}
By Young's inequality, we get
\begin{align*}
\frac{d}{dt}(R(t)S(t)) & \geq-2K^2 \iint_{\mathbb T \times \mathbb R} |\partial_{\theta}(\sigma\cdot J) f|^{2}f ~ d\theta d\omega-K^2 \iint_{\mathbb T \times \mathbb R} |\partial_{\theta}(\sigma\cdot J)f|^{2}f ~ d\theta d\omega \\
 &-\iint_{\mathbb T \times \mathbb R}\omega^{2}f ~ d\theta d\omega -K^2 \iint_{\mathbb T \times \mathbb R} |\partial_{\theta}(\sigma\cdot J)|^{2}f ~ d\theta d\omega-\iint_{\mathbb T \times \mathbb R}\omega^{2}f ~ d\theta d\omega\\
 & \geq-4KR(t)S(t)-2M^{2}.
\end{align*}
Then, Grownwall's lemma yields
\[
R(t)S(t)\geq\bigg(R(T)S(T)+\frac{M^{2}}{2K}\bigg)e^{-4K(t-T)}-\frac{M^{2}}{2K}. 
\]
Since $S\geq0$ and $R\leq1$, we also obtain
\begin{equation}\label{sbound}
S(t) \geq R(t) S(t) \geq \bigg(R(T)S(T)+\frac{M^{2}}{2K}\bigg)e^{-4K(t-T)}-\frac{M^{2}}{2K}. 
\end{equation}

\vspace{0.5cm}

\noindent $\bullet$ Step B (Lower bound of $R$): We next claim: for some $d > 0$, 
\begin{equation} \label{lin-4}
R\geq R(T), \quad \mbox{in}~~[T,T+d). 
\end{equation}
For the proof of claim \eqref{lin-4}, we first define a positive constant $d$ by the following implicit relation:
\begin{equation}\label{ddef}
\bigg[\frac{1}{3}\underline{R}\dot{R}(T)- \bigg(\frac{M^{2}}{6\underline{R}K}-\frac{M^{2}}{4K}\bigg)\bigg]e^{-4Kd}-\frac{1}{4}\frac{M^{2}}{K}-\frac{M^{2}}{2\underline{R}K}=0.
\end{equation}
The unique existence of such $d$ is guaranteed by the condition \eqref{N-K}. We introduce a set ${\mathcal T}$ and its supremum as follows.
\[  \mathcal{T} :=\Big\{ 
t\in[T,T+d) ~ 
: ~ R(t^{*})\geq R(T)\hspace{2mm}\forall t^{*}\in [T,t] \Big\}, \quad \tau :=\sup\mathcal{T}. \]
Since $T \in {\mathcal T}$, the set ${\mathcal T}$ is non-empty and $\tau$ is well defined. To prove a claim \eqref{lin-4}, it suffices to show
\[ \tau \geq T + d. \]
Suppose not, i.e. $\tau<T+d$.  By the continuity of $R$ which is guaranteed by Lemma \ref{L6.1} and definition of $\tau$, we have
\[
R(\tau)=R(T), \quad  \dot{R}(\tau)\leq0.
\] 
On the other hand, definition of $\tau$ allows us to use inequality \eqref{lin} 
in \eqref{sbound}, for every $t$ in the interval $[T,\tau]$. By doing so, we obtain  
\begin{align*}
S(t)&\geq\bigg(\frac{2}{3}R(T)\dot{R}(T) -R(T)\frac{M^{2}}{3\underline{R}K}+\frac{M^{2}}{2K}\bigg)e^{-4K(t-T)}-\frac{M^{2}}{2K}\\
&\geq\bigg(\frac{2}{3}R(T)\dot{R}(T)-\frac{M^{2}}{3\underline{R}K}+\frac{M^{2}}{2K}\bigg)e^{-4K(t-T)}-\frac{M^{2}}{2K},  \quad \mbox{in $[T,\tau)$},    
\end{align*}
where we have used the fact that $R\leq1.$ Hence, another application of \eqref{lin} yields
\begin{align*}
\dot{R}(\tau) & \geq\bigg[\frac{1}{3}\underline{R}\dot{R}(T)- \bigg(\frac{M^{2}}{6\underline{R}K}-\frac{M^{2}}{4K}\bigg)\bigg]e^{-4K(\tau-T)}-\frac{1}{4}\frac{M^{2}}{K}-\frac{M^{2}}{2\underline{R}K}.\\
 & >\bigg[\frac{1}{3}\underline{R}\dot{R}(T)- \bigg(\frac{M^{2}}{6\underline{R}K}-\frac{M^{2}}{4K}\bigg)\bigg]e^{-4Kd}-\frac{1}{4}\frac{M^{2}}{K}-\frac{M^{2}}{2\underline{R}K}\\
 & =0, 
\end{align*}
In the second inequality, we have used the condition on $K$ in \eqref{N-K}, the assumption
$\tau<T+d,$
and the strict 
monotonicty of the exponential function.
Thus, we reach a contradiction. Hence, we 
conclude $\tau=T+d.$
By the previous argument, we have
\begin{equation}\label{drl}
\dot{R}(t)\geq\bigg[\frac{1}{3}\underline{R}\dot{R}(T)- \bigg(\frac{M^{2}}{6\underline{R}K}-\frac{M^{2}}{4K}\bigg)\bigg]e^{-4K(t-T)}-\frac{1}{4}\frac{M^{2}}{K}-\frac{M^{2}}{2\underline{R}K}, \quad \mbox{in $[T,T+d).$}
\end{equation}
On the other hand, we use definition of $d$ to see
\[
d=\frac{1}{4K}\log\bigg[\frac{1}{3}\underline{R}\dot{R}(T)-\bigg(\frac{M^{2}}{6\underline{R}K}-\frac{M^{2}}{4K}\bigg)\bigg]-\frac{1}{4K}\log\bigg(\frac{1}{4}\frac{M^{2}}{K}+\frac{M^{2}}{2\underline{R}K}
 \bigg).\]
For notational simplicity, we set
\[ a :=\frac{1}{3}\underline{R}\dot{R}(T)-\bigg(\frac{M^{2}}{6\underline{R}K}-\frac{M^{2}}{4K}\bigg) \quad \mbox{and} \quad b:=\frac{1}{4}\frac{M^{2}}{K}+\frac{M^{2}}{2\underline{R}K}.\]
It follows from $\eqref{drl}$ that we have
\begin{align*}
R(T+d)-R(d) &=\int_{T}^{T+d}\frac{d}{dt}R(t) ~ dt\geq\int_{T}^{T+d}[ae^{-4K(t-T)}-b] ~ dt,\\
&=[-\frac{a}{4K}e^{-4K(t-T)}-bt]\bigg|_{t=T}^{t=T+d}\\ 
&=\frac{a}{4K}\bigg(1-e^{-4Kd}\bigg)-bd.
\end{align*}
Then, by \eqref{ddef}, the assumption that $\dot{R}(T)=K\mu$ and definition of $a$ and $b$, we have the desired result.
\begin{align*}
&R (T+d)-R(d) \\
&\hspace{1.5cm} \geq\bigg[\frac{1}{12}\underline{R}\frac{\dot{R}(T)}{K}-\frac{1}{4K}\bigg(\frac{M^{2}}{6\underline{R}K}-\frac{M^{2}}{4K}\bigg)\bigg] \bigg[1-\frac{\frac{M^{2}}{4K}+\frac{M^{2}}{2\underline{R}K}}{\frac{1}{3}\underline{R}\dot{R}(T)-\bigg(\frac{M^{2}}{6\underline{R}K}-\frac{M^{2}}{4K}\bigg)}\bigg]\\
&\hspace{1.5cm} -\bigg(\frac{M^{2}}{4K}+\frac{M^{2}}{2\underline{R}K}\bigg)\bigg[\frac{1}{4K}\log\frac{\frac{1}{3}\underline{R}\dot{R}(T)-\bigg(\frac{M^{2}}{6\underline{R}K}-\frac{M^{2}}{4K}\bigg)}{\frac{M^{2}}{4K}+\frac{M^{2}}{2\underline{R}K}}\bigg]\\
& \hspace{1.5cm}\geq\frac{1}{12}\underline{R}\mu-\frac{1}{12}\underline{R}\mu\frac{\frac{M^{2}}{4K}+\frac{M^{2}}{2\underline{R}K}}{\frac{1}{3}\underline{R}K\mu-\bigg(\frac{M^{2}}{6\underline{R}K}-\frac{M^{2}}{4K}\bigg)} \\
& \hspace{1.5cm}-\frac{1}{4K}\bigg(\frac{M^{2}}{6\underline{R}K}-\frac{M^{2}}{4K}\bigg)\bigg[1-\frac{\frac{M^{2}}{4K}+\frac{M^{2}}{2\underline{R}K}}{\frac{1}{3}\underline{R}K\mu- \bigg(\frac{M^{2}}{6\underline{R}K}-\frac{M^{2}}{4K}\bigg)}\bigg]\\
 &\hspace{1.5cm}-\bigg(\frac{M^{2}}{4K}+\frac{M^{2}}{2\underline{R}K}\bigg)\bigg[\frac{1}{4K}\log\frac{\frac{1}{3}\underline{R}K\mu-\bigg(\frac{M^{2}}{6\underline{R}K}-\frac{M^{2}}{4K}\bigg)}{\frac{M^{2}}{4K}+\frac{M^{2}}{2\underline{R}K}}\bigg].
\end{align*}

\bigskip

\section{Proof of Proposition \ref{P6.1}} \label{App-D}
\setcounter{equation}{0}
Suppose that the assumptions $({\mathcal H}1)$ - $({\mathcal H}3)$ hold, and assume that there exists $t_{0}\geq0$ and $\underbar{R}>0$ such that 
\[
 R(t_{0})\geq R_0, \quad  \inf_{0 \leq t \leq t_0} R(t) > \underbar{R}.
\]
Let $t\geq t_{0}$ be an instant satisfying $\dot{R}(t)\leq0$. Then for such $t$, in which $R$ is in non-increasing mode,  we claim:
\begin{equation}\label{E-NN}
\mathcal{M}(L^+_{\frac{\pi}{3}}(t))\geq\frac{1}{2}(R(t_{0})+1)-E_{1}.
\end{equation}
For the proof of claim, we consider a set ${\mathcal N}(t_0)$ consisting of non-increasing moments of $R$ after $t_0$:
\[  \mathcal{N}(t_0) := \bigg\{ t\geq t_{0} ~ : ~ \dot{R}(t)\leq0\bigg\},\]
and the set $\mathcal{T}(t_0)$:
\[
\mathcal{T}(t_0) := \bigg\{ s \in [t_{0}, \infty)~:~\eqref{E-NN} ~\ \mbox{holds} \hspace{1em}\forall t\in [t_{0},s] \cap {\mathcal N}(t_0)\bigg\}.
\]
We set 
\[ T^*(t_0):=\sup \mathcal{T}(t_0). \]
Notice that $[t_0, T^*(t_0)) \subset \mathcal T(t_0)$ and it suffices now to prove $T^*(t_0)=\infty$.
Since the proof is rather long, we split its proof into several steps. \newline

\noindent $\bullet$ Step A (the set $\mathcal{T}(t_0)$ is not empty): \newline

\noindent If ${\dot R}(t_0) > 0$,  the defining relation of the set ${\mathcal T}(t_0)$ holds trivially. Thus, $t_0 \in {\mathcal T}(t_0)$. \newline 
If ${\dot R}(t_0) \leq 0$, then it follows from Lemma \ref{L6.3} that \eqref{E-NN} holds for $t=t_0$.

Thus, $t_{0}\in\mathcal{T}(t_0)$. In any case, the set ${\mathcal T}(t_0)$ is not empty. Thus, its supremum exists and lies in the set $[t_0, \infty]$. \newline

\noindent $\bullet$ Step B (the supremum $T^*(t_0) = \infty$): Suppose not, i.e.,
\[ T^*(t_0) < \infty. \]

\noindent $\diamond$ Step B.1: We want to show
\begin{equation}\label{rlo5}
R(t) \geq R(t_{0})-2E_{1}-E_{2}, \quad t \in [t_{0},T^*(t_0)),
\end{equation}
where $E_{1}$ and $E_{2}$ were defined in assumption $({\mathcal H}3)$. To see this, first note that
\[
R(t_0) \geq R(t_0) -2E_1 - E_2
\]
from $(\mathcal H 3)$. Second, $R$ and $\dot R$ are Lipschitz continuous due to the Lemma \ref{L6.2}. Thus, for each $t \in [t_0, T^*(t_0))$ such that $\dot R(t)\leq 0$, from the definition of $T^*(t_0)$, we have \eqref{E-NN}, which implies from Lemma \ref{L6.3}
\[
R(t) \geq R(t_0) -2E_2 - E_1.
\]
This means the quantity $R(t_0) - 2E_2 - E_1$ is a lower bound for $R$ in $[t_0, T^*(t_0))$. This shows the claim.

\vspace{0.5cm}

\noindent $\diamond$ Step B.2: we claim:
\begin{equation} \label{Extra-1}
 \dot{R}(T^*(t_0))=0. 
\end{equation} 
This property comes directly form the continuity of $\dot R$ and $\mathcal M(L^+_\frac{\pi}{3}(t))$. \newline

$\clubsuit$~~Case A: Suppose $\dot R(T^*(t_0)) > 0$. Then there exists a time interval $(T^*(t_0)-\eta, T^*(t_0)+\eta)$ such that
\[
\dot R(t)>0 \quad \mbox{for} \quad t\in (T^*(t_0) - \eta, T^*(t_0) + \eta),
\]
which contradicts to definition of $T^*(t_0) = \sup\mathcal T$. \newline

$\clubsuit$~~Case B: Suppose $\dot R(T^*(t_0)) < 0$. In this case, we have
\[
\dot R(t) < 0 \quad \mbox{for}\quad t \in (T^*(t_0)-\eta, T^*(t_0)+\eta).
\]
By Lemma \ref{L6.4}, we have 
\[ \frac{d}{dt}\mathcal M(L^+_\frac{\pi}{3}(t)) \geq 0 \quad \mbox{for}~~t \in (T^*(t_0) - \eta, T^*(t_0)+\eta). \]
Here we used Step B.1 to satisfy the condition $R(t)\geq \underline R$.
This gives
\[
\mathcal M(L^+_\frac{\pi}{3}(t)) \geq\mathcal M(L^+_\frac{\pi}{3}(T^*(t_0))) \geq \frac{1}{2}(R(t_0) +1) -E_1  \quad \mbox{for}\quad T^*(t_0) \leq t < T^*(t_0) + \eta,
\]
which also contradicts to definition of $T^*(t_0)$. Thus, we obtain the desired result \eqref{Extra-1}.

\vspace{0.5cm}

\noindent $\diamond$ Step B.3: In this part, we want to show that the mass in the interval $L^+_{\frac{\pi}{3}}$ at $T^*(t_0)$ satisfies   
\[
\mathcal M(L^+_{\frac{\pi}{3}}(T^*(t_0))) \geq\frac{1}{2}(R(t_{0})+1)-E_{1}. 
\] 
Notice that from Step B.2 and Lemma \ref{L6.3}, we have
\[
\mathcal M(L^+_{\frac{\pi}{3}}(T^*(t_0))) \geq \frac{1}{2}\Big(R(T^*(t_0)) + 1\Big) - E_1
\]
thus it suffices to show $R(T^*(t_0)) \geq R(t_0)$. Now, consider the two cases:\newline

$\clubsuit$~~Case A ($\mathcal N(t_0) \cap [t_{0},T^*(t_0))=\emptyset$): In this case, since we have $\dot R(t) \geq 0$ for $t\in [t_0, T^*(t_0)]$, we get
\[ R(T^*(t_0)) \geq R(t_0).  \]

$\clubsuit$~~Case B ($\mathcal N(t_0) \cap [t_0, T^*(t_0)) \neq \emptyset$): We define $t_s := \sup \Big( \mathcal N(t_0) \cap [t_0, T^*(t_0)) \Big)$. 
\begin{itemize}
\item Suppose there exists a sequence $\{t_k\} \subset \mathcal N(t_0) \cap [t_0, T^*(t_0))$ such that $t_k\uparrow T^*(t_0)$ and $\dot R(t_k) \leq 0$ for all $k\in \mathbb N$, i.e., $t_s = T^*(t_0)$. Then, we have
\[
\mathcal M(L^+_\frac{\pi}{3}(t_k)) \geq \frac{1}{2}(R(t_0) +1) - E_1 \quad \mbox{for} \quad k \in \mathbb N.
\]
Thus, by the continuity of $\mathcal M(L^+_\frac{\pi}{3}(t))$, we obtain
\[
\mathcal M(L^+_\frac{\pi}{3}(T^*(t_0))) \geq \frac{1}{2}(R(t_0) +1) - E_1.
\]
\item For the case of $t_s < T^*(t_0)$, we have 
\[ \dot R(t) >0 \quad \mbox{for}~~ t \in (t_s, T^*(t_0)). \]
Lemma \ref{L6.3} and $\dot R(t_s)\leq 0$ imply
\[
\mathcal M(L^+_\frac{\pi}{3}(t_s)) \geq \frac{1}{2}(R(t_0) +1) - E_1.
\]
We now investigate the mass $\mathcal M(L^+_\frac{\pi}{3}(t))$ for $t\in(t_s, T^*(t_0)]$:
\begin{enumerate}
\item If $\dot R(t) < K\mu$ for all $t \in (t_s, T^*(t_0)]$, by Lemma \ref{L6.4}, the mass is non-decreasing, i.e., $\frac{d}{dt} \mathcal M(L^+_\frac{\pi}{3}(t)) \geq 0$ for $t \in (t_s, T^*(t_0)$. Thus, we attain
\[
\mathcal M(L^+_\frac{\pi}{3}(T^*(t_0))) \geq\mathcal M(L^+_\frac{\pi}{3}(t_s)) \geq \frac{1}{2}(R(t_0)+1) - E_1.
\]
\item Suppose $\dot R(t) \geq K\mu$ for some $t \in (t_s, T^*(t_0)]$. Since $\dot R(t_s)\leq 0$ and the continuity of $\dot R$, there exists $t_c$ such that
\[
\dot R(t) < K\mu \quad \mbox{for} \quad t\in (t_s, t_c) \quad \mbox{and} \quad \dot R(t_c) = K\mu.
\]
By Lemma \ref{L6.5}, there exists a positive constant $d$ such that
\[ \dot R(t) > 0 \quad \mbox{for} \quad t\in [t_c, t_c+d) \quad \mbox{and} \quad R(t_c+d) - R(t_c) \geq \frac{R_0}{24}\mu - E_3. \]
Since we have $\dot R(T^*(t_0)) = 0$ from Step B.2 and the contiunity of $\dot R$, we attain $t_c +d < T^*(t_0).$ By the definition of $t_s$, we have $\dot R(t)>0$ for $t\in (t_s, T^*(t_0))$, which implies
\[
R(T^*(t_0)) - R(t_s) \geq R(t_c + d) - R(t_c) \geq \frac{R_0}{24}\mu - E_3
\]
By Step B.1 and $(\mathcal H3)$, we obtain
\begin{align*}
R(T^*(t_0)) &\geq R(t_s) + \frac{R_0}{24}\mu - E_3\\
&\geq R(t_0) + \frac{R_0}{24}\mu - E_3 - E_2 - 2E_1\\
&\geq R(t_0).
\end{align*}
We again use Lemma \ref{L6.3} for the result of Step B.2 to get
\[
\mathcal M(L^+_\frac{\pi}{3}(T^*(t_0))) \geq \frac{1}{2}(R(T^*(t_0)) + 1) - E_1 \geq \frac{1}{2}(R(t_0) + 1) - E_1, 
\]
which conclude Step B.3.
\end{enumerate}
\end{itemize}

\noindent $\diamond$ Step B.4: Finally, we show $T^*(t_0) = \infty$.
Since $\dot R(T^*(t_0)) = 0$ from Step B.2 and the continuity of $\dot R$, there is a small time interval such that
\[
\dot R(t) < K\mu \quad \mbox{for} \quad t\in (T^*(t_0) - \eta, T^*(t_0) + \eta),
\]
which implies
\[
\frac{d}{dt} \mathcal M(L^+_\frac{\pi}{3} (t)) \geq 0 \quad \mbox{for} \quad t\in(T^*(t_0) - \eta, T^*(t_0) + \eta)
\]
by Lemma \ref{L6.4} where we use Step B.3 to satisfy the condition $R(t)\geq \underline R$. Thanks to the result of Step B.3, we have
\begin{align*}
\mathcal M(L^+_\frac{\pi}{3} (t)) &\geq \mathcal M(L^+_\frac{\pi}{3} (T^*(t_0)))\\
&\geq \frac{1}{2}(R(t_0)+1) - E_1 \quad \mbox{for} \quad T^*(t_0) \leq t < T^*(t_0) + \eta,
\end{align*}
which contradicts to the definition of $T^*(t_0)$. Therefore, we conclude that $T^*(t_0) = \infty$.

\vspace{1cm}

\section{Proof of Corollary \ref{C6.2}} \label{App-E}
\setcounter{equation}{0}

We next show that the $L^{2}$ norm of $\varrho$ in an interval
of length $\frac{\pi}{3},$ centered at $-\phi(t),$ decays exponentially after some finite time. For this, we define for each $t \geq 0$ and $\omega$ in $[-M,M]$,  a functional
\[
\Gamma^-_{\frac{\pi}{3},\omega}(t):=\int_{L^-_{\frac{\pi}{3}}(t)}|\varrho(\theta,\omega, t)|^{2} ~ d\theta.
\] 
By Corollary \ref{C6.1} and $({\mathcal H}2)$, we have 
\[ \inf_{0 \leq t < \infty} R(t) \geq \frac{R_0}{2}. \]
Let $\varepsilon>0$ be sufficiently small so that
\begin{equation}\label{con2}
\frac{2M}{KR_0}+\frac{4}{K}\frac{M}{R_0^{2}}+\frac{2\sqrt{2}}{R_0 \sqrt{R_0}}\sqrt{\frac{M}{K}+\mu+\varepsilon}-\frac{1}{2}<0. 
\end{equation}
The existence of such $\varepsilon$ is guaranteed by the assumption $({\mathcal H}2)$.  We set boundary values:
\[
\tilde{B}_{-,\omega}(t) :=\varrho \Big(\phi(t)+\frac{\pi}{2}+\frac{\pi}{3}, \omega, t \Big),  \qquad 
\tilde{B}_{+,\omega}(t) := \varrho \Big(\phi(t)+\frac{3\pi}{2}-\frac{\pi}{3}, \omega, t \big). 
\]
By the same argument as in Section \ref{sec:5.3.2}, we have
\begin{align}
\begin{aligned}\label{T-0-0}
\frac{d}{dt} \Gamma^-_{\frac{\pi}{3},\omega}(t) & =\dot{\phi}(t)\big(\tilde{B}_{+,\omega}(t)\big)^{2}-\dot{\phi}(t)\big(\tilde{B}_{-,\omega}(t)\big)^{2}+2\int_{L^-_{\frac{\pi}{3}}(t)} \varrho \partial_{t} \varrho ~ d\theta\\
 & =:\dot{\phi}(t)\Big[\big(\tilde{B}_{+,\omega}(t)\big)^{2}-\big(\tilde{B}_{-,\omega}(t)\big)^{2}\Big]+ {\mathcal I}_4,
\end{aligned}
\end{align}
where 
\begin{align}
\begin{aligned}\label{T-0}
\mathcal{I}_4(t) & =-2\int_{L^-_{\frac{\pi}{3}}(t)} \varrho \partial_{\theta}\Big[\varrho \big(\omega-KR\sin(\theta-\phi)\big)\Big] ~ d\theta\\
 & =-2\int_{L^-_{\frac{\pi}{3}}(t)} \Big[ (\varrho \partial_{\theta} \varrho)\big(\omega-KR \sin(\theta-\phi)\big)- KR \varrho^{2} \cos(\theta-\phi)  \Big] ~ d\theta\\
 & =-\int_{L^-_{\frac{\pi}{3}}(t)} (\partial_{\theta} \varrho^{2} )\big(\omega-KR \sin(\theta-\phi)\big) ~ d\theta +2 KR
  \int_{L^-_{\frac{\pi}{3}}(t)}\varrho^2 \cos(\theta-\phi) ~ d\theta\\
 & =:\mathcal{I}_{41}(t)+\mathcal{I}_{42}(t).
\end{aligned}
\end{align}
Below, we estimate the terms ${\mathcal I}_{4i},~i=1,2$ separately. \newline

\noindent $\bullet$ (Estimate of ${\mathcal I}_{41}$): By integration by parts, we have
\begin{align} \label{T-1}
\begin{aligned}\mathcal{I}_{41}(t) & =-\Big[\big(\tilde{B}_{+,\omega}(t)\big)^{2}\big(\omega-KR \sin(\frac{3\pi}{2}-\frac{\pi}{3})\big)
-\big(\tilde{B}_{-,\omega}(t)\big)^{2}\big(\omega-KR \sin(\frac{\pi}{2}+\frac{\pi}{3})\big)\Big]\\
 & \quad-KR\int_{L^-_{\frac{\pi}{3}}(t)} \varrho(\theta,\omega, t)^{2}\cos(\theta-\phi) ~ d\theta\\
 & =:\mathcal{I}_{411}(t)+\mathcal{I}_{412}(t).
\end{aligned}
\end{align}
By rearranging the terms in ${\mathcal I}_{411}$, we obtain
\begin{equation} \label{T-2}
 {\mathcal I}_{411} = -\omega\Big[\big(\tilde{B}_{+,\omega}(t)\big)^{2}-\big(\tilde{B}_{-,\omega}(t)\big)^{2}\Big]- \frac{KR}{2}\Big[\big(\tilde{B}_{+,\omega}(t)\big)^{2}+\big(\tilde{B}_{-,\omega}(t)\big)^{2}\Big]. 
\end{equation} 
We also combine the terms ${\mathcal I}_{42}$ and ${\mathcal I}_{412}$ and use 
\[  \cos(\theta-\phi) \leq -\sin\frac{\pi}{3} = -\frac{\sqrt{3}}{2} \quad \mbox{on}~~L^-_{\frac{\pi}{3}}(t),     \]
to obtain
\begin{align}
\begin{aligned}  \label{T-3}
\mathcal{I}_{42}(t)+\mathcal{I}_{412}(t)  &=   KR \int_{L^-_{\frac{\pi}{3}}(t)} \varrho(\theta, \omega, t)^2 \cos(\theta-\phi) ~ d\theta \\
 &\leq-\frac{KR \sqrt{3}}{2} \Gamma^-_{\frac{\pi}{3},\omega}(t).
\end{aligned}
\end{align}

Finally, in \eqref{T-0-0} we combine \eqref{T-0}, \eqref{T-1}, \eqref{T-2} and \eqref{T-3} to obtain
\begin{align*}
\frac{d}{dt} \Gamma^-_{\frac{\pi}{3},\omega}(t) & \leq(\dot{\phi}(t)-\omega)\Big[\big(\tilde{B}_{+,\omega}(t)\big)^{2}-\big(\tilde{B}_{-,\omega}(t)\big)^{2}\Big]\\
 &- \frac{KR}{2} \Big[\big(\tilde{B}_{+,\omega}(t)\big)^{2}+\big(\tilde{B}_{-,\omega}(t)\big)^{2}\Big]- \frac{KR \sqrt{3}}{2} \Gamma^-_{\frac{\pi}{3},\omega}(t) \\
 & \leq\big(-\frac{KR}{2} +|\dot{\phi}(t)-\omega|\big)\Big[\big(\tilde{B}_{+,\omega}(t)\big)^{2}+\big(\tilde{B}_{-,\omega}(t)\big)^{2}\Big] \\
 &-\frac{KR \sqrt{3}}{2}  \Gamma^-_{\frac{\pi}{3},\omega}(t).
\end{align*}
By Corollary \ref{C6.1}, there exists $T\geq0$ such that
\[ \dot{R}\leq K\mu+K\varepsilon,  \quad \mbox{in $[T,\infty)$}. \]
Using similar arguments as in Lemma \ref{L6.2}, we have
\[
|\dot{\phi|} < \frac{2M}{R_0}+\sqrt{\frac{2K}{R_0}}\sqrt{M+\dot{R}} \leq  \frac{2M}{R_0}+\sqrt{\frac{2K}{R_0}}\sqrt{M+ K(\mu + \varepsilon)},
\]
where we used $({\mathcal H}3)$ and Corollary \ref{C6.1} to see
\[ R(t) > \frac{R_0}{2}, \quad t \in [0,\infty),  \]
Thus, by assumption, for any $t \geq T$ we have
\begin{align*}
\frac{d}{dt} \Gamma^-_{\frac{\pi}{3},\omega}(t)&\leq\bigg(\frac{2M}{R_0}+\sqrt{\frac{2K}{R_0}}\sqrt{M+K\mu+K\varepsilon}+M-K\frac{R_0}{4}\bigg)\\
&\times\Big[\big(\tilde{B}_{+,\omega}(t)\big)^{2}+\big(\tilde{B}_{-,\omega}(t)\big)^{2}\Big] -\frac{1}{4}KR_0 \Gamma^-_{\frac{\pi}{3},\omega}(t)\\
&=\frac{KR_0}{2}\bigg(\frac{2M}{KR_0}+\frac{4}{K}\frac{M}{R_0^{2}}+\frac{2\sqrt{2}}{R_0 \sqrt{R_0}}\sqrt{\frac{M}{K}+\mu+\varepsilon}-\frac{1}{2}\bigg)\\
 &\times\Big[\big(\tilde{B}_{+,\omega}(t)\big)^{2}+\big(\tilde{B}_{-,\omega}(t)\big)^{2}\Big] -\frac{1}{4}KR_0 \Gamma^-_{\frac{\pi}{3},\omega}(t) \\
 &\leq -\frac{1}{4}KR_0 \Gamma^-_{\frac{\pi}{3},\omega}(t),
\end{align*}
where we used \eqref{con2}. Then, Gronwall's lemma yields the desired exponential decay:
\[ \Gamma^-_{\frac{\pi}{3},\omega}(t)\leq e^{-\frac{KR_0}{4}(t-T)} \Gamma^-_{\frac{\pi}{3},\omega}(T),  \quad t \in [T, \infty).
\]
On the other hand for $t \in [T, \infty)$ we have
\begin{align*}
& \iint_{L^-_{\frac{\pi}{3}}(t) \times \mathbb R}|f|^2  ~ d\theta d\omega  \\
& \hspace{1cm} =\int_{\mathbb R} g^{2}(\omega)\int_{L^-_{\frac{\pi}{3}}(t)} |\varrho|^{2} ~ d\theta d\omega  =\int_{\mathbb R} g^{2}(\omega)\Gamma^-_{\frac{\pi}{3},\omega}(t) ~ d\omega\\
& \hspace{1cm} \leq e^{-\frac{KR(0)}{4}(t-T)}\int_{\mathbb R} g^{2}(\omega) \Gamma^-_{\frac{\pi}{3},\omega}(T) ~ d\omega =e^{-\frac{KR(0)}{4}(t-T)}\iint_{L^-_{\frac{\pi}{3}}(T)\times \mathbb R} |f(T)|^2  ~ d\theta d\omega.
\end{align*}
Thus, we obtain the estimate \eqref{decay2}. The second estimate in \eqref{decay2}
is a consequence of the first inequality in \eqref{decay2} and Cauchy-Schwarz inequality.

\vspace{1cm}

\section{Dynamics of the Kuramoto-Sakaguchi vector field}  \label{App-F}
\setcounter{equation}{0}
In this section, we study analytical properties of integral curves for the Kuramoto-Sakaguchi vector field ${\mathcal X}$ defined by
\begin{equation} \label{KV}
{\mathcal X}(\theta,\omega,t) := \Big(\omega-KR(t)\sin(\theta-\phi(t)),0,1 \Big). 
\end{equation}
Before we study several properties of the integral curves associated with \eqref{KV}, we briefly discuss well-posedness of an autonomous ODE. In the sequel, we assume that $T_{\kappa}$ is a positive constant satisfying
\begin{equation}\label{rlower}
 R>\kappa\hspace{1em}\mbox{in\hspace{1em}}[T_{\kappa},\infty)
 \hspace{1em} {\rm{and}} \hspace{1em} R> {\underline R} \quad {\rm{in}} \hspace{1em} [0,\infty),
\end{equation}
for some positive constants $\underline R$ and $\kappa.$ \newline

It follows from Lemma \ref{L5.3} and \ref{L6.2} that ${\mathcal X}$ is Lipschitz in the given domain. Recall from $({\mathcal H}4)$ just before Proposition \ref{P6.1}  that
\begin{equation} \label{ek}
  \varepsilon_{\kappa}=\frac{\kappa+1}{\kappa^{2}}\frac{M}{K}+\frac{(1-\kappa)}{\kappa} < 1.
\end{equation}
Under the assumptions $({\mathcal H}1) - ({\mathcal H}4)$, by Proposition \ref{P6.2}, there exists $\kappa$ satisfying \eqref{ek} and \eqref{rlower} for some $T_{\kappa}.$

We study properties of the integral curves for \eqref{KV} which have been used in the proof of Theorem \ref{T3.3}. \newline

For a given $(\theta^{*},\omega^{*})$ in $\mathbb{T}\times[-M,M]$ and $t^{*}$
  in $[T_{\kappa},\infty),$ let $(\theta(t) := (\theta(t;t^*, \theta^*, \omega^*),  \omega(t):= \omega(t;t^*, \theta^*, \omega^*))$ be a characteristic curve of \eqref{K-S}, i.e., it is a solution to the Cauchy problem for the following ODE:
\begin{equation}\label{a}
\begin{dcases}
{\dot \theta}(t) =\omega(t)-KR \sin(\theta(t)-\phi(t)), \quad {\dot \omega}(t)=0, \quad t > T_{\kappa}, \\
 (\theta(t^{*}), \omega(t^*))=(\theta^{*}, \omega^{*}).
\end{dcases}
\end{equation}
Since the vector field ${\mathcal X}$ is Lipschitz and $\mathbb{T}\times[-M,M]$ is compact, by Cauchy Lipschitz theorem, characteristics $(\theta(t), \omega(t))$ exists globally and is unique. Below, we will study how the inner product between $e^{{\mathrm i} \theta(t)}$ and $e^{{\mathrm i}\phi(t)}$ can be controlled from above by the solution of an 
 autonomous first order ODE.  For this, we first define a {\it barrier}.
 
\begin{definition}
\emph{(Barrier)}
For  $(t^*, p^*)$ satisfying 
\[ t^{*}\in[T_{\kappa},\infty) \quad \mbox{and} \quad -\sqrt{1-\varepsilon_{\kappa}^{2}} \leq p^* \leq \sqrt{1-\varepsilon_{\kappa}^{2}}, \]
the map $p$:
 \[
p ~ : ~ [T_{\kappa},t^{*}]\rightarrow(-1,1),
\] 
is said to be a barrier through $p^{*}$ at $t^{*}$ if it satisfies 
\begin{equation} \label{Ode2}
p(t^{*})=p^{*}, \quad \mbox{and} \quad  \dot{p}= \kappa K \big(\sqrt{1-p^{2}}-\varepsilon_{\kappa} \big)\sqrt{1-p^{2}}.
\end{equation}
\end{definition}

Since the right-hand side of \eqref{Ode2} is not Lipscthiz at $|p| = 1$, uniqueness is not clear a priori. However, it can be shown that there exists a unique such map $p = p(t).$ 

\begin{lemma}\label{LG2} Suppose that $\mbox{supp}~g \subset [-M, M]$, and~ $p^{*}, t^{*}$ are positive constants satisfying
\begin{equation}\label{b}
-\sqrt{1-\varepsilon_{\kappa}^{2}} \leq p^{*} \leq \sqrt{1-\varepsilon_{\kappa}^{2}}, \qquad 
 t^{*}\geq T_{\kappa}.
\end{equation}
Then, the barrier $p = p(t)$ through $p^{*}$
at $t^{*}$ is unique and satisfies
\[ -\sqrt{1-\varepsilon_{\kappa}^{2}}\leq p(t)\leq\sqrt{1-\varepsilon_{\kappa}^{2}}, \quad \mbox{for all} ~t. \]
\end{lemma}
\begin{proof}
Let $\varepsilon>0$ be a positive constant satisfying
\[ \sqrt{1-\varepsilon_{\kappa}^{2}}+\varepsilon<1. \]
Recall $F_\kappa$ defined in \eqref{FF}:
\[F_\kappa (q) := \kappa K (\sqrt{1-q^{2}}-\varepsilon_{\kappa} )\sqrt{1-q^{2}}, \quad q \in [-1, 1]. \]
Choose a Lipschitz function $\tilde{F}_\kappa$ compactly supported on $(-1,1)$ that coincides with $F_\kappa$ in $[-\sqrt{1-\varepsilon^2_\kappa}, 
\sqrt{1-\varepsilon^2_\kappa}]$.
Since $\tilde{F}_\kappa$ has compact support, it follows from Cauchy-Lipschitz
theorem that the equation 
\[
\dot{q}=\tilde{F}_\kappa (q),
\]
has a unique solution. Moreover, since
\[
q_{1}(t)=-\sqrt{1-\varepsilon_{\kappa}^{2}} \quad \mbox{and} \quad q_{2}(t)=\sqrt{1-\varepsilon_{\kappa}^{2}},
\]
are solutions,  uniqueness of ODE implies that any solution satisfying
\[-\sqrt{1-\varepsilon_{\kappa}^{2}}\leq q(t^*)\leq\sqrt{1-\varepsilon_{\kappa}^{2}},\]
satisfies 
\[
-\sqrt{1-\varepsilon_{\kappa}^{2}}\leq q(t)\leq\sqrt{1-\varepsilon_{\kappa}^{2}}, \quad \mbox{for all} ~ t.\]
This complete the proof.
\end{proof}
Next we study a quantitative growth estimate to be used in Corollary \ref{CG2}. 
\begin{lemma}\label{LG3}
Let $p^{*},$ $\varepsilon,$ and $t^{*}$ be positive constants satisfying
\begin{equation} \label{c}
0 < \varepsilon<\sqrt{1-\varepsilon_{\kappa}^{2}}, \qquad 
p^{*}=\sqrt{1-\varepsilon_{\kappa}^{2}}-\varepsilon, \qquad D(\varepsilon,\kappa) <t^{*}-T_{\kappa}.
\end{equation}
Then, there exists a unique constant $d<D(\varepsilon, \kappa)$ such that the barrier $p$ through $p^{*}$ at $t^{*}$ satisfies
\[
p(t^{*}-d)=-\sqrt{1-\varepsilon_{\kappa}^{2}}+\varepsilon,
\]
Here, $D(\varepsilon, \kappa)= \frac{2(\sqrt{1-\varepsilon_\kappa^2}-\varepsilon)}{F_\kappa(\sqrt{1-\varepsilon_\kappa^2}-\varepsilon)}$ as given in \eqref{FF}
\end{lemma}
\begin{proof}
It follows from Lemma \ref{LG2} that we have 
\[ -\sqrt{1-\varepsilon_{\kappa}{}^{2}}<p(t)<\sqrt{1-\varepsilon_{\kappa}{}^{2}}, \quad \mbox{for all $t$}. \]
Since
\[
\dot{p}=F_{\kappa}(p)\geq F_{\kappa}\big(\sqrt{1-\varepsilon_{\kappa}^{2}}-\varepsilon\big)>0,
\]
whenever $p$ is contained in the interval
\[
\bigg[-\sqrt{1-\varepsilon_{\kappa}^{2}}+\varepsilon,\sqrt{1-\varepsilon_{\kappa}^{2}}-\varepsilon\bigg],
\]
there exist a unique $d$ such that
\begin{equation}\label{d4}
p(t^{*}-d)=-\sqrt{1-\varepsilon_{\kappa}^{2}}+\varepsilon.
\end{equation}
Then, we have
\[
p(t^{*})-p(t^{*}-d) =\int_{t^{*}-d}^{t*}\dot{p} ~ dt  \geq\int_{t^{*}-d}^{t*}F_{\kappa}\big(\sqrt{1-\varepsilon_{\kappa}^{2}}-\varepsilon\big)dt =F_\kappa \big(\sqrt{1-\varepsilon_{\kappa}^{2}}-\varepsilon\big)d.
\]
This yields 
\[
d<\frac{p(t^{*})-p(t^{*}-d)}{F_{\kappa}\big(\sqrt{1-\varepsilon_{\kappa}^{2}}-\varepsilon\big)},
\]
and the desired estimate follows from \eqref{c} and \eqref{d4}.
\end{proof} 

 \begin{lemma}\label{LG1} Let $t^{*},\theta^{*},\omega^{*}$
and $p^{*}$ be constants satisfying 
\[ p*\in\big[-\sqrt{1-\varepsilon_{\kappa}^{2}},\sqrt{1-\varepsilon_{\kappa}^{2}}\big], \quad 
 \omega^*\in[-M,M], \quad t^{*}\geq T_{\kappa} \quad 
\mbox{and} \quad 
\cos(\theta^{*}-\phi(t^{*}))\leq p^{*}.
\]
Then, the characteristics $(\theta(t), \omega(t))$ through $(\theta^{*},\omega^{*})$
at $t^{*}$ satisfies
\[ \cos(\theta(t)-\phi(t))\leq p(t)
\hspace{1em}\forall t \in [T_{\kappa},t^{*}],
\]
where  $p$ is the barrier through $p^{*}$ at $t^{*}.$
\end{lemma}
 \begin{proof}
 Let $P ~ : ~ [T_{\kappa},t^{*}]\rightarrow[-1,1]$ be defined by the following relation:
 \[ P(t) :=\cos(\theta(t)-\phi(t)). \]
It follows from \eqref{a} and Lemma \ref{L5.2} that for $[T_{\kappa},t^{*}]$, 
\begin{align*}
{\dot P}(t) &=-\big(\dot{\theta}(t) -
\dot{\phi}(t) \big)\sin(\theta(t) -\phi(t))\\
&=-\big(\omega(t)-KR(t) \sin(\theta(t)-\phi(t))-\dot{\phi}(t) \big)\sin(\theta(t)-\phi(t))\\
&=KR(t) \sin^{2}(\theta(t)-\phi(t))+
[\dot{\phi}(t)-\omega(t)]\sin(\theta(t)-\phi(t))\\
&=KR(t)-KR(t)\cos^{2}(\theta(t)-\phi(t))+
[\dot{\phi}(t)- \omega(t)]\sin(\theta(t)-\phi(t))\\
&\geq KR(t) \big(1-\cos^{2}(\theta(t)-
\phi(t))\big)-(|\dot{\phi}(t)|+M)
\sqrt{1-\cos^{2}(\theta(t)-\phi(t))}\\
&>K\kappa \big(1-|P(t)|^{2} \big)-\bigg[\bigg(1+\frac{1}{\kappa}\bigg)M+K(1-\kappa)
\bigg]\sqrt{1-|P(t)|^{2}},
\end{align*}
where we used \eqref{dot phi r} and \eqref{rlower} in the last line. Thus, we have
\begin{equation}\label{Ode3}
\dot{P}>K\kappa (\sqrt{(1-P^{2})}-\varepsilon_{\kappa} )\sqrt{1-P^{2}}.
\end{equation}
Thus, since $P(t^{*})\leq p^{*},$ the desired result follows from \eqref{Ode2}, \eqref{Ode3} and Lemma \ref{LG2}.
\end{proof}
As a direct application of Lemma \ref{LG1}, we have the following two corollaries.

\vspace{0.5cm}

\begin{corollary}\label{CG1}
Let $\delta,$ $\theta^{*}, \omega^{*},$ and $t^*$ be positive constants such that
\[ 0<\delta<\sqrt{1-\varepsilon_{\kappa}^2}, \quad 
 \omega^*\in[-M,M], \quad t^{*}\geq T_{\kappa} \quad \mbox{and} \quad 
 \cos(\theta^{*}-\phi(t^{*})) \leq -\delta.
\]
Then, the characteristics $(\theta(t), \omega(t))$ through $(\theta^{*},\omega^{*})$ at $t^{*}$ satisfies
\[
 \cos (\theta(t)-\phi(t))\leq -\delta, \quad \forall t \in [T_{\kappa},t^{*}].
\]
\end{corollary}
\begin{proof}
Let $p$ be the barrier through $-\delta$ at $t^{*}.$ By Lemma \ref{LG2} and \eqref{Ode2}, we have that $p$ is nondecreasing, and since
\[
 \cos (\theta(t^{*})-\phi(t^{*}))\leq  p(t^{*}) = -\delta,
\]
the desired result follows from Lemma \ref{LG1} using  $-\delta=-p^{*}$.
\end{proof}

\vspace{0.5cm}

\begin{corollary}\label{CG2} 
Let $\varepsilon,$ $\theta^{*}, \omega^{*},$ and $t^*$ be positive constants satisfying
\begin{align*}
 & 0<\varepsilon<\sqrt{1-\varepsilon_{\kappa}^2}, \quad  D(\varepsilon,\kappa) <t^{*}-T_{\kappa}, \quad  \omega^*\in[-M,M], \cr
&  \mbox{and} \quad 
 \cos(\theta^{*}-\phi(t^{*})) \leq \sqrt{1-\varepsilon_{\kappa}^2}-\varepsilon.
\end{align*}
Then, there exists a positive constant $d$ satisfying
\[ d < D(\varepsilon, \kappa) \quad \mbox{and} \quad \cos (\theta(t^{*}-d)-\phi(t^{*}-d))\leq -\sqrt{1-\varepsilon_{\kappa}^2}+\varepsilon.
\]
where $D(\varepsilon, \kappa)$ is a positive constant defined in \eqref{FF}, and $(\theta(t), \omega(t))$ is the characteristics passing through $(\theta^{*},\omega^{*})$ at $t^{*}$.
\end{corollary}
\begin{proof}
Let $p$ be the barrier through $\sqrt{1-\varepsilon_{\kappa}^2}-\varepsilon$ at $t^{*}.$ By Lemma \ref{LG1}, we have
\[
 \cos (\theta(t)-\phi(t))\leq p(t), \qquad t \in [T_{\kappa},t^{*}],
\]
and the desired result follows by \eqref{d4}.
\end{proof}

\vspace{1cm}

\section*{Acknowledgement}
The work of S.-Y. Ha is partially supported by a National Research Foundation of Korea Grant (2014R1A2A2A05002096) 
funded by the Korea government, and  the work of J. Park was supported by NRF(National Research Foundation of Korea) Grant funded by the Korean Government(NRF-2014-Fostering Core Leaders of the Future Basic Science Program)
The work of Y.-H.Kim is partially supported by Natural Sciences and Engineering Research Council of Canada Discovery Grants 371642-09 and 2014-05448 as well as the  Alfred P. Sloan Research Fellowship 2012--2016.
Part of this research has been done while the authors were participating in the fall Semester 2015 in Analysis at \'Ecole Normale Sup\'erior de Lyon (ENS-Lyon), France. We are grateful for the hospitality of ENS-Lyon and Prof. Albert Fathi.

\end{document}